\documentclass[english]{article}

\usepackage{amsthm,amsmath, amssymb}
\usepackage{url} 
\usepackage[authoryear]{natbib}

\usepackage[colorlinks=true,linkcolor=blue,citecolor=blue,urlcolor=blue,breaklinks]{hyperref}
\usepackage{babel}
\usepackage{caption}
\usepackage{subcaption}
\usepackage{graphicx} 
\usepackage{color}
\usepackage{framed}
\usepackage{lscape}
\usepackage{rotating}
\usepackage{algorithm}
\usepackage[noend]{algpseudocode}
\usepackage{stefan_tex}

\newtheorem*{assumption*}{\assumptionnumber}
\providecommand{\assumptionnumber}{}
\makeatletter
\newenvironment{assumption}[1]
 {%
  \renewcommand{\assumptionnumber}{Assumption #1}%
  \begin{assumption*}%
  \protected@edef\@currentlabel{#1}%
 }
 {%
  \end{assumption*}
 }

\makeatletter
\def\@footnotecolor{red}
\define@key{Hyp}{footnotecolor}{%
 \HyColor@HyperrefColor{#1}\@footnotecolor%
}
\def\@footnotemark{%
    \leavevmode
    \ifhmode\edef\@x@sf{\the\spacefactor}\nobreak\fi
    \stepcounter{Hfootnote}%
    \global\let\Hy@saved@currentHref\@currentHref
    \hyper@makecurrent{Hfootnote}%
    \global\let\Hy@footnote@currentHref\@currentHref
    \global\let\@currentHref\Hy@saved@currentHref
    \hyper@linkstart{footnote}{\Hy@footnote@currentHref}%
    \@makefnmark
    \hyper@linkend
    \ifhmode\spacefactor\@x@sf\fi
    \relax
  }%
\makeatother
\hypersetup{footnotecolor=blue}

\newtheorem{theorem}{Theorem}[section]

\newtheorem{lemma}[theorem]{Lemma}

\newtheorem{corollary}[theorem]{Corollary}

\newtheorem{alemma}{Lemma}[section]
\newtheorem*{theorem*}{Theorem}

\theoremstyle{definition}
\newtheorem*{hypo*}{Hypothesis}

\makeatletter
\def\BState{\State\hskip-\ALG@thistlm}
\makeatother

\newcommand{\BIGFIG}{0.7}

\makeatletter
\def\blfootnote{\gdef\@thefnmark{}\@footnotetext}
\makeatother

\begin{document}

\title{High-Dimensional Asymptotics of Prediction: \\
Ridge Regression and Classification
}
\author{Edgar Dobriban \and Stefan Wager}
\date{Stanford University\blfootnote{{\it E-mail:} \texttt{\{dobriban, swager\}@stanford.edu}}}

\maketitle

\begin{abstract}
We provide a unified analysis of the predictive risk of ridge regression and regularized discriminant analysis in a dense random effects model.
We work in a high-dimensional asymptotic regime where $p, n \to \infty$ and $p/n \to \gamma \in (0, \, \infty)$, and allow for arbitrary covariance among the features. 
For both methods, we provide an explicit and efficiently computable expression for the limiting predictive risk, which depends only on the spectrum of the feature-covariance matrix, the signal strength, and the aspect ratio $\gamma$.
Especially in the case of regularized discriminant analysis, we find that predictive accuracy has a nuanced dependence on the eigenvalue distribution of the covariance matrix, suggesting that analyses based on the operator norm of the covariance matrix may not be sharp.
Our results also uncover several qualitative insights about both methods: for example, with ridge regression, there is an exact inverse relation between the limiting predictive risk and the limiting estimation risk given a fixed signal strength.
Our analysis builds on recent advances in random matrix theory.
\end{abstract}

\section{Introduction}
\label{sec:intro}

Suppose a statistician observes $n$ training examples $\p{x_i, \, y_i} \in \RR^p \times \yy$ drawn independently from an unknown distribution $\dd$, and wants to find a rule for predicting $y$ on future unlabeled draws $x$ from $\dd$. In other words, the statistician seeks a function \smash{$h: \RR^p \rightarrow \yy$}, $h(x) = g(c^\top x)$ for which \smash{$\EE[\dd]{\ell\p{y, \, h\p{x}}}$} is small, where $\ell\p{\cdot, \, \cdot}$ is a loss function; in regression $\yy = \RR$ and $\ell$ is the squared error loss, while in classification $\yy = \{0, \, 1\}$ and $\ell$ is the 0--1 loss. Such prediction problems lie at the heart over several scientific and industrial endeavors in fields ranging from genetics \citep{wray2007prediction} and computer vision \citep{russakovsky2014imagenet} to Medicare resource allocation \citep{kleinberg2015prediction}.

There are various enabling hypotheses that allow for successful prediction in high dimensions. These encode domain-specific knowledge and guide model fitting.
Popular options include the ``sparsity hypothesis'', i.e., that there is a good predictive rule depending only on $w \cdot x$ for some sparse weight vector $w$ \citep{candes2007dantzig,hastie2015statistical}, the ``manifold hypothesis'' positing that the $x_i$ have useful low-dimensional geometric structure \citep{rifai2011manifold,simard2000transformation}, and several variants of an ``independence hypothesis'' that rely on independence assumptions for the feature distribution \citep{bickel2004some,ng2001discriminative}.
The choice of enabling hypothesis is important from a practical perspective, as it helps choose which predictive method to use, e.g., the lasso with sparsity, neighborhood-based methods under the manifold hypothesis, or naive Bayes given independent features.

There are several applications, however, where the above enabling hypotheses are not known to apply,
and where practitioners have achieved accurate high-dimensional prediction using dense---i.e., non-sparse---ridge-regularized linear methods trained on highly correlated features.
One striking example is the case of document classification with dictionary-based features of the form ``how many times does the $j$-th word in the dictionary appear in the current document.''
Even though $p \gg n$, dense ridge-regularized methods reliably work well across a wide range of problem settings \citep{sutton2006introduction,toutanova2003feature}, and sometimes even achieve state-of-the-art performance on important engineering tasks \citep{wang2012baselines}. As another example, in a recent bioinformatics test of prediction algorithms \citep{bernau2015cross-study}, ridge regression---and a method that was previously proposed by those same authors---performed best, better than lasso regression and boosting.

The goal of this paper is to gain better understanding of when dense, ridge-regularized linear prediction methods can be expected to work well.
We focus on a random-effects hypothesis: we assume that the effect size of each feature is drawn independently at random. This can be viewed as an average-case analysis over dense parameters.
Our hypothesis is of course very strong; however, it yields a qualitatively different theory for high-dimensional prediction than popular approaches, and thus may motivate future conceptual developments.

\begin{hypo*}[random effects]
Each predictor has a small, independent random effect on the outcome.
\end{hypo*}

This hypothesis is fruitful both conceptually and methodologically. Using random matrix theoretic techniques \citep[see, e.g.,][]{bai2010spectral}, we derive closed-form expressions for the limiting predictive risk of idge-regularized regression and discriminant analysis, allowing for the features $x$ to have a general covariance structure $\Sigma$.
The resulting formulas are pleasingly simple and depend on $\Sigma$ through the Stieltjes transform of the limiting empirical spectral distribution.
More prosaically, $\Sigma$ only enters into our formulas through the almost-sure limits of \smash{$p^{-1} \, \tr((\hSigma + \lambda I_{p \times p})^{-1})$} and \smash{$p^{-1} \tr((\hSigma + \lambda I_{p \times p})^{-2})$}, where \smash{$\hSigma$} is the sample covariance and $\lambda > 0$ the ridge-regularization parameter. Notably, the same mathematical tools can describe the two problems.

From a practical perspective, we identify several high-dimensional regimes where mildly regularized discriminant analysis performs strikingly well.
Thus, it appears that the random-effects hypothesis can at least qualitatively reproduce the empirical successes of \citet{bernau2015cross-study}, \citet{sutton2006introduction}, \citet{toutanova2003feature}, \citet{wang2012baselines}, and others.
We hope that further work motivated by generalizations of the random-effects hypothesis could yield a new theoretical underpinning for dense high-dimensional prediction.

\subsection{Overview of Results}

We begin with an informal overview of our results; in Section \ref{sec:rmt}, we switch to a formal and fully rigorous presentation. 
In this paper we analyze the predictive risk of ridge-regularized regression and classification when $n, \, p \rightarrow \infty$ jointly. We work in a high-dimensional asymptotic regime where $p/n$ converges to a limiting aspect ratio $p/n \rightarrow \gamma > 0$. The spectral distribution---i.e., the cumulative distribution function of the eigenvalues---of the feature covariance matrix $\Sigma$ converges weakly to a limiting spectral measure supported on $[0, \, \infty)$. This allows $\Sigma$ to be general, and we will see several examples later. 
In random matrix theory, this framework goes back to \citet{marchenko1967distribution}; see, e.g.,  \citet{bai2010spectral}. It has been used in statistics and wireless communications by, among others, \citet{couillet2011random}, \citet{serdobolskii2007multiparametric}, \citet{tulino2004random}, and \citet{yao2015large}. 

In this paper, we present results for ridge regression \citep{hoerl1970ridge} and regularized discriminant analysis (RDA) \citep{friedman1989regularized,serdobolskii1983minimum}.
 Our first result derives the asymptotic predictive risk of ridge regression. We observe a sample from the linear model $Y = Xw + \varepsilon$, where each row of $X$ is random with $\Cov{x} = \Sigma$ and where $\varepsilon$ is independent centered noise with coordinate-wise variance 1.
Given a new test example $x$, ridge regression then predicts $\hy = \hw_\lambda \cdot x$ for $\hw_\lambda = (X^\top X + n\lambda I_{p \times p})^{-1} X^\top Y$; the tuning parameter $\lambda > 0$ governs the strength of the regularization. When the regression coefficient $w$ is normally distributed with identity covariance, ridge regression is a Bayes estimator, thus a random effects hypothesis is natural. We consider a more general random-effects hypothesis where $\smash{\EE{w} = 0}$, and   $\smash{\Var{w}}$ =  $\smash{p^{-1} \alpha^2 I_{p \times p}}$. Let \smash{$\alpha^2 = \EE{\lVert w\rVert_2^2}$} be the expected signal strength, and let \smash{$\hSigma$} be the sample covariance matrix of the features.
Then, under assumptions detailed in Section \ref{sec:ridge}, we show the following result:

\begin{theorem*}[informal statement]
The predictive risk of ridge regression, i.e.,
\begin{equation*}
\Err\p{\hw_{\lambda}} := \EE[x, \, y \sim \dd]{\p{y - \hw_{\lambda} \cdot x}^2}
\end{equation*}
has an almost-sure limit under high-dimensional asymptotics.
This limit only depends on the signal strength $\alpha^2$, the aspect ratio $\gamma$, the regularization parameter $\lambda$, and 
the Stieltjes transform of the limiting eigenvalue distribution of $\hSigma$.
For the optimal tuning parameter, $\lambda^* = \gamma/\alpha^2$
\begin{equation}
\label{eq:ridge_example_2}
\Err\p{\hw_{\lambda^*}}\rightarrow_{a.s.}\frac{1}{\lambda^*v(-\lambda^*)},
\end{equation}
where $v$ is the companion Stieltjes transform of the limiting eigenvalue distribution of $\hSigma$, defined in Section \ref{sec:rmt}.
\end{theorem*}

The required functionals of the limiting empirical eigenvalue distribution can be written in terms of almost-sure limits of simple quantities. 
For example, the result \eqref{eq:ridge_example_2} can be written as
\begin{align*}
\label{eq:ridge_example}
 \Err\p{\hw_{\lambda^*}} - \p{\frac{\gamma^2}{\alpha^2} \, \frac{1}{p} \tr\left[\p{\hSigma + \frac{\gamma}{\alpha^{2}} I_{p \times p}}^{-1}\right] + 1 - \gamma}^{-1} \rightarrow_{a.s.} 0.
\end{align*}
For general $\lambda$, the limiting error rate depends on the almost sure limits of both \smash{$p^{-1} \, \tr((\hSigma + \lambda I_{p \times p})^{-1})$} and \smash{$p^{-1} \tr((\hSigma + \lambda I_{p \times p})^{-2})$}.

Thanks to the simple form of \eqref{eq:ridge_example_2}, we can use our results to gain qualitative insights about the behavior of ridge regression.
We show that, when the signal-to-noise ratio is high, i.e., $\alpha \gg 1$, the accuracy of ridge regression has a sharp phase transition at $\gamma = 1$ regardless of $\Sigma$, essentially validating a conjecture of \citet{liang2010interaction} on the ``regimes of learning'' problem.
We also find that ridge regression obeys an \emph{inaccuracy principle}, whereby there are no correlation structures $\Sigma$ for which prediction and estimation of $w^*$ are both easy. For $\gamma=1$ this simplifies to
\begin{equation*}
\label{eq:inaccuracy}
\Err\p{\hw_{\lambda}} \cdot \EE{\Norm{\hw_\lambda - w^*}_2^2} \geq {\alpha^2}; 
\end{equation*}
this bound is tight for optimally-tuned ridge regression. We refer to Section \ref{sec:pred-estimation} for the general relation. We find the simplicity of the inverse relation remarkable. 

In the second part of the paper, we study regularized discriminant analysis in the two-class Gaussian problem
\begin{equation}
\label{eq:setup}
y \sim \cb{\pm 1} \text{ with } \PP{y = 1} = 1/2, \text{ and } x \sim \nn\p{\mu_y, \, \Sigma},
\end{equation}
where $\smash{\mu_{\pm 1}}$ and $\Sigma$ are unknown. While our results cover unequal class probabilities, for simplicity here we present the case when $\PP{y = 1} = 1/2$. We instantiate the random-effects hypothesis by assuming that the pairs $\smash{(\mu_{- 1, \, i}, \, \mu_{+ 1, \, i})}$ are independently and identically distributed for $i = 1, \, ..., \, p$, and write \smash{$\alpha^2 = \EE{\lVert \delta \rVert_2^2}$} with \smash{$\delta = (\mu_{+1} - \mu_{-1})/2$}.
We again work in a high-dimensional regime where $p/n \rightarrow \gamma > 0$ and the within-class covariance $\Sigma$ has a limiting spectral distribution.
Given this notation, the Bayes-optimal decision boundary is orthogonal to $w^* = \Sigma^{-1}\delta$; meanwhile, the regularized discriminant classifier predicts $\hy = \sign\p{\hw_\lambda \cdot x}$, where the form of the weight-vector $\hw_\lambda$ is given in Section \ref{sec:twoclass}.

\begin{theorem*}[informal statement]
In high dimensions, and in the metric induced by $\Sigma$, the angle between $w^*$ and $\hw_\lambda$, i.e.,
\begin{equation*}
\cos_\Sigma\p{\hw_\lambda, \, w^*} = \frac{\angles{\hw_\lambda, \, w^*}_\Sigma}{\sqrt{\angles{\hw_\lambda, \, \hw_\lambda}_\Sigma \, \angles{w^*, \, w^*}_\Sigma}}, \ \ \angles{u, \, v}_\Sigma = u^\top \Sigma v,
\end{equation*}
has an almost-sure limit.
The classification error of regularized discriminant analysis converges to an almost-sure limit that  depends only on this limiting angle, as well as the limiting Bayes error.
The limiting risk can be expressed in terms of $\alpha$, $\gamma$, $\lambda$,  as well the Stieltjes transform of the limit eigenvalue distribution of the empirical within-class covariance matrix.
\end{theorem*} 

We can again use our result to derive qualitative insights about the behavior of RDA. We find that the limiting angle between $\hw_\lambda$ and $w^*$ converges to a non-trivial quantity as $\alpha^2 \to \infty$, implying that our analysis is helpful in understanding the asymptotics of RDA even in a very high signal-to-noise regime. Finally, by studying the limits $\lambda \rightarrow 0 $ and $\lambda \rightarrow \infty$, we can recover known high-dimensional asymptotic results about Fisher's linear discriminant analysis and naive Bayes methods going back to \citet{bickel2004some}, \citet{raudys1967determining}, \citet{saranadasa1993asymptotic}, and even early work by Kolmogorov.

Mathematically, our results build on recent advances in random matrix theory.
The main difficulty here is finding explicit limits of certain trace functionals involving both the sample and the population covariance matrix.
For instance, the Stieltjes transform $m$ of the empirical spectral distribution satisfies $m(-\lambda) = \lim_{p\to\infty}\smash{p^{-1} \tr((\hSigma + \lambda I_{p \times p})^{-1})}$.
However, standard random matrix theory does not provide simple expressions for the limits of functionals like $\smash{p^{-1} \tr(\Sigma(\hSigma + \lambda I_{p \times p})^{-1})}$ or $\smash{p^{-1} \tr([\Sigma(\hSigma + \lambda I_{p \times p})]^{-2})}$ that involve both $\Sigma$ and $\hSigma$.
For this we leverage and build on recent results, including the work of \citet{chen2011regularized}, \citet{hachem2007deterministic}, and \citet{ledoit2011eigenvectors}.
Our contributions include some new explicit formulas, for which we refer to the proofs.
These formulas may prove useful for the analysis of other statistical methods under high-dimensional asymptotics, such as principal component regression and kernel regression.

\subsection{A First Example}

A key contribution of our theory is a precise understanding of the effect of correlations between the features on regularized discriminant analysis. Correlated features have a non-trivial effect, and cannot be summarized using standard notions such as the condition number of $\Sigma$ or the classification margin. The full eigenvalue spectrum of $\Sigma$ matters. This is in contrast with popular analyses of high-dimensional classification methods, in which the bounds often depend on the operator norm $\|\Sigma\|_{op}$ (see for instance the review \citet{fan2011high}), thus suggesting that existing analyses of many classification methods are not sharp.

Consider the following examples: First, $\Sigma$ has eigenvalues corresponding to evenly-spaced quantiles of the standard {\tt Exponential} distribution; Second, $\Sigma$ has a depth-$d$ {\tt BinaryTree} covariance structure that has been used in population genetics to model the correlations between populations whose evolutionary history is described by a balanced binary tree \citep{pickrell2012inference}.
In both cases, we set the class means $\mu_y$ such as to keep the Bayes error constant across experiments.
Figure \ref{fig:first} plots our formulas for the asymptotic error rate along with empirical realizations of the classification error.

Both covariance structures are far from the identity, and have similar condition numbers.
However, the {\tt Exponential} problem is vastly more difficult for RDA than the {\tt BinaryTree} problem.
This example shows that classical notions like the classification margin or the condition number of $\Sigma$ cannot satisfactorily explain the high-dimensional predictive performance of RDA;
meanwhile, our asymptotic formulas are accurate even in moderate sample sizes.
Our computational results are reproducible and open-source software to do so is available from \url{https://github.com/dobriban/high-dim-risk-experiments/}.

\newcommand{\FW}{0.45\textwidth}
\newcommand{\TRA}{0}
\newcommand{\TRB}{0}
\newcommand{\TRC}{0}
\newcommand{\TRD}{0}
\newcommand{\PBW}{0.3}

\begin{figure}[p]
\centering
\begin{tabular}{ccc}
& {\large \tt BinaryTree} & {\large \tt Exponential} \\
{\begin{sideways}\parbox{\PBW\columnwidth}{\centering $\gamma = 0.5$}\end{sideways}} &
\includegraphics[width=\FW, trim = \TRA mm \TRB mm \TRC mm \TRD mm, clip = TRUE]{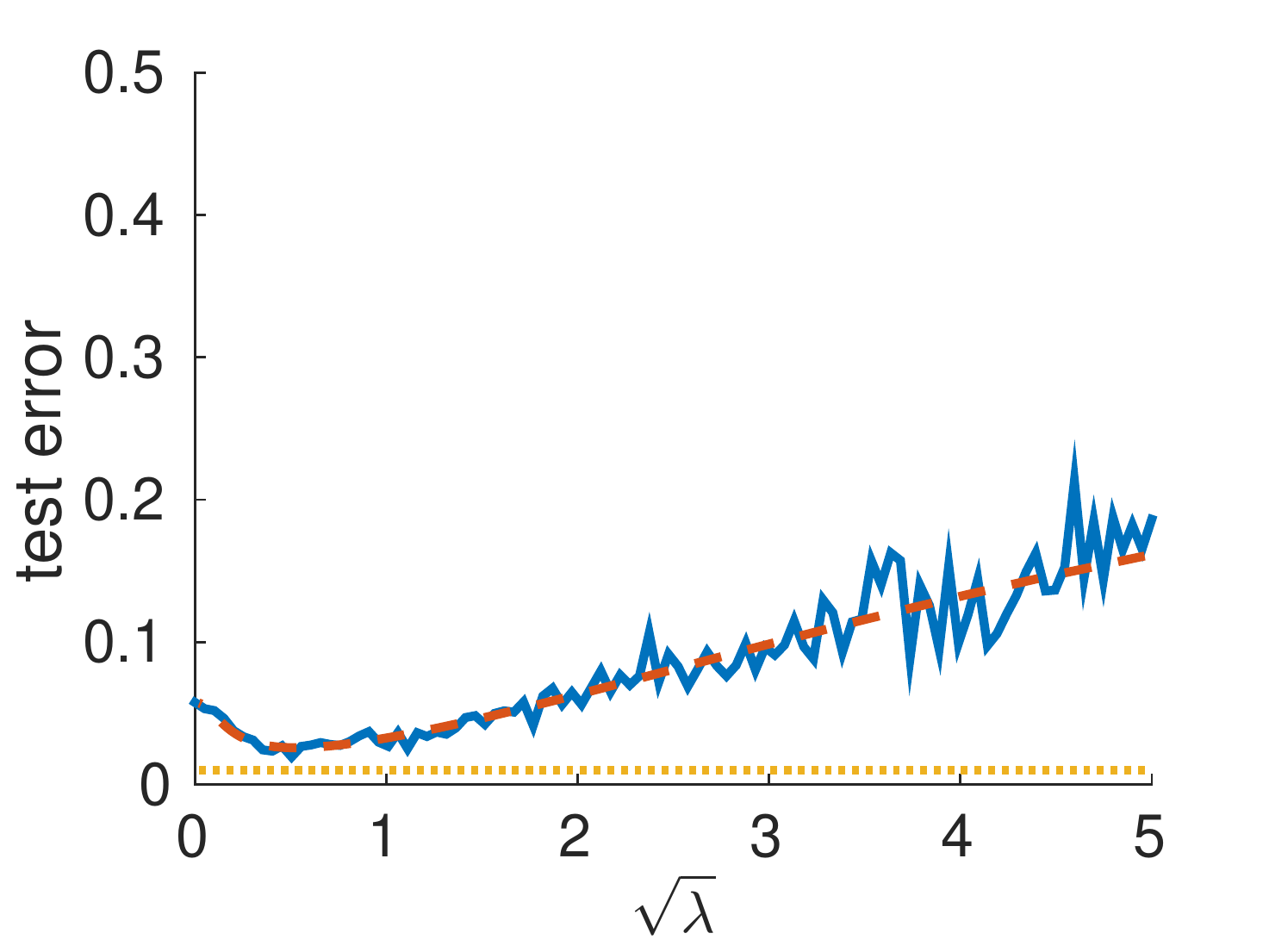} &
\includegraphics[width=\FW, trim = \TRA mm \TRB mm \TRC mm \TRD mm, clip = TRUE]{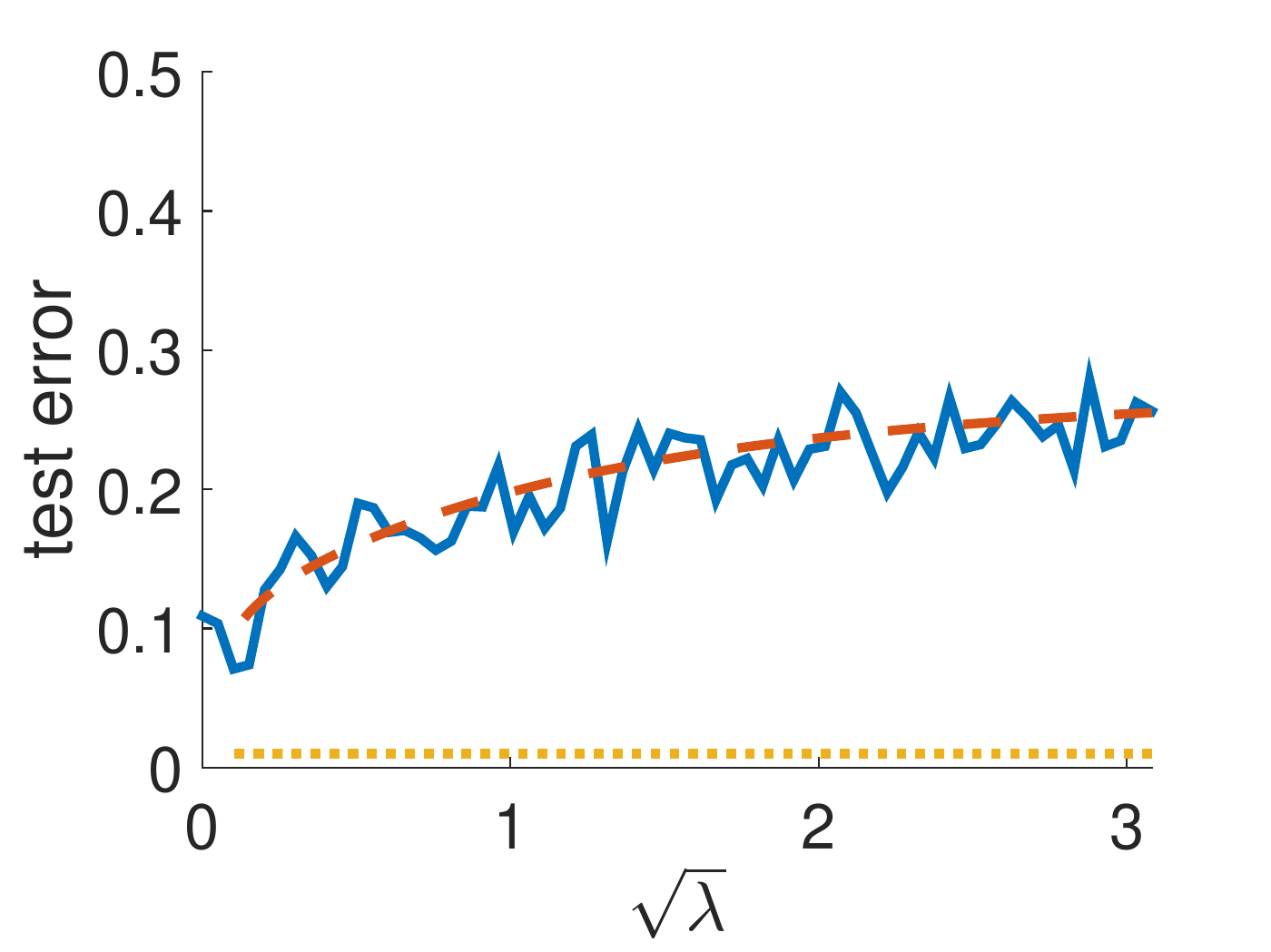} \\
{\begin{sideways}\parbox{\PBW\columnwidth}{\centering $\gamma = 1$}\end{sideways}} &
\includegraphics[width=\FW, trim = \TRA mm \TRB mm \TRC mm \TRD mm, clip = TRUE]{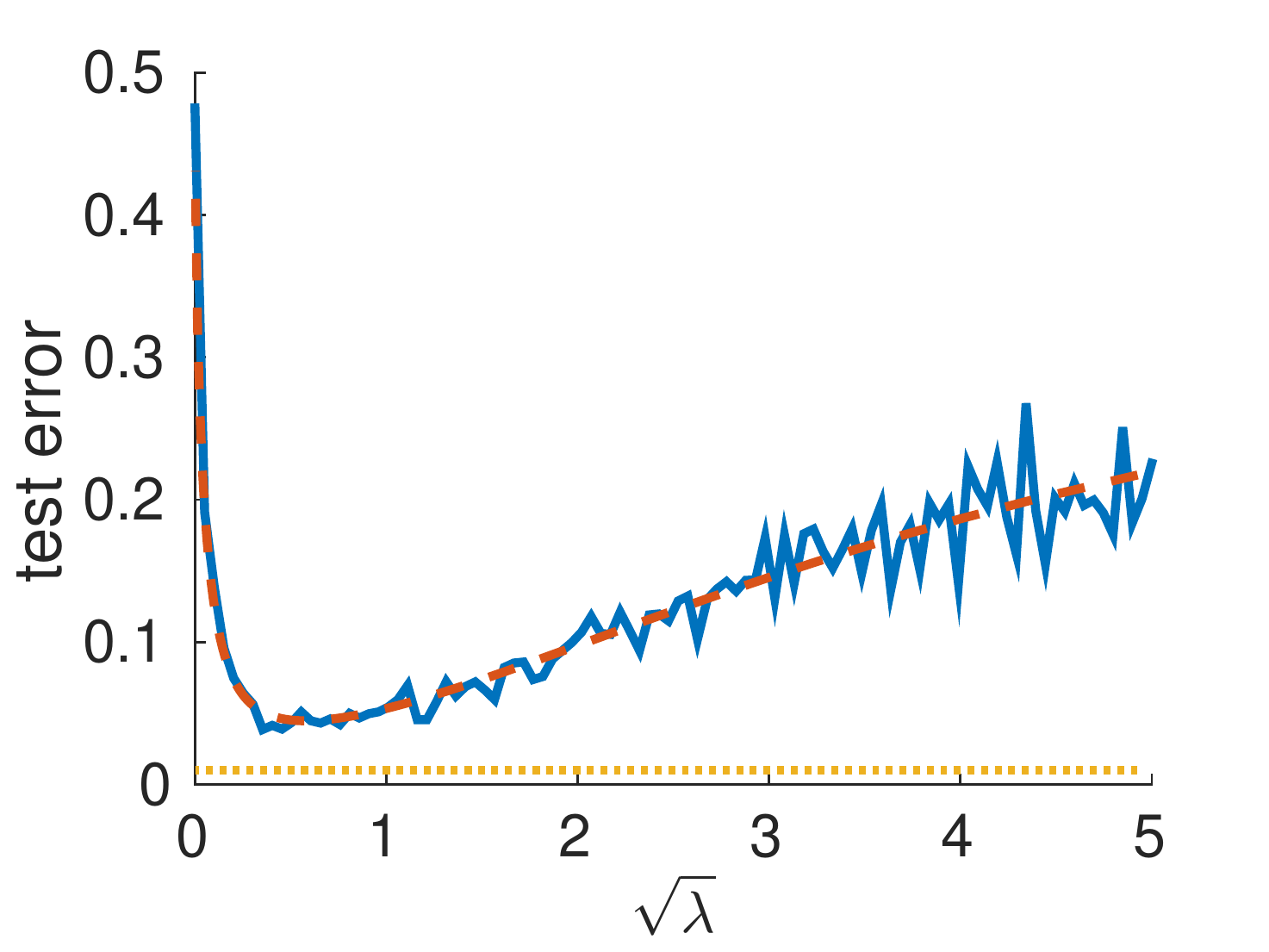} &
\includegraphics[width=\FW, trim = \TRA mm \TRB mm \TRC mm \TRD mm, clip = TRUE]{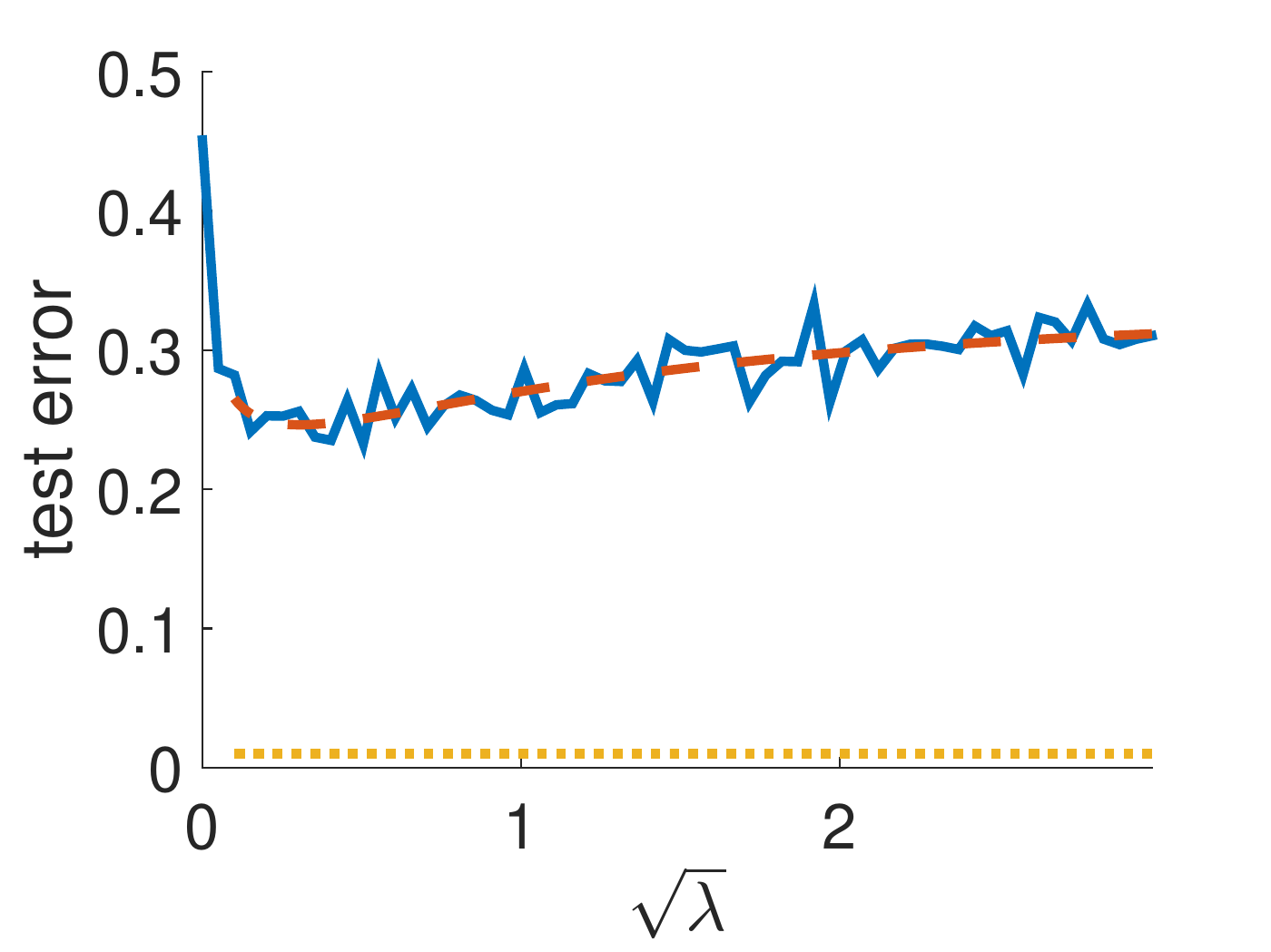}  \\
{\begin{sideways}\parbox{\PBW\columnwidth}{\centering $\gamma = 2$}\end{sideways}} &
\includegraphics[width=\FW, trim = \TRA mm \TRB mm \TRC mm \TRD mm, clip = TRUE]{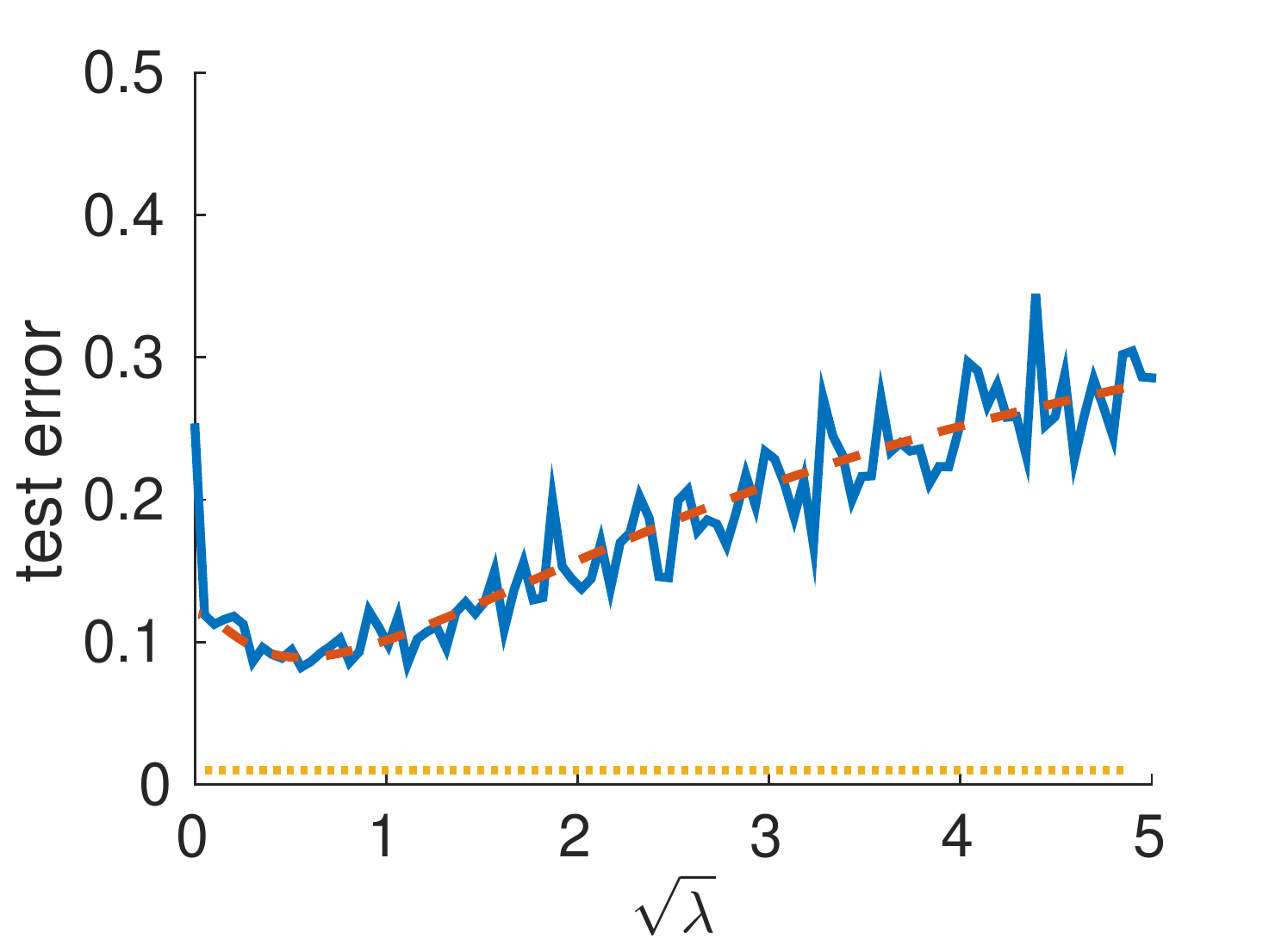} &
\includegraphics[width=\FW, trim = \TRA mm \TRB mm \TRC mm \TRD mm, clip = TRUE]{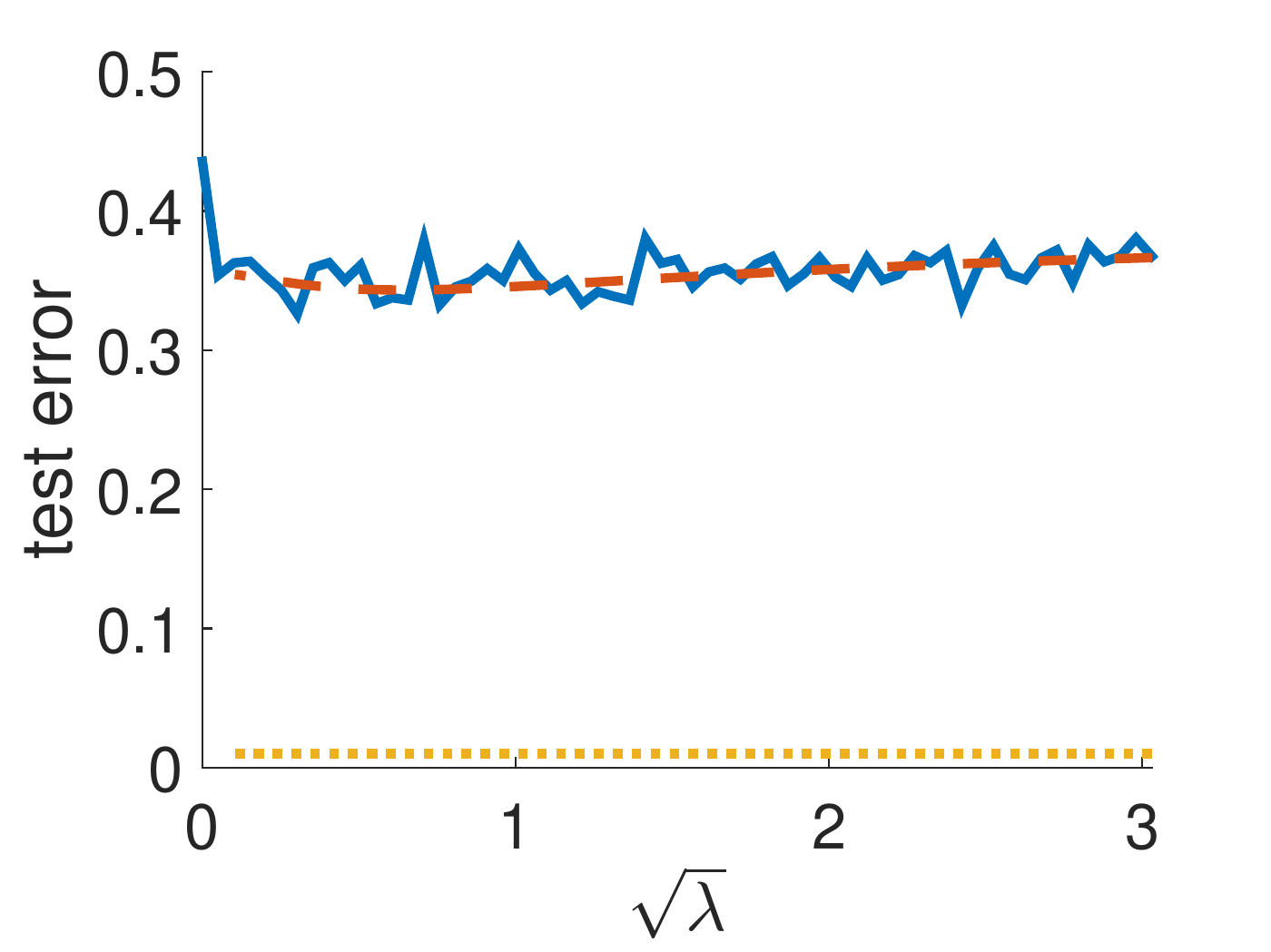} \\
\end{tabular}
\caption{Classification error of RDA in the {\tt BinaryTree} and {\tt Exponential} models. The theoretical formula (red, dashed) is overlaid with the results from simulations (blue, solid); we also display the oracle error (yellow, dotted).
The class means are drawn from $\smash{\mu_{\pm 1} \sim \nn\p{0, \, \alpha^2 \, p^{-1} I_{p \times p}}}$, where $\alpha$ is calibrated such that the oracle classifier always has an error rate of $1\%$.  For {\tt BinaryTree}, we train on $n = \gamma^{-1} p$ samples, where $p=1024$; for {\tt Exponential}, we use $n = 500$ samples. We test the trained model on 10,000 new samples, and report the average classification error. Our asymptotically-motivated theoretical formulas appear to be accurate here, even though we only have a moderate problem size. The parameter $\lambda$, defined in Section \ref{sec:twoclass}, quantifies the strength of the regularization.}
\label{fig:first}
\end{figure}

\subsection{Related Work}

Random matrix theoretic approaches  have been used to study regression and classification in high-dimensional statistics \citep{serdobolskii2007multiparametric,yao2015large}, as well as wireless communications \citep{couillet2011random,tulino2004random}. Various regression and M-estimation problems have been studied in high dimensions using approximate message passing \citep{bayati2012lasso,donoho2015variance} as well as methods inspired by random matrix theory \citep{bean2013optimal}. We also note a remarkably early random matrix theoretic analysis of regularized discriminant analysis by \citet{serdobolskii1983minimum}.
In the wireless communications literature, the estimation properties of ridge regression are well understood; however its prediction error has not been studied.

\citet{karoui2011geometric} study the geometric sensitivity of random matrix results, and discuss the consequences to ridge regression and regularized discriminant analysis, under weak theoretical assumptions.
In contrast, we make stronger assumptions that enable explicit formulas for the limiting risk of both methods, and allow us to uncover several qualitative phenomena.
Our use of \citet{ledoit2011eigenvectors}'s results simplifies the proof.

We review the literature focusing on ridge regression or RDA specifically in Sections \ref{sec:ridge_review} and  \ref{sec:rda_review} respectively.
Important references include, among others, \citet{bickel2004some}, \citet{dicker2014ridge}, \citet{karoui2013asymptotic}, \citet{fujikoshi2011multivariate},  \citet{hsu2014random}, \citet{saranadasa1993asymptotic}, and \citet{zollanvari2015generalized}.

\subsection{Basics and Notation}
\label{sec:rmt} 

We begin the formal presentation by reviewing some key concepts and notation from random matrix theory that are used in our high-dimensional asymptotic analysis.
Random matrix theory lets us describe the asymptotics of the eigenvalues of large matrices \citep[see e.g.,][]{bai2010spectral}.
These results are typically stated in terms of the \emph{spectral distribution}, which for a symmetric matrix $A$ is the cumulative distribution function of its eigenvalues: $\smash{F_A(x)} = \smash{ p^{-1} \sum_{i = 1}^p \mathrm{I}(\lambda_i(A) \le x)}$.
In particular, the well-known Marchenko-Pastur theorem, given below, characterizes the spectral distribution of covariance matrices. We will assume the following high-dimensional asymptotic model:

\begin{assumption}{A}\label{assume:A}(high-dimensional asymptotics)
The following conditions hold.
\begin{enumerate}
\item The data $X \in \RR^{n \times p}$ is generated as
$\smash{X = Z \, \Sigma^{1/2}}$ for an $n \times p$ matrix $Z$ with i.i.d. entries satisfying $\smash{\EE{Z_{ij}} = 0}$ and $\smash{\Var{Z_{ij}} = 1}$, and a deterministic $p \times p$ positive semidefinite covariance matrix $\Sigma$. 
\item The sample size $n \to \infty$ while the dimensionality $p \to \infty$ as well, such that the aspect ratio $p/n \to \gamma >0$.  
\item The spectral distribution $\smash{F_\Sigma}$ of $\Sigma$ converges to a limit probability distribution $H$ supported $\smash{[0, \, \infty)}$, called the population spectral distribution (PSD).
\end{enumerate}
\end{assumption}

Families of covariance matrices $\Sigma$ that fit the setting of this theorem include the identity covariance, {\tt BinaryTree}, {\tt Exponential} and the autoregressive {\tt AR-1} model with $\Sigma_{ij} = \rho^{|i-j|}$ \citep[see][for the last one]{grenander1984toeplitz}.

\begin{theorem*}[\citet{marchenko1967distribution}; \citet{silverstein1995strong}]
Under assumption \ref{assume:A}, the spectral distribution $\smash{F_{\hSigma}}$ of the sample covariance matrix $\smash{\hSigma}$ 
also converges weakly, with probability 1,  to a limiting distribution supported on $[0, \, \infty)$.
\end{theorem*}

The limiting distrbution $F$ is called the empirical spectral distribution (ESD), and is determined uniquely by a fixed point equation for its \emph{Stieltjes transform}, which is defined for any distribution $G$ supported on $[0,\infty)$ as
$$m_G(z) = \int_{l=0}^{\infty} \frac{dG(l)}{l-z}, \ \  z \in \mathbb{C}\setminus \mathbb{R}^+ . \ \ 
$$
Given this notation, the Stieltjes transform of the spectral measure of $\smash{\hSigma}$ satisfies
\begin{equation}
\label{eq:mp_lemma}
m_{\hSigma}\p{z} = \frac{1}{p} \tr \p{\p{\hSigma - z \, I_{p \times p}}^{-1}} \text{ converges to } m\p{z}
\end{equation}
both almost surely and in expectation, 
for any $z \in \mathbb{C}\setminus \mathbb{R}^+$; here, we wrote $m\p{z}:= m_F\p{z}$.
We also define the companion Stieltjes transform $v(z)$, which is the Stieltjes transform of the limiting spectral distribution of the matrix $\smash{\underline \hSigma = n^{-1} X X^\top}$. This is related to $m(z)$ by
\begin{equation}
\label{dual.ST}
\gamma\left(m(z)+1/z\right) = v(z)+1/z \ \text{ for all } \ z \in \mathbb{C}\setminus \mathbb{R}^+.
\end{equation}
\sloppy{In addition, we write the derivatives as\footnote{We will denote by $v'(-\lambda)$ the derivative of the Stieltjes transform,  $v'(z)$, evaluated at $z = -\lambda$; and not the derivative of the function $\lambda \to v(-\lambda)$.}
$ \smash{m'(z) = \int_{l=0}^{\infty} {dG(l)}/\p{l-z}^2}$
and $\smash{v'(z) = \gamma(m'(z) - z^{-2}) + z^{-2}}$.}
These derivatives can also be understood in terms of empirical observables, through the relation
\begin{equation*}
\frac{1}{p} \tr\p{\p{\hSigma + \lambda I_{p \times p}}^{-2}} \rightarrow_{a.s.} m'(-\lambda).
\end{equation*}
Finally, our analysis also relies on several more recent formulas for limits of trace functionals involving both $\Sigma$ and $\hSigma$. In particular, we use a formula due to \citet{ledoit2011eigenvectors}, who in the analysis of eigenvectors of sample covariance matrices showed that, under certain moment conditions:\footnote{See the supplement for more details about this result.}
\begin{equation}
\frac{1}{p}\tr\left( \Sigma \left(\hSigma + \lambda I_{p \times p} \right)^{-1}\right) \to_{a.s.} \frac{1}{\gamma}\left(\frac{1}{\lambda \, v(-\lambda)} -1\right) \ \ \text{ as } n, \, p \rightarrow \infty;
\label{eq:ledoit_peche}
\end{equation}

\section{Predictive Risk of Ridge Regression}
\label{sec:ridge}

In the first part of the paper, we study the predictive behavior of ridge regression under large-dimensional asymptotics.  Suppose that we have data drawn from a $p$-dimensional random-design linear model with $n$ independent observations $y_i = x_i \cdot w + \varepsilon_i$. The noise terms $\varepsilon_i$ are independent, centered, with variance one, and are independent of the other random quantities. The $x_i$ are arranged as the rows of the $n \times p$ matrix $X$, and $y_i$ are the entries of the $n \times 1$ vector $Y$. We estimate $w$ by ridge regression: $\smash{\hw_\lambda}$ = $\smash{(X^\top X + \lambda \, n \, I_{p \times p})^{-1} X^\top Y}$, for some $\lambda>0$.  We make the following random weights assumption, where $\smash{\alpha^2 = \mathbb{E}[\Norm{w}_2^2]}$ is the expected signal strength.

\begin{assumption}{B}\label{assume:B} (random regression coefficients)
The true weight vector $w$ is random with $\smash{\EE{w} = 0}$, and   $\smash{\Var{w}}$ =  $\smash{p^{-1} \alpha^2 I_{p \times p}}$.
\end{assumption}

Our result about the predictive risk of ridge regression is stated in terms of the expected predictive risk $\smash{r_\lambda(X) = \EE{(y - \hy_\lambda)^2 \cond X}}$, where $(x, \, y)$ is taken to be an independent test example from the same distribution as the training data, and $\hy_\lambda = \hw_\lambda \cdot x$.

\begin{theorem}
\label{theo:ridge}
Under Assumptions \ref{assume:A} and \ref{assume:B}, suppose moreover that the eigenvalues of $\Sigma$ are uniformly bounded above:\footnote{Below, $C$ will denote an arbitrary fixed constant whose meaning can change from line to line.} $\|\Sigma\|_{op} \le C$, for all $p$. Also assume  $\smash{\EE{Z_{ij}^{12}} < C}$ for all $p$. Then,
\begin{enumerate}
\item Writing $\gamma_p = p/n$ and $\lambda_p^* =\gamma_p \alpha^{-2}$, the finite sample predictive risk $r_{\lambda^*_p}(X)$ converges almost surely
\begin{align}
\label{eq:ridge}
 r_{\lambda_p^*}(X) &=  1 + \frac{\gamma_p}{p} \tr \p{ \Sigma \p{\hSigma + \frac{\gamma_p}{\alpha^2} I_{p \times p} }^{-1}}
\\
\notag
& \rightarrow_{a.s.} R^*(H,\alpha^2,\gamma):= \frac{1}{\lambda^* v(-\lambda^*)},
\end{align}
where $\lambda^* = \gamma \alpha^{-2}$.
\item Moreover, for any $\lambda>0$, the predictive risk converges almost surely to the limiting predictive risk $R_\lambda(H,\alpha^2,\gamma)$, where
\begin{equation*}
\label{eq:ridge_general}
R_\lambda(H,\alpha^2,\gamma)= \frac{1}{\lambda v(-\lambda)} \left\{ 1 + \left(\frac{\lambda\alpha^2}{\gamma}-1\right) \left(1-\frac{\lambda v'(-\lambda)}{v(-\lambda)}\right) \right\}. 
\end{equation*}
The choice $\lambda^*$ minimizes $R_\lambda(H,\alpha^2,\gamma)$.
\end{enumerate}
\begin{proof}[Proof outline]
The proof of part 1 is sufficiently simple to outline here; see the supplement for part 2.
We begin by verifying formula \eqref{eq:ridge}: 
\begin{align*}
r_{\lambda_p^*}(X)  &=  1+ \EE{ \p{x \cdot \p{\hw_{\lambda_p^*} - w}}^2 \cond X} \\
&=  1+ \EE{\p{\hw_{\lambda_p^*} - w}^\top \Sigma \p{\hw_{\lambda_p^*} - w} \cond X} \\
&= 1+  \lambda_p^{* \, 2} \, \alpha^2 \, n \, \tr\p{\Sigma\p{X^\top X + \lambda_p^* \, n \, I_{p \times p}}^{-2}} \\
&\ \ \ \ \ \ \ \  + \tr\p{\Sigma\p{X^\top X + \lambda_p^* \, n \, I_{p \times p}}^{-1} X^\top X \p{X^\top X + \lambda_p^* \, n \, I_{p \times p}}^{-1}} \\
&=  1 + \frac{\gamma_p}{p} \tr \p{ \Sigma \p{\hSigma + \frac{\gamma_p}{\alpha^2} I_{p \times p} }^{-1}}.
\end{align*}
\sloppy{On the last line we used the choice of $\lambda_p^*$. From the results of \citet{ledoit2011eigenvectors}, and from $\gamma_p\to\gamma$, it can be verified that $\smash{p^{-1} \tr ( \Sigma (\hSigma + \gamma_p\alpha^{-2} I_{p \times p} )^{-1})}$ converges almost surely to limit in \eqref{eq:ledoit_peche}, finishing part 1.}
\end{proof}
\end{theorem}

This result fully characterizes the first order behavior of the predictive risk of ridge regression under high-dimensional asymptotics. To verify its finite-sample accuracy, we perform a simulation with the {\tt BinaryTree} and {\tt Exponential} models. We compute the limit risks using the algorithms in the supplement. The results in Figure \ref{fig:second} show that the formulas given in Theorem \ref{theo:ridge} appear to be accurate, even in small sized problems. In Figure \ref{fig:second}, for {\tt BinaryTree} we train on $n = \gamma^{-1} p$ samples, where $p=2^4$; for {\tt Exponential} on $n=20$. We set the signal strength to $\alpha^2=1$ and generated $w$, $X$, and $\varepsilon$  as Gaussian random variables with i.i.d. entries and the desired variance. The results are averaged over 500 simulation runs; we evaluated the empirical prediction error using a test set of size 100.

\begin{figure}[p]
\centering
\begin{tabular}{ccc}
& {\large \tt BinaryTree} & {\large \tt Exponential} \\
{\begin{sideways}\parbox{\PBW\columnwidth}{\centering $\gamma = 0.5$}\end{sideways}} &
\includegraphics[width=\FW, trim = \TRA mm \TRB mm \TRC mm \TRD mm, clip = TRUE]{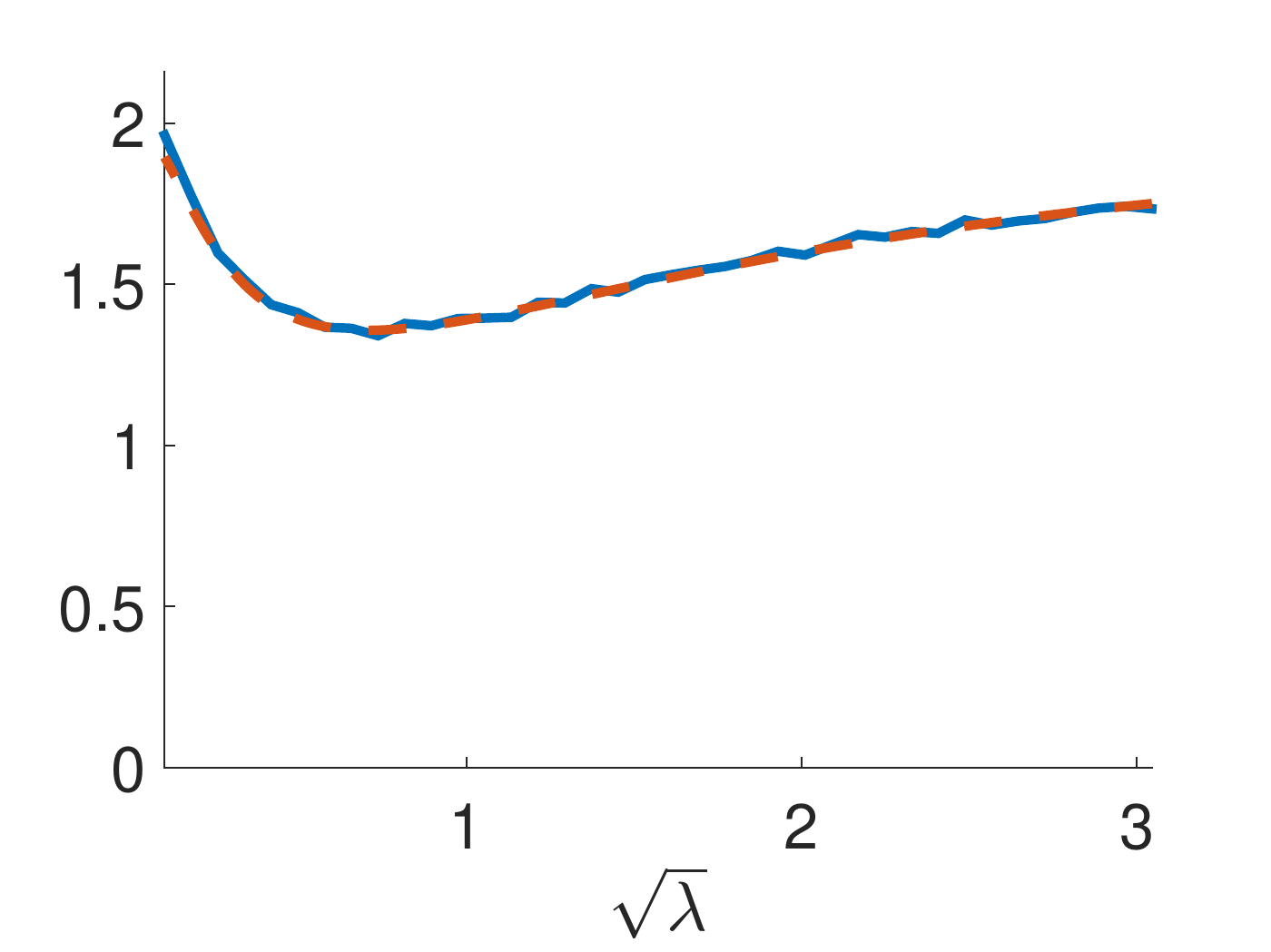} &
\includegraphics[width=\FW, trim = \TRA mm \TRB mm \TRC mm \TRD mm, clip = TRUE]{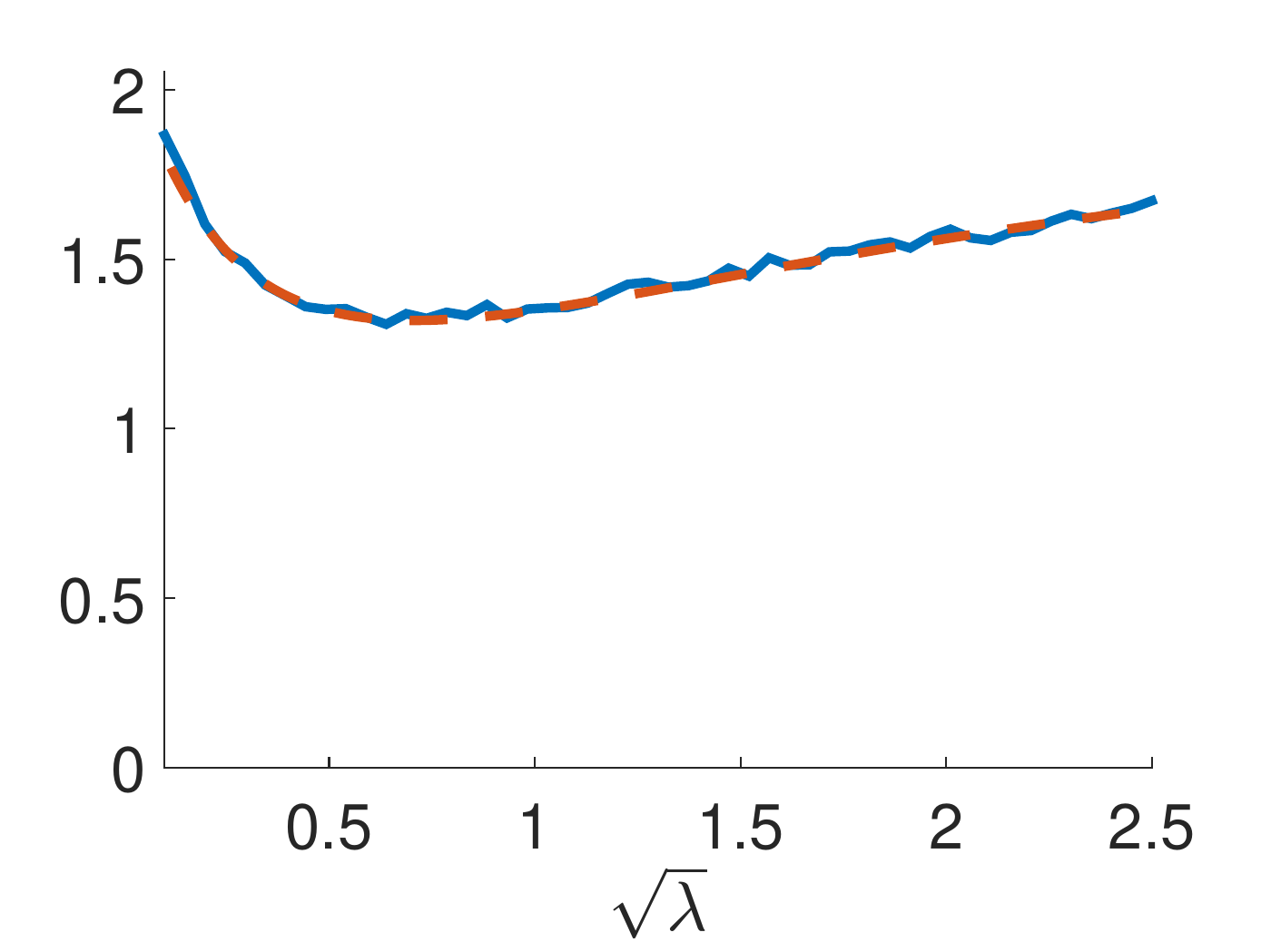} \\
{\begin{sideways}\parbox{\PBW\columnwidth}{\centering $\gamma = 1$}\end{sideways}} &
\includegraphics[width=\FW, trim = \TRA mm \TRB mm \TRC mm \TRD mm, clip = TRUE]{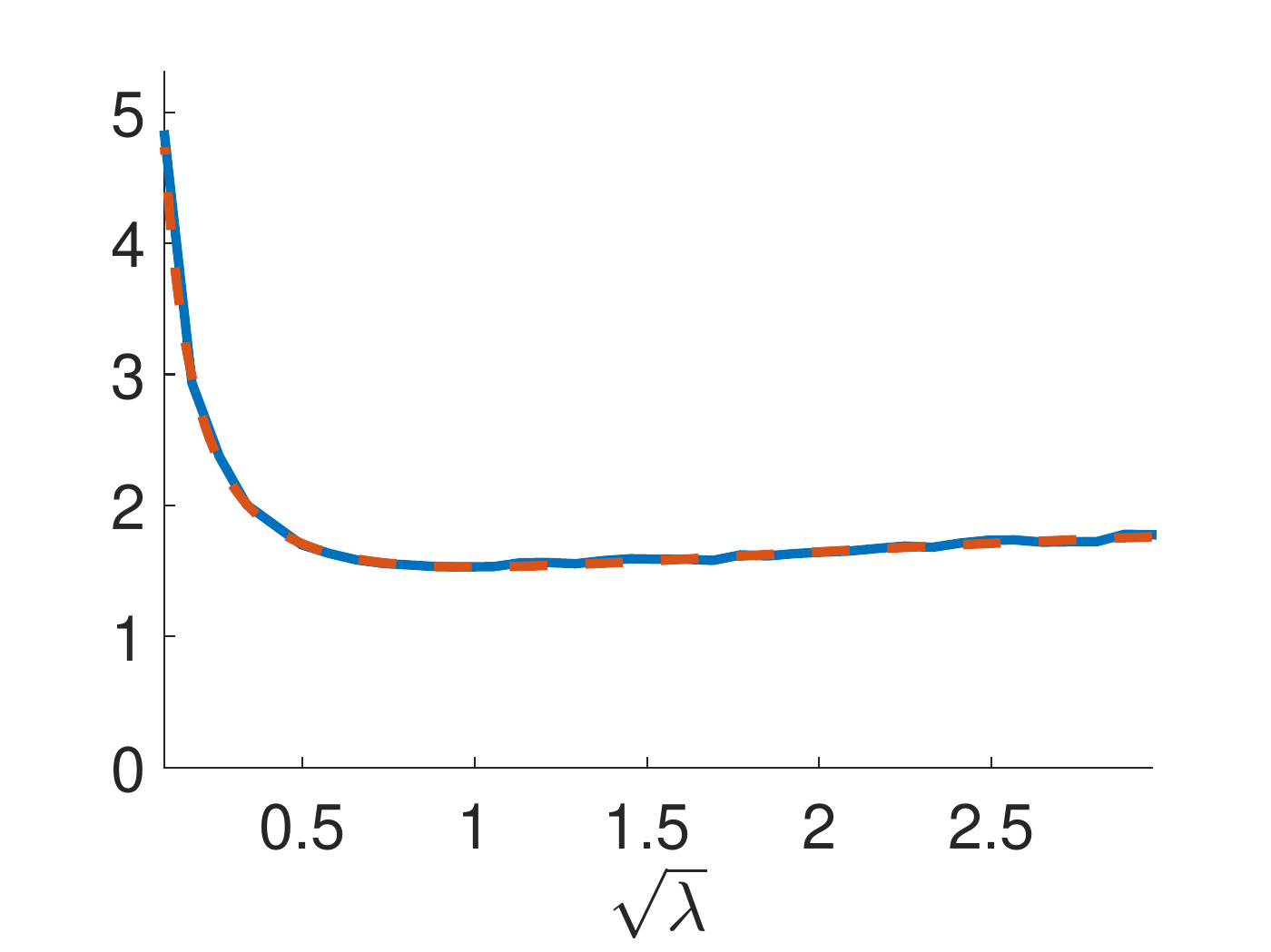} &
\includegraphics[width=\FW, trim = \TRA mm \TRB mm \TRC mm \TRD mm, clip = TRUE]{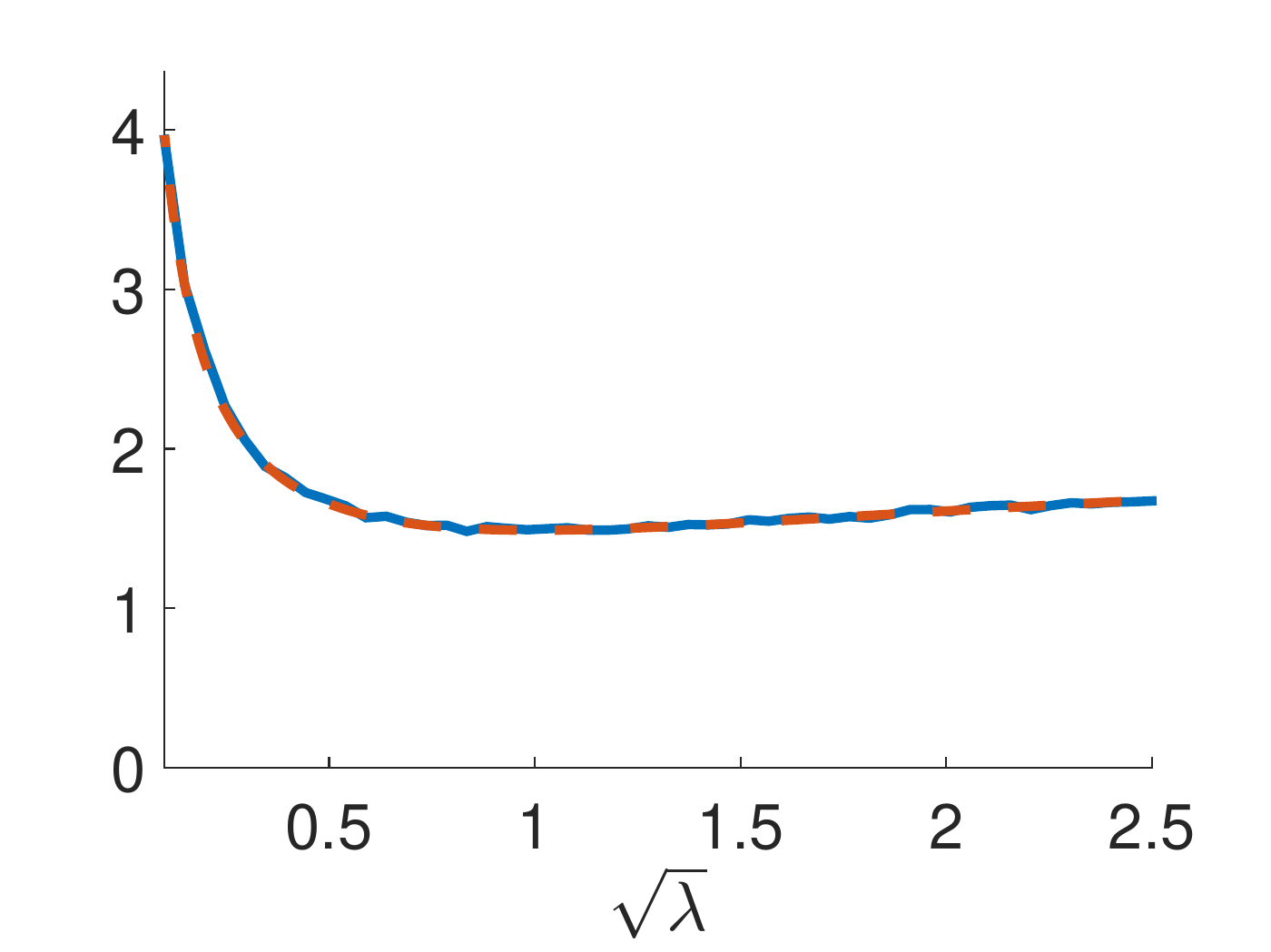}  \\
{\begin{sideways}\parbox{\PBW\columnwidth}{\centering $\gamma = 2$}\end{sideways}} &
\includegraphics[width=\FW, trim = \TRA mm \TRB mm \TRC mm \TRD mm, clip = TRUE]{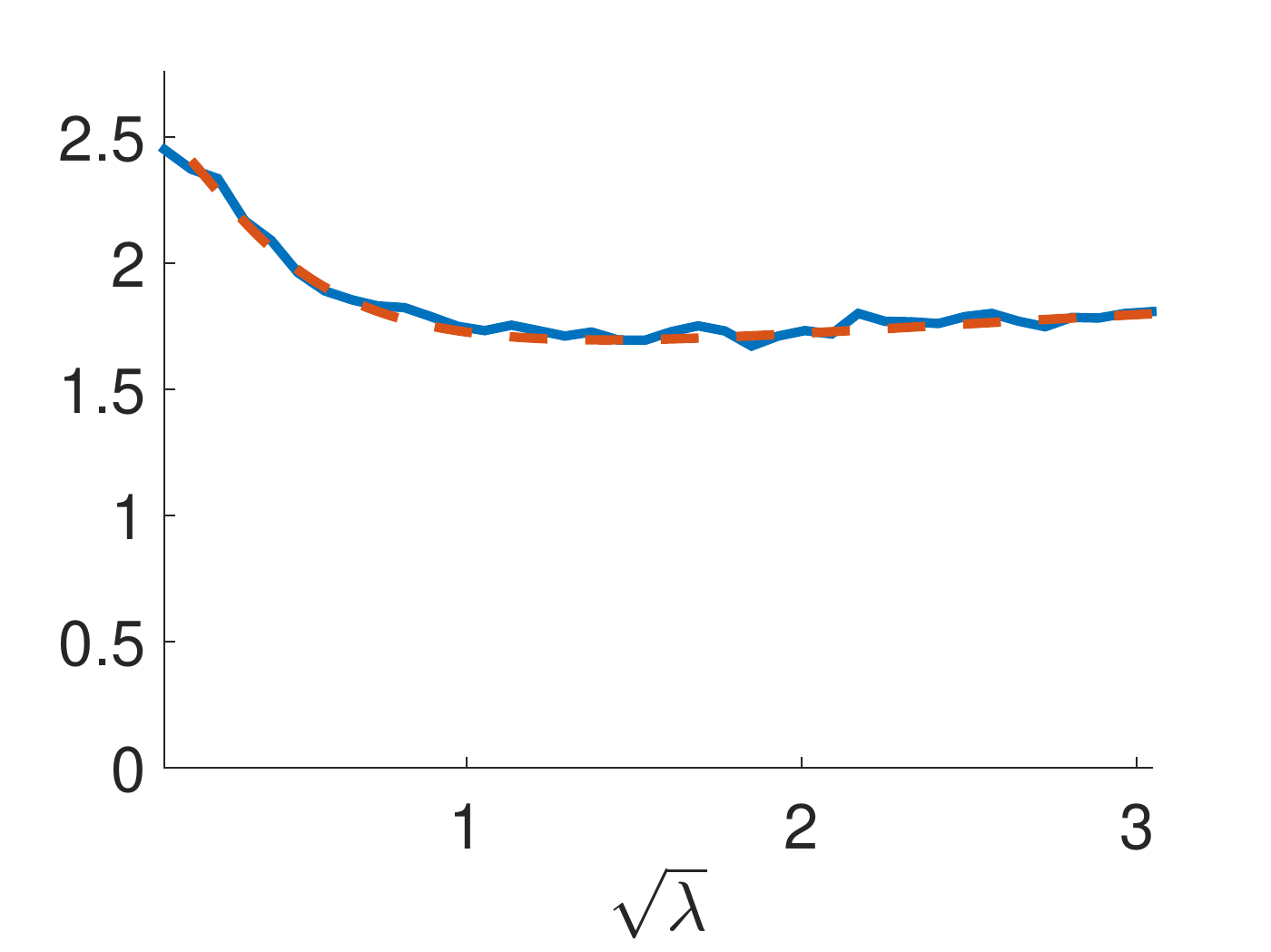} &
\includegraphics[width=\FW, trim = \TRA mm \TRB mm \TRC mm \TRD mm, clip = TRUE]{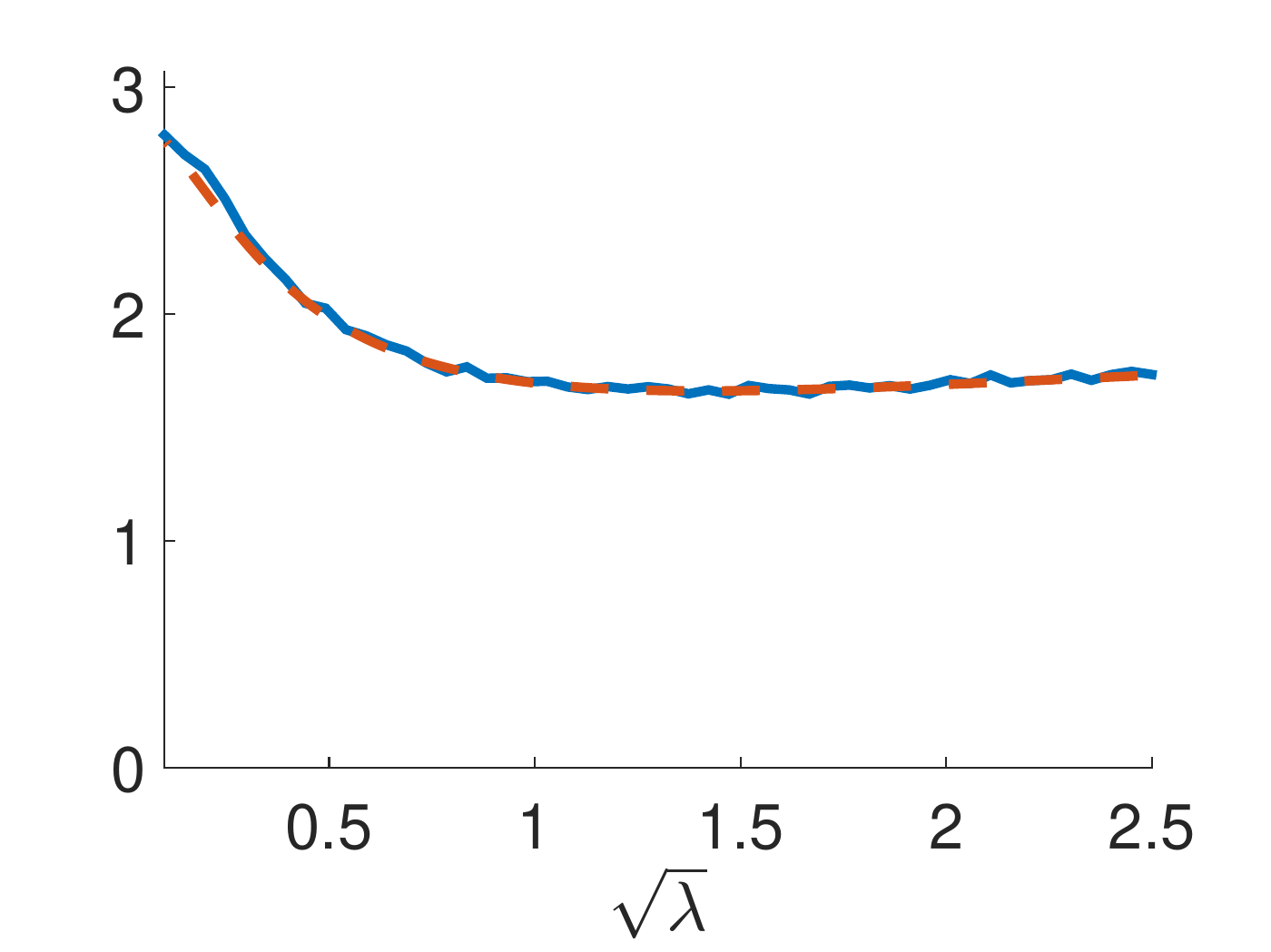} \\
\end{tabular}
\caption{Prediction error of ridge regression in the {\tt BinaryTree} and {\tt Exponential} model. The theoretical formula (red, dashed) is overlaid with the results from simulations (blue, solid).
The signals are drawn from $\smash{w \sim \nn\p{0, p^{-1} I_{p \times p}}}$.  For {\tt BinaryTree}, we train on $n = \gamma^{-1} p$ samples, where $p=2^4$; for {\tt Exponential} on $n=20$. We take 100 instances of random training data sets, and for each we test on 500 samples. We report the average test error over all 50,000 test cases. }
\label{fig:second}
\end{figure}

Intriguingly, Figure \ref{fig:second} shows that the prediction performance of ridge regression is very similar on the two problems. This presents a marked contrast to the RDA example given in the introduction, where the two covariance structures led to very different classification performance.
Thus, it appears that the loss function and the spectrum of $\Sigma$ can interact non-trivially. 

Meanwhile, in the special case of identity covariance, the quantity $R_\lambda-1$ coincides with the normalized estimation error, so we recover known results described in, e.g., \citet{tulino2004random}. When $\smash{\Sigma = I_{p \times p}}$, we have an explicit expression for the Stieltjes transform \citep[e.g.,][p. 52]{bai2010spectral}, valid for $\lambda>0$:
\begin{equation}
m_I(-\lambda; \gamma) = \frac{-(1 - \gamma  + \lambda) + \sqrt{(1 - \gamma + \lambda)^2 + 4\gamma  \lambda}}{2 \gamma \lambda}.
\label{identity_stieltjes}
\end{equation}
Theorem \ref{theo:ridge} implies that the limit predictive risk of ridge regression for general $\lambda$ equals
\begin{equation*}
\label{eq:ridge_identity1}
R_\lambda(\alpha^2,\gamma) = 1 + \gamma m_I(-\lambda; \gamma) + \lambda \left( \lambda \alpha^2 -\gamma \right)  m'_I(-\lambda; \gamma),
\end{equation*}
which has an explicit form. Furthermore, the optimal risk has a particularly simple form: 
\begin{equation}
 R^*(\alpha^2,\gamma)  = \frac{1}{2}\left[ 1 + \frac{\gamma-1}{\gamma}\alpha^2 + \sqrt{\left(1 - \frac{\gamma-1}{\gamma}\alpha^2\right)^2 + 4 \alpha^2}\right].
\label{eq:closed_form}
\end{equation}
See the supplement for details on these derivations.

\subsection{Regimes of Learning}
\label{regimes_of_learn:ridge}

As an application of Theorem \ref{theo:ridge}, we study how the difficulty of ridge regression depends on the signal strength $\alpha^2$. \citet{liang2010interaction} call this the {\it regimes of learning} problem and argue that, for small $\alpha^2$ the complexity of ridge regression should be tightly characterized by dimension-independent Rademacher bounds, while for large $\alpha^2$ the error rate should only depend on $\gamma$. \citet{liang2010interaction} justify their claims using generalization bounds for the identity-covariance case $\Sigma = I_{p \times p}$, and conjecture that similar relationships should hold in general. Using our results, we can give a precise characterization of the regimes of learning of optimally-regularized ridge regression with general covariance $\Sigma$.

From Theorem \ref{theo:ridge}, we know that given a signal strength $\alpha^2$, the asymptotically optimal choice for $\lambda$ is $\smash{\lambda^*(\alpha) = \gamma / \alpha^2}$, in which case the predictive risk of ridge regression converges to
\begin{equation*}
R^*(H,\alpha^2,\gamma) = \frac{1}{\lambda^* v(-\lambda^*)}.
\end{equation*}
We now use this formula to examine the two limiting behaviors of the risk, for weak and strong signals. The results below are proved in the supplement, assuming that the population spectral distribution $H$ is supported on a set bounded away from 0 and infinity. 

The weak-signal limit is relatively simple. First, $\lim_{\alpha^2 \rightarrow 0} R^*(H,\alpha^2,\gamma) = 1$, reflecting that for a small signal, we predict a near-zero outcome due to a large regularization. Second, $\smash{\lim_{\alpha^2 \rightarrow 0} (R^*(H,\alpha^2,\gamma) - 1)/\alpha^2 = \mathbb{E}_HT}$, where $\mathbb{E}_HT$ is the large-sample limit of the normalized traces $p^{-1}\tr\Sigma$. Therefore, for small $\alpha$, the difficulty of the prediction is determined to first order by the average eigenvalue, or equivalently by the average variance of the features, and does not depend on the size of the ratio $\gamma = p/n$.

Conversely, the strong-signal limiting behavior of the risk depends the aspect ratio $\gamma$, and experiences a phase transition at $\gamma = 1$.
When $\gamma < 1$, the predictive risk converges to
\begin{equation*}
\label{eq:gamma_small_regime}
\lim_{\alpha^2 \rightarrow \infty} R^*(H,\alpha^2,\gamma)  = \frac{1}{1 - \gamma}
\end{equation*}
regardless of $\Sigma$. This quantity is known to be the $n,\,p \rightarrow \infty$, $p/n \to \gamma$ limit of the {risk of ordinary least squares} (OLS).
Thus when $p < n$ and we have a very strong signal, ridge regression cannot outperform OLS, although of course it can do much better with a small $\alpha$.

When $\gamma > 1$, the risk $R^*(H,\alpha^2,\gamma) $ can grow unboundedly large with $\alpha$. Moreover, we can verify that
\begin{equation}
\label{eq:gamma_big_regime}
\lim_{\alpha^2 \rightarrow \infty} \alpha^{-2} R^*(H,\alpha^2,\gamma) = \frac{1}{\gamma \, v(0)} \geq 0.
\end{equation}
Thus, the limiting error rate depends on the covariance matrix through $v(0)$.
In general there is no closed-form expression for $v(0)$, which is instead characterized as the unique $c>0$ for which 
$$\frac{1}{\gamma}=\int_{t=0}^{\infty}\frac{tc }{1+tc}\,dH(t).$$
In the special case $\Sigma = I_{p \times p}$, however, the limiting expression simplifies to $\smash{1/(\gamma v(0))}$ =  $\smash{(\gamma - 1) / \gamma}$.
In other words, when $p > n$, optimally tuned ridge regression can capture a {constant fraction of the signal}, and its test-set fraction of explained variance tends to $\gamma^{-1}$.

Finally, in the threshold case $\gamma = 1$, the risk  $R^*(H,\alpha^2,\gamma)$ scales with $\alpha$:
\begin{equation}
\label{eq:gamma_1_regime}
\lim_{\alpha^2 \rightarrow \infty} \alpha^{-1} R^*(H,\alpha^2,\gamma) = \frac{1}{\sqrt{\EE[H]{T^{-1}}}},
\end{equation}
where  $\EE[H]{T^{-1}}$ is the large-sample limit of $p^{-1}\tr\p{\Sigma^{-1}}$.
Thus, the absolute risk $R^*$ diverges to infinity, but the normalized error $\alpha^{-2} R^*(H,\alpha^2,\gamma)$ goes to 0. This appears to be a rather unusual risk profile.
In the case $\Sigma = I_{p \times p}$, our expression simplifies further and we get the finite-$\alpha$ formula
$R^*\p{\alpha^2, \, 1}  =  (\sqrt{4 \alpha^2 + 1} + 1)/2$,
which scales like $\alpha$.

\begin{figure}[t]
\centering
\includegraphics[width = \BIGFIG\textwidth]{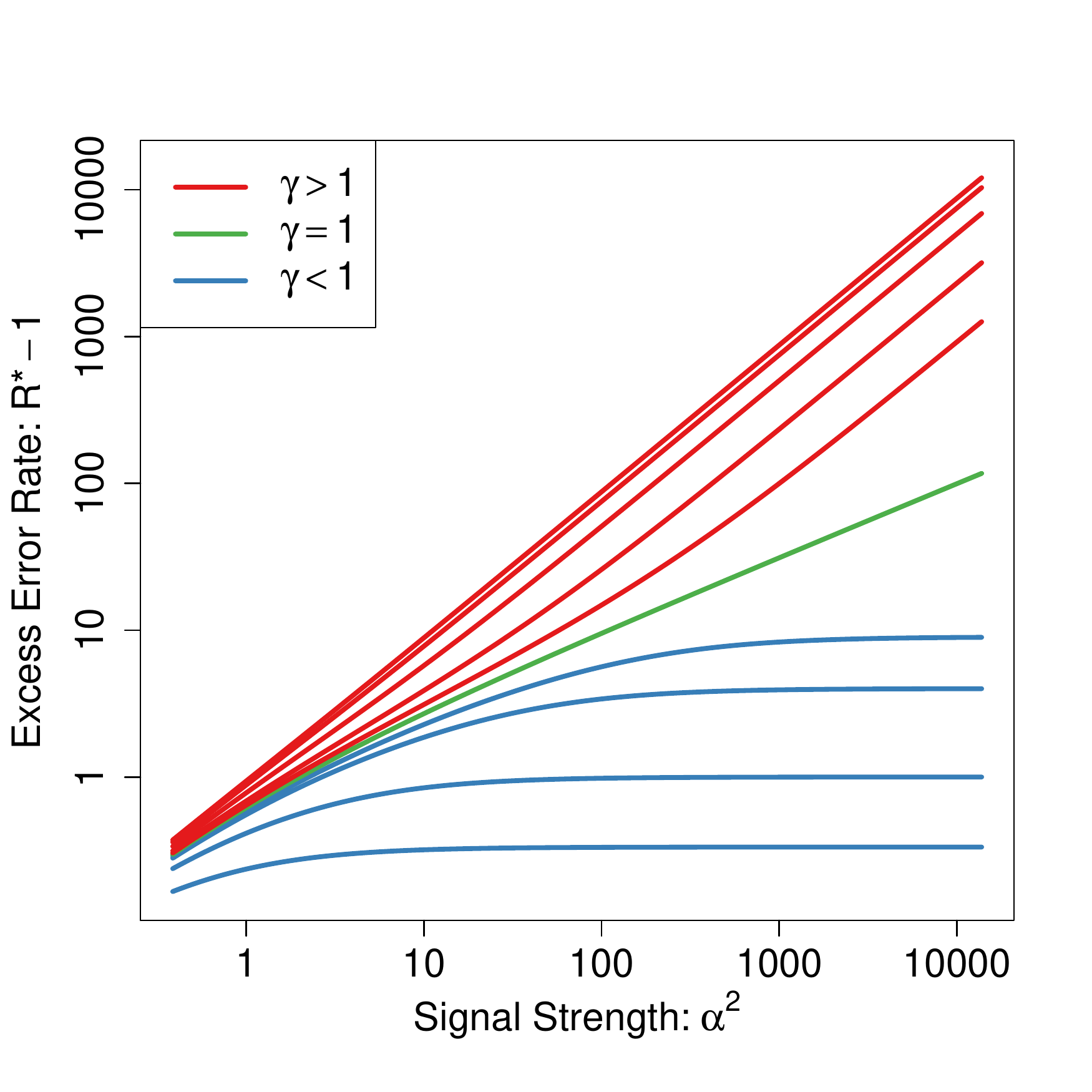}
\caption{Phase transition for predictive risk of ridge regression with identity covariance $\Sigma = I_{p \times p}$. Error rates are plotted for $\gamma = $ 0.25, 0.5, 0.8, 0.9, 1, 1.1, 1.3, 2, 4, and 8.}
\label{fig:ridge}
\end{figure}

In summary, we find that for general covariance $\Sigma$, the strong-signal risk $\smash{R^*(\alpha^2, \, \gamma)}$ scales as $\smash{\Theta(1)}$ if $\gamma < 1$, as $\smash{\Theta(\alpha)}$ if $\gamma = 1$, and as $\smash{\Theta(\alpha^2)}$ if $\gamma > 1$.
We illustrate this phenomenon in Figure \ref{fig:ridge}, in the case of the identity covariance $\Sigma = I_{p \times p}$. We see that when $\gamma < 1$ the error rate stabilizes, whereas when $\gamma > 1$, the error rate eventually gets a slope of 1 on the log-log scale. Finally, when $\gamma = 1$, the error rate has a log-log slope of $1/2$.

Thus, thanks to Theorem \ref{theo:ridge}, we can derive a complete and exact answer the regimes of learning question posed by \citet{liang2010interaction} in the case of ridge regression.
The results \eqref{eq:gamma_big_regime} and \eqref{eq:gamma_1_regime} not only show that the scalings found by  \citet{liang2010interaction} with $\smash{\Sigma = I_{p \times p}}$ hold for arbitrary $\Sigma$, but make explicit how the slopes depend on the limiting population spectral distribution.
The ease with which we were able to read off this scaling from Theorem \ref{theo:ridge} attests to the power of the random matrix approach.

\subsection{An Inaccuracy Principle for Ridge Regression}
\label{sec:pred-estimation}

Our results also reveal an intriguing inverse relationship between the prediction and estimation errors of ridge regression.
Specifically, denoting the mean-squared estimation error as $R_{E,n}(\lambda)=\EE{\lVert \hw_{\lambda}-w^* \rVert^2}$, it is well known that optimally tuned ridge regression satisfies, under the conditions of Theorem \ref{theo:ridge}, 
$$R_{E,n}(\lambda^*) \to_{a.s.} R_{E}  := \gamma \, m\p{-\lambda^*} \eqfor \lambda^* = \gamma \alpha^{-2},$$
where $m$ is the Stieltjes transform of the limiting empirical spectral distribution \citep[see, e.g.,][Chapter 3]{tulino2004random}.
Combining this result with Theorem \ref{theo:ridge} and \eqref{dual.ST}, we find the following relationship between the limiting predictive risk $R_P$ and the limiting estimation risk $R_E$.

\begin{corollary}
Under the conditions of Theorem \ref{theo:ridge}, the asymptotic predictive and estimation risks of optimally-tuned ridge regression are inversely related,
$$1 - \frac{1}{R_P} = \gamma \left(1 - \frac{R_E}{\alpha^2} \right).$$
\end{corollary}

The equation holds for all limit eigenvalue distributions $H$ of the covariance matrices $\Sigma$. Both sides of the above equation are non-negative: $R_P$ cannot fall below the intrinsic noise level $\Var{Y \cond X} = 1$, while $R_E \leq \limsup_{p\to\infty} R_{E, \, n} (\lambda^*) \leq \limsup_{p\to\infty} R_{E, \, n}(0) = \alpha^2$.
When $\gamma = 1$, we get the even simpler equation
$$ R_{E} \, R_{P}  = \alpha^2. $$
The product of the estimation and prediction risks equals the signal strength. 
Since this holds for the optimal $\lambda^*$, it also implies that for any $\lambda$ we have the lower bound $R_{E}(\lambda) \cdot R_{P}(\lambda)  \ge \alpha^2$; we find the explicit formula relating the two risks remarkable. The inverse relationship may be somewhat surprising, but it has an intuitive explanation.\footnote{A similar heuristic was given by \citet{liang2010interaction}, without theoretical justification.}
When the features are highly correlated and $v$ is correspondingly large, prediction is easy because $y$ lies close to the ``small'' column space of the feature matrix $X$, but estimation of $w$ is hard due to multi-collinearity.
As correlation decreases, prediction gets harder but estimation gets easier.

\subsection{Related Work for High-Dimensional Ridge Regression}
\label{sec:ridge_review}

Random-design ridge regression in high dimensions is a thoroughly studied topic. In particular,  \citet{karoui2013asymptotic} and \citet{dicker2014ridge} study ridge regression with identity covariance $\Sigma = I_{p \times p}$ in an asymptotic framework similar to ours; this special case is considerably more restrictive than a general covariance.
The study of the estimation error $\smash{\mathbb{E}\big[\Norm{\hw_\lambda - w}^2\big]}$ of ridge regression has received substantial attention in the wireless communication literature; see, e.g., \citet{couillet2011random} and \citet{ tulino2004random} for references. To our knowledge, however, that literature has not addressed the behavior of prediction error.
Finally, we also note the work of \citet{hsu2014random}, who provide finite-sample concentration inequalities on the out-of-sample prediction error of random-design ridge regression, without obtaining limiting formulas. In contrast to these results, we give explicit limiting formulas for the prediction error.

\section{Regularized Discriminant Analysis}
\label{sec:twoclass}

In the second part of the paper, we return to regularized discriminant analysis and the two-class Gaussian discrimination problem \eqref{eq:setup}. For simplicity we will first discuss the case when the population label proportions are balanced. 
In this case, the Bayes oracle predicts using \citep{anderson1958introduction}
\begin{align*}
&\hy(x) = \sign\p{\delta^\top \Sigma^{-1} \p{x - \frac{\mu_{-1} + \mu_{+ 1}}{2}}}  \with \delta = \frac{\mu_{+1} - \mu_{-1}}{2},  
\end{align*}
and has an error rate \smash{$\text{Err}_\text{Bayes} = \Phi\p{-\Delta_{n,p}}$},
where \smash{$\Delta_{n,p} = \sqrt{\delta^\top \Sigma^{-1} \delta}$}
is half the between-class Mahalanobis distance.
The Gaussian classification problem has a rich history, going back to Fisher's pioneering work on linear discriminant analysis (LDA). When we have the same number of examples from both the positive and negative classes, i.e., $\smash{n_{-1} = n_{+1} = n/2}$, LDA classifies using the linear rule
\begin{align*}
&\hy = \sign\p{\hdelta^\top \hSigma_c^{-1} \p{x - \frac{\hmu_{-1} + \hmu_{+1}}{2}}}, \where \\
&\hdelta = \frac{\hmu_{+1} - \hmu_{-1}}{2}, \
\hSigma_c = \frac{1}{n - 2} \sum_{i = 1}^n \p{x_i - \hmu_{y_i}}^{\otimes 2}, \eqand
\hmu_{\pm 1} = \frac{2}{n} \sum_{\cb{i : y_i = \pm 1}}  x_i.
\end{align*}
Here $\hSigma_c$ is the centered covariance matrix. In the low-dimensional case where $n$ gets large while $p$ remains fixed, LDA is the natural plug-in rule for Gaussian classification and efficiently converges to the Bayes discrimination function \citep{anderson1958introduction,efron1975efficiency}.
When $p$ is on the same order as $n$, however, the matrix inverse $\smash{\hSigma_c^{-1}}$ is unstable and the performance of LDA declines, as discussed among others by \cite{bickel2004some}.
Instead, we will study regularized discriminant analysis, the linear rule $\smash{\hy = h_{\hw_{\lambda}}\p{x - (\hmu_{-1} + \hmu_{+1})/{2}}}$ where $\smash{h_{w}(x) = \sign(w \cdot x)}$ and the weight vector is \footnote{The notation was chosen to emphasize the similarities between ridge regression and RDA. There will be no possibility for confusion with the ridge regression weight vector, also denoted $\hw_{\lambda}$.} $\smash{\hw_{\lambda} = (\hSigma_c + \lambda I_{p \times p})^{-1}\hdelta}$ \citep{friedman1989regularized, serdobolskii1983minimum}.

\subsection{High-Dimensional Asymptotics}

Throughout this section, we make a random-effects assumption about the class means, where the expected signal strength is $\EE{\|\delta\|_2^2}=\alpha^2$. 
We denote the classification error of RDA as $\Err\p{\hw_\lambda} = \PP{y \neq \sign\{\hw_\lambda \cdot (x-\hmu)\}}$. The probability is with respect to a independent test data point $(x,y)$ from the same distribution as the training data.

\begin{assumption}{C}\label{assume:C}(random weights in classification) The following conditions hold:
\begin{enumerate}
\item  $\mu_{-1}$ and $\mu_{+1}$ are randomly generated as $\mu_{-1} = \bmu - \delta$ and $\mu_{+1} = \bmu + \delta$, where $\delta$ has i.i.d. coordinates with
$$\EE{\delta_i} = 0, \ \Var{\delta_i} = \frac{\alpha^2}{p}, \eqand \EE{\delta_i^{4 + \eta}} \leq \frac{C}{p^{2 + \eta/2}} $$
for some fixed constants $\eta>0$ and $C$.
\item $\smash{\bmu = (\mu_{-1} + \mu_{+1})/2}$ is either fixed, or random and independent of $\delta$, $X$ and $y$, and satisfies $\limsup_{p\to\infty} \|\bmu\|^{2}_2/p^{1/2-\zeta} \leq C$ almost surely for some fixed constants $\zeta>0$ and $C$.
\end{enumerate}
\end{assumption}

\begin{theorem}
\label{theo:rda}
Under parts 2 and 3 of Assumption \ref{assume:A}, and Assumption \ref{assume:C}, suppose moreover that that the eigenvalues of $\Sigma$ are uniformly bounded as $0 < b < \lambda_{\min}\p{\Sigma} \leq \lambda_{\max}\p{\Sigma} \leq B$ for some fixed constants $b$ and $B$.
Finally, suppose that we have a balanced population $\PP{y = 1} = 1/2$ and equal class sizes $n_{-1}=n_{+1}$.
Then, the classification error of RDA converges almost surely
\begin{equation*}
\label{eq:rda_risk}
 \Err\p{\hw_\lambda} \rightarrow_{a.s.} \Phi\p{ - \Theta(\lambda)}, \ \where \ \Theta(\lambda) = \frac{\alpha^2 \, \tau\p{\lambda} }{\sqrt{\alpha^2 \, \eta\p{\lambda} + \xi\p{\lambda}}}
\end{equation*}
and $\tau$, $\eta$, and $\xi$ are determined by the problem parameters $H$ and $\gamma$:
\begin{equation*}
\label{eq:wlambda}
\tau\p{\lambda}  = \lambda m v, \,\,\,\,
\eta\p{\lambda}  = \frac{v-\lambda v'}{\gamma}, \,\,\,\,
\xi\p{\lambda} =\frac{v'}{v^2}-1 .
\end{equation*}
Here, $m=m(-\lambda)$ is the Stieltjes transform of the limit empirical spectral distribution $F$ of the covariance matrix $\smash{\hSigma_c}$, and $v=v(-\lambda)$ is the companion Stieltjes transform defined in \eqref{dual.ST}.
\end{theorem}

The proof of Theorem \ref{theo:rda} (in the supplement) is similar to Theorem \ref{theo:ridge}, but more involved. The main difficulty is to evaluate explicitly the limits of certain functionals of the population and sample covariance matrices. We extend the result of \citet{ledoit2011eigenvectors}, and rely on additional results and ideas from \citet{hachem2007deterministic} and \citet{chen2011regularized}. In particular, we use a derivative trick for Stieltjes transforms, similar to that employed in a related context by \citet{karoui2011geometric}, \citet{rubio2012performance}, and \citet{zhang2013finite}. The limits are then combined with results on concentration of quadratic forms.

The above result can also be extended to RDA with uneven sampling proportions. Since the limit error rates get more verbose, this is the only place where we discuss uneven sampling. Suppose that the conditions of Theorem \ref{theo:rda} hold, except now we have unequal underlying class probabilities $\mathbb{P}(y_i=+1)=\pi_+$, $\mathbb{P}(y_i=-1)=\pi_-$, and our training set is comprised of $n_{\pm 1}$ samples with label $y_i= \pm 1$ such that $p/n_{\pm 1} \to \gamma_{\pm 1}>0$. We do not assume that $n_{-1} / (n_{-1} + n_{+1}) \rightarrow \pi_-$. 
Consider a general regularized classifier $\sign(\hat f_\lambda(x))$, where $\hat f_\lambda(x) = [x-(\hmu_{+1}+\hmu_{-1})/2]^\top  (\hSigma_c + \lambda I_{p \times p})^{-1} [\hmu_{+1}-\hmu_{-1}] + c$ for some $c \in \RR$, where $\hSigma_c$, $\hmu_{\pm1}$ are defined in the usual way. We prove in the supplement that 

\begin{theorem}
\label{theo:rda_unequal}
Under the conditions of Theorem \ref{theo:rda}, and with unequal sampling, the classification error of RDA converges almost surely:
\begin{equation}
\label{eq:unequal_sampling}
\mathbb{P}(\sign(\hat f_\lambda(x) \neq y)) \to_{a.s.}  \pi_- \Phi\p{-\Theta_-} + \pi_+ \Phi\p{-\Theta_+},
\end{equation}
where the effective classification margins have the form
\begin{align*}
&\Theta_{\pm} = \mp \frac{ \pm \alpha^2 m(-\lambda) + \frac{\gamma_{-1}-\gamma_{+1}}{4} \frac{1}{\gamma}\p{\frac{1}{\lambda v} - 1} + c }{\sqrt{Q}  }, \eqand \\
&Q = \alpha^2 \frac{v-\lambda v'}{\gamma (\lambda v)^2} + \frac{\gamma_{-1}+\gamma_{+1}}{4}  \frac{v'-v^2}{ \lambda^2 v^4}.
\end{align*}
\end{theorem}

In this section we assumed Gaussianity, but by a Lindeberg--type argument it should be possible to extend the result to non-Gaussian observations with matching moments. It should also be interesting to study how sensitive the results are to the particular assumptions of our model, such as independence across samples.

It is worth mentioning that the regression and classification problems are very different statistically. In the random effects linear model, ridge regression is a linear Bayes estimator, thus the ridge regularization $\hSigma+\lambda I_{p\times p}$ of the covariance matrix is justified statistically. However, for classification, the ridge regularization is merely a heuristic to help with the ill-conditioned sample covariance. It is thus interesting to know how much this heuristic helps improve upon un-regularized LDA, and how close we get to the Bayes error. We now turn to this problem, which can be studied equivalently from a geometric perspective. 

\subsection{The Geometry of RDA}
\label{sec:rel_margin}

The asymptotics of RDA can be understood in terms of a simple picture. The angle between the Bayes decision boundary hyperplane and the RDA discriminating hyperplane tends to an asymptotically deterministic value in the metric of the covariance matrix, and the limiting risk of RDA can be described in terms of this angle.

Recall that, in the balanced $\pi_+ = \pi_-$ and $n_+ = n_-$ case, the estimated RDA weight vector is $\hw_{\lambda} = (\hSigma_c + \lambda I_{p\times p})^{-1} \hdelta$, while the Bayes weight vector is $w^* = \Sigma^{-1} \delta$. Writing $\langle a,b\rangle_\Sigma = a^\top \Sigma b$ for the inner product in the $\Sigma$-metric, the cosine of the angle---in the same metric---between the two is
\begin{equation*}
\cos_{\,\Sigma}(w,\hw_\lambda) = \frac{\left\langle \hw_{\lambda}, \, \Sigma^{-1} \delta\right\rangle_\Sigma}{\Norm{ \hw_{\lambda}}_\Sigma \, \Norm{ \Sigma^{-1} \delta}_\Sigma}  = \frac{\hw_{\lambda}^{\top} \delta}{\sqrt{\hw_{\lambda}^\top \Sigma \hw_{\lambda}} \sqrt{\delta^\top \Sigma^{-1} \delta}}
\end{equation*}
As discussed earlier, the Bayes error rate for the two-class Gaussian problem is $
\Err_{\text{Bayes}} = \Phi\p{-\Delta_{n,p}}$,  with  $\Delta_{n,p} = \sqrt{\delta^\top \Sigma^{-1} \delta} \rightarrow_{a.s.} \Delta = \alpha \sqrt{\mathbb{E}_H\left[T^{-1}\right]}$.\footnote{To show this, we observe that the quantity $\delta^\top \Sigma^{-1} \delta$ converges almost  surely to $\alpha^2\mathbb{E}_H\left[T^{-1}\right]$ under the conditions of Theorem \ref{theo:rda}. This follows by a concentration of quadratic forms argument stated in the supplement, and because the spectrum of $\Sigma$ converges to $H$, so in particular $\EE{\delta^\top \Sigma^{-1} \delta} = \alpha^2 p^{-1}\tr \Sigma^{-1} \to 
\alpha^2\mathbb{E}_H\left[T^{-1}\right]$.}
Therefore, from Theorem \ref{theo:rda} it follows that the cosine of the angle $\cos_{\,\Sigma}(w,\hw_\lambda)$ has a deterministic limit, denoted by $\Gamma(H,\gamma,\alpha^2, \lambda)$, given by 
\begin{equation*}
\label{eq:arm}
\Gamma\p{H,\gamma,\alpha^2,\lambda} = \Theta\p{H,\gamma,\alpha^2,\lambda} \big/ \Delta \in \sqb{0, \, 1},
\end{equation*}
where $\Theta\p{H,\gamma,\alpha^2,\lambda}$ is the classification margin of RDA with the dependence on each parameter made explicit.
The limit of the cosine quantifies the inefficiency of the RDA estimator relative to the Bayes one.

We gain some insight into this angle for two special cases: when $H=\delta_1$, and by taking the limit $\alpha^2\to \infty$. First, with $H = \delta_1$, or equivalently $\Sigma = I_{p\times p}$, we curiously find that the effects of the signal strength $\alpha^2$ and regularization rate $\lambda$ decouple completely, as shown in Corollary \ref{cor:rda_id} below.
See the supplement for a proof of the following result.

\begin{corollary}
\label{cor:rda_id}
Under the conditions of Theorem \ref{theo:rda}, let $\smash{\Sigma = I_{p \times p}}$ for all $p$. Then, the limiting cosine $\Gamma$ of the angle between the Bayes and RDA hyperplanes is
\begin{equation*}
 \Gamma(\delta_1,\gamma,\alpha^2, \lambda) =  \frac{\alpha}{\sqrt{\alpha^2 + \gamma}} \, \sqrt{\frac{1 + \gamma \lambda  \, m_I^2(-\lambda;\gamma)}{1+ \gamma \,  m_I(-\lambda;\gamma)}},
 \end{equation*}
where the Stieltjes transform $m_I(-\lambda;\gamma)$ for $\Sigma = I_{p \times p}$ is given in \eqref{identity_stieltjes}.
For $\gamma=1$, this expression simplifies further to
\begin{equation*}
 \Gamma(\delta_1,1,\alpha^2, \lambda) = \frac{\alpha}{\sqrt{\alpha^2 + 1}} \, \frac{2 \left[\lambda\p{\lambda+4}\right]^{1/4}}{\lambda^{1/2}+ \p{\lambda+4}^{1/2}}.
\end{equation*}
\end{corollary}

Examining $\Gamma(\delta_1,\gamma,\alpha^2, \lambda)$, we can attribute the suboptimality to two sources of noise:
We need to pay a price $\smash{{\alpha}/ {\sqrt{\alpha^2 + \gamma}}}$ for estimating $\smash{ \mu_{\pm 1}}$, while the cost of estimating $\Sigma$ is $\smash{([1 + \gamma \lambda  \, m_I^2(-\lambda;\gamma)] / [1+ \gamma \,  m_I(-\lambda;\gamma)])^{1/2}}$.
If we knew that $\Sigma = I$, we could send $\lambda \rightarrow \infty$. It is easy to verify that this would eliminate the second term, leading to a loss of efficiency $\alpha / \sqrt{\alpha^2 + \gamma}$.

In the case of a general covariance matrix $\Sigma$ we get a similar asymptotic decoupling in the strong-signal limit $\alpha^2 \rightarrow \infty$.
The following claim follows immediately from Theorem \ref{theo:rda}.

\begin{corollary}
\label{coro:arm}
Under the conditions of Theorem \ref{theo:rda}, the cosine of the angle between the optimal and learned hyperplanes has the limit as $\alpha^2 \to \infty$: 
\begin{equation*}
\label{eq:arm_strong}
\lim_{\alpha \rightarrow \infty} \Gamma(H,\gamma,\alpha^2, \lambda) = \frac{\tau\p{\lambda}}{\sqrt{\eta\p{\lambda} \, \mathbb{E}_H\left[T^{-1}\right]}}.
\end{equation*}
\end{corollary}

Thus, RDA is in general inconsistent for the Bayes hyperplane in the case of strong signals. Corollary \ref{coro:arm} also implies that, in the limit $\alpha \rightarrow \infty$, the optimal $\lambda$ for RDA converges to a non-trivial limit that only depends on the spectral distribution $H$. No such result is true for ridge regression, where $\lambda^* = \alpha^{-2} \gamma \rightarrow 0$ as $\alpha \rightarrow \infty$, regardless of $\Sigma$. This contrast arises because ridge regression is a linear Bayes estimator, while RDA is a heuristic which is strictly suboptimal to the Bayes classifier. 

\begin{figure}[t]
\centering
\includegraphics[width=\BIGFIG\columnwidth]{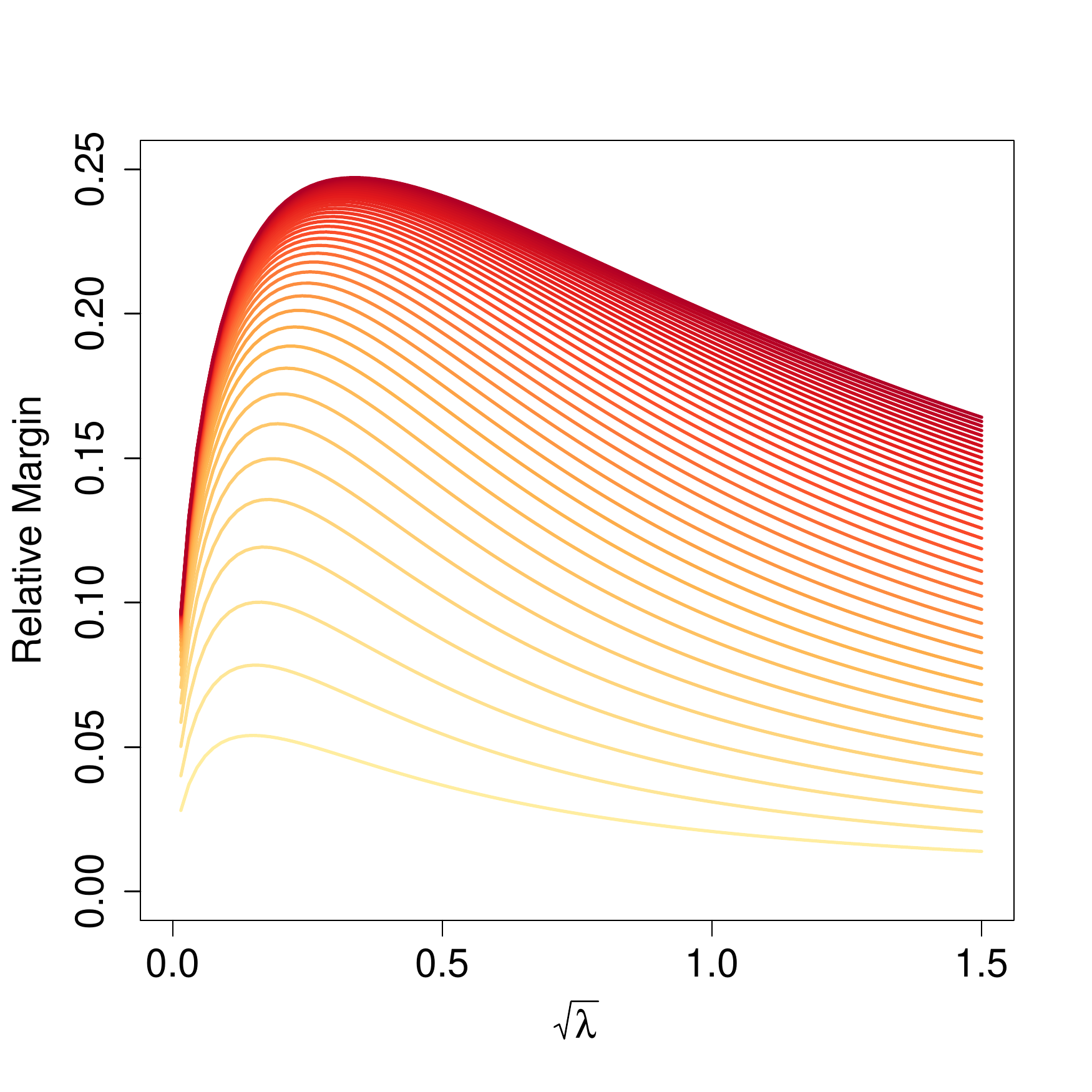}
\caption{The cosine $\Gamma(H,\gamma,\alpha^2, \lambda)$ for the {\tt AR-1}(0.9) model, with $\alpha \in [0.1, 2]$. The values of $\alpha$ used for each curve are evenly spaced, with a gap of 0.05 between each curve.  The cosine quickly converges to a limit as $\alpha$ increases.}
\label{fig:arm}
\end{figure}

We illustrate the behavior of the cosine $\Gamma$ for the {\tt AR-1}(0.9) model in Figure \ref{fig:arm}, which displays $\Gamma$ for values of $\alpha$ ranging from $\alpha = 0.1$ to $\alpha = 2$. We see that the $\Gamma$-curve converges to its large-$\alpha$ limit fairly rapidly. Moreover, somewhat strikingly, we see that the optimal regularization parameter $\lambda^*$, i.e., the maximizer of $\Gamma$, increases with the signal strength $\alpha^2$. 

Finally, we note that \citet{efron1975efficiency} studies the angle $\Gamma$ in detail for low-dimensional asymptotics where $p$ is fixed while $n \rightarrow \infty$; in this case, $\Gamma$ converges in probability to $1$, and the sampling distribution of $n(1 - \Gamma)$ converges to a (scaled) $\chi_{p - 1}^2$ distribution. Establishing the sampling distribution in high dimensions is interesting future work.

\subsection{Do existing theories explain the behavior of RDA?}
\label{sec:ar1}

\begin{figure}[p]
\centering
\begin{tabular}{ccc}
& \scalebox{1.2}{$\alpha^2 = 1$} &  \scalebox{1.2}{$\sqrt{\EE{\Delta_{n,p}^2}} = 2.3$} \\
{\begin{sideways}\parbox{\PBW\columnwidth}{\centering $\rho = 0.1$}\end{sideways}} &
\includegraphics[width=\FW, trim = \TRA mm \TRB mm \TRC mm \TRD mm, clip = TRUE]{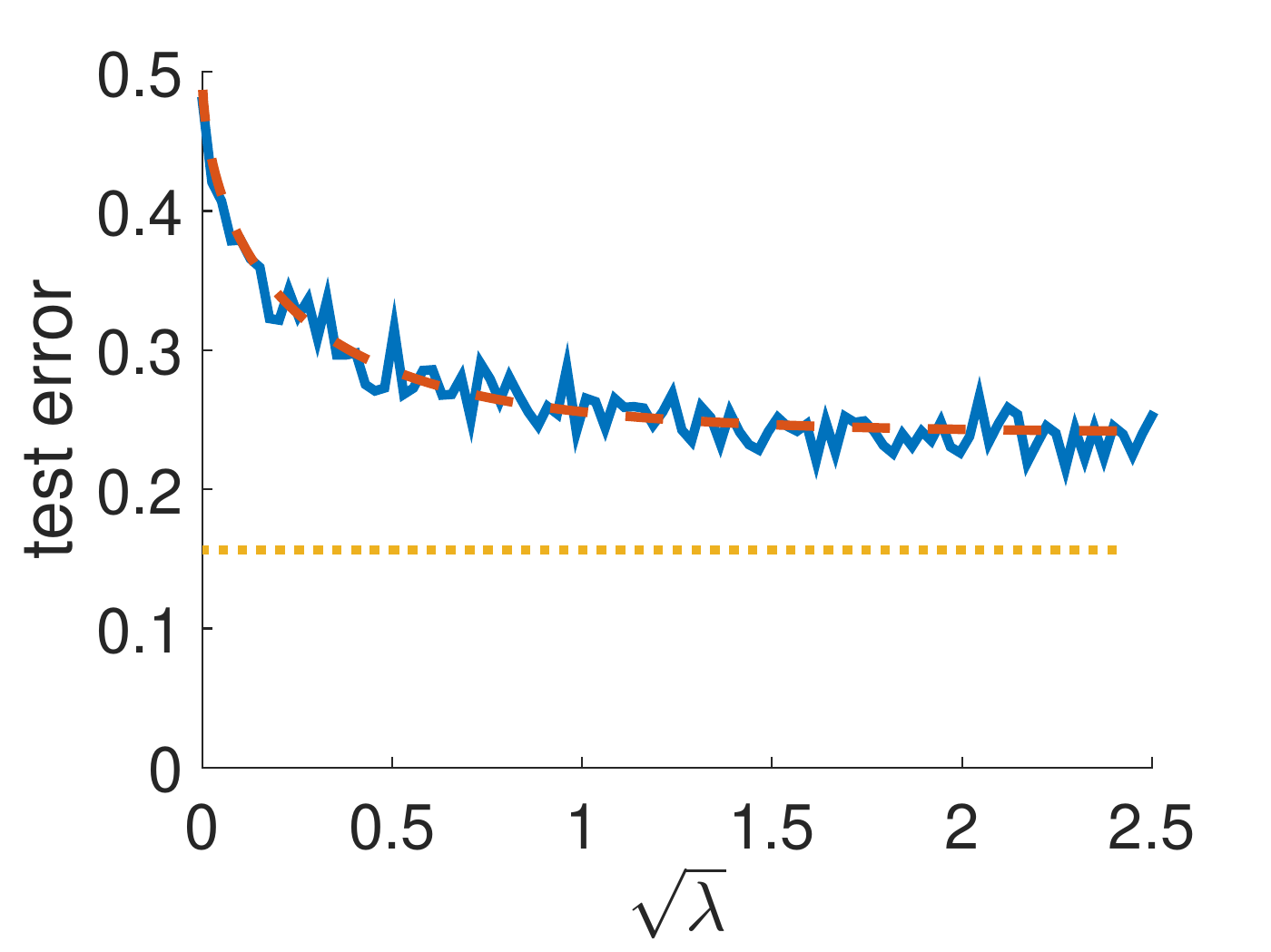} &
\includegraphics[width=\FW, trim = \TRA mm \TRB mm \TRC mm \TRD mm, clip = TRUE]{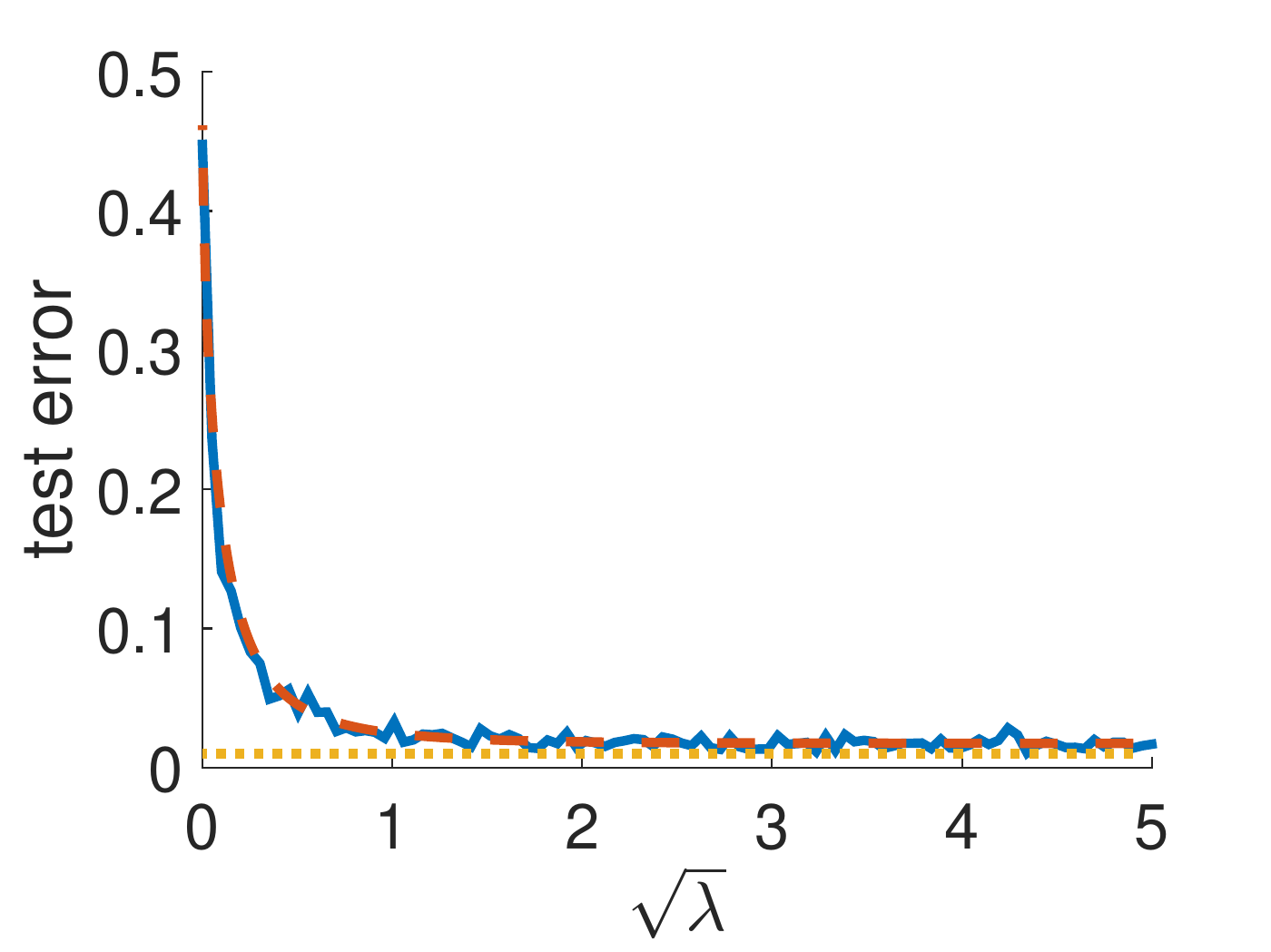} \\
{\begin{sideways}\parbox{\PBW\columnwidth}{\centering $\rho = 0.5$}\end{sideways}} &
\includegraphics[width=\FW, trim = \TRA mm \TRB mm \TRC mm \TRD mm, clip = TRUE]{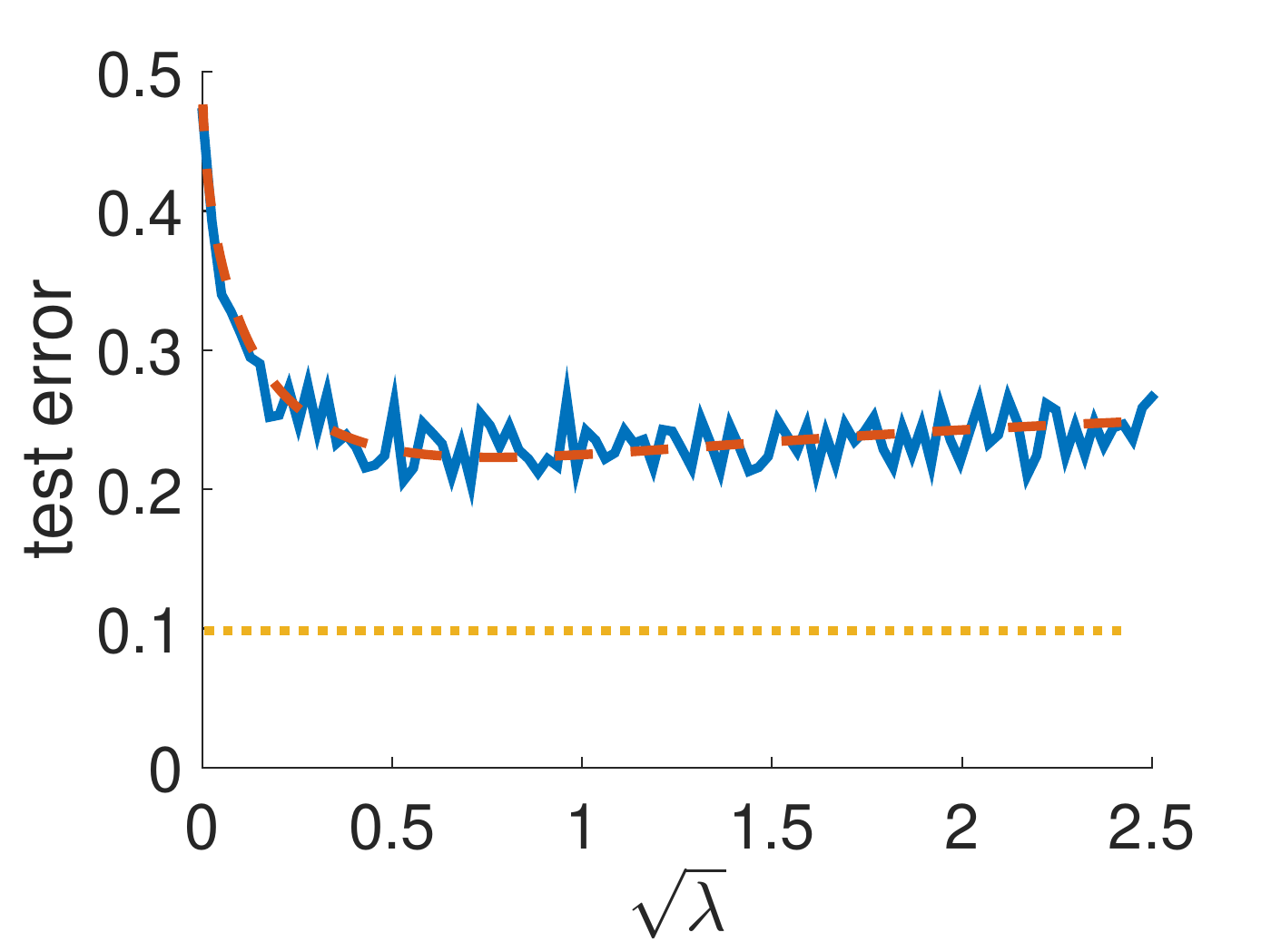} &
\includegraphics[width=\FW, trim = \TRA mm \TRB mm \TRC mm \TRD mm, clip = TRUE]{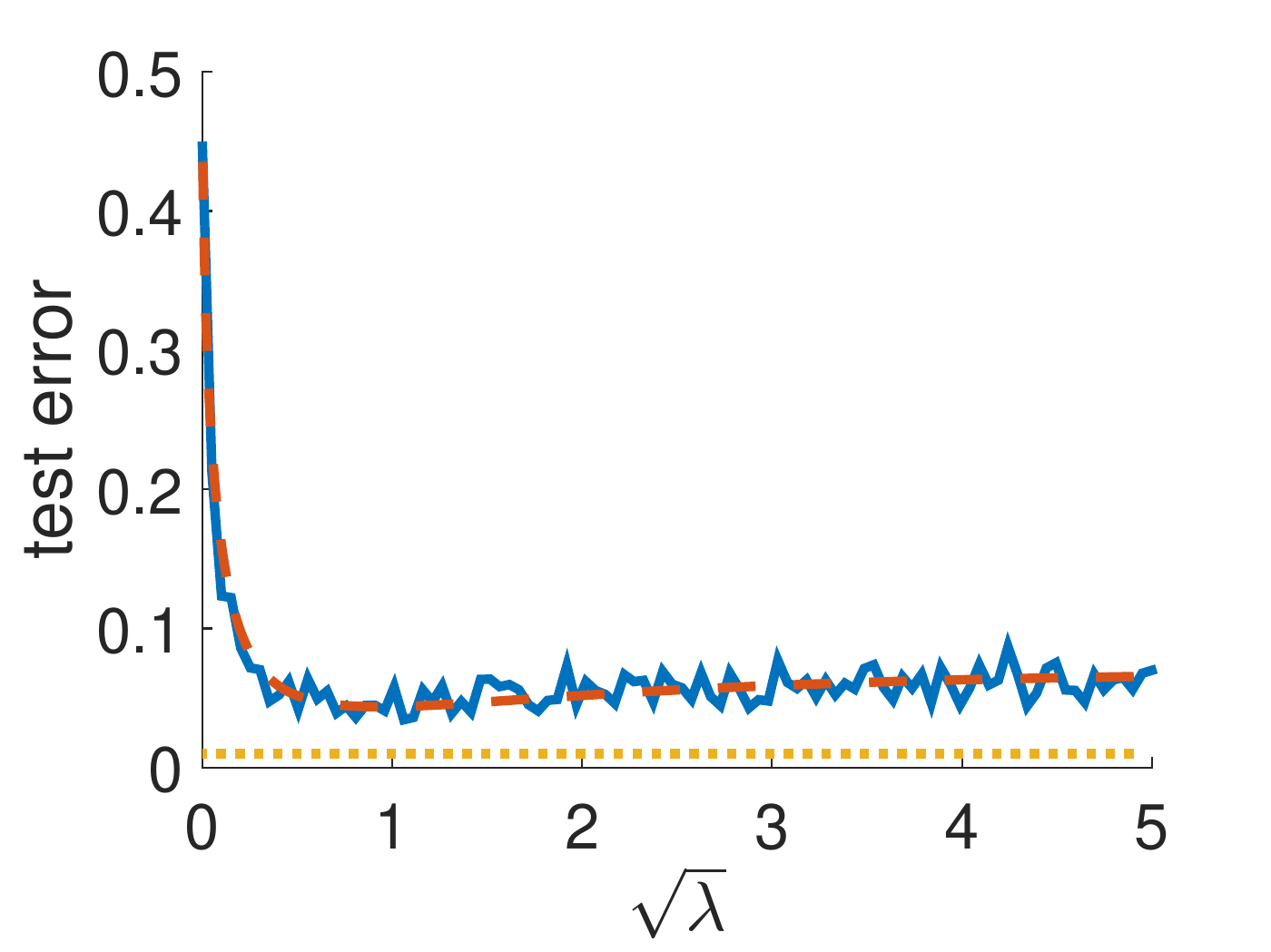}  \\
{\begin{sideways}\parbox{\PBW\columnwidth}{\centering $\rho = 0.9$}\end{sideways}} &
\includegraphics[width=\FW, trim = \TRA mm \TRB mm \TRC mm \TRD mm, clip = TRUE]{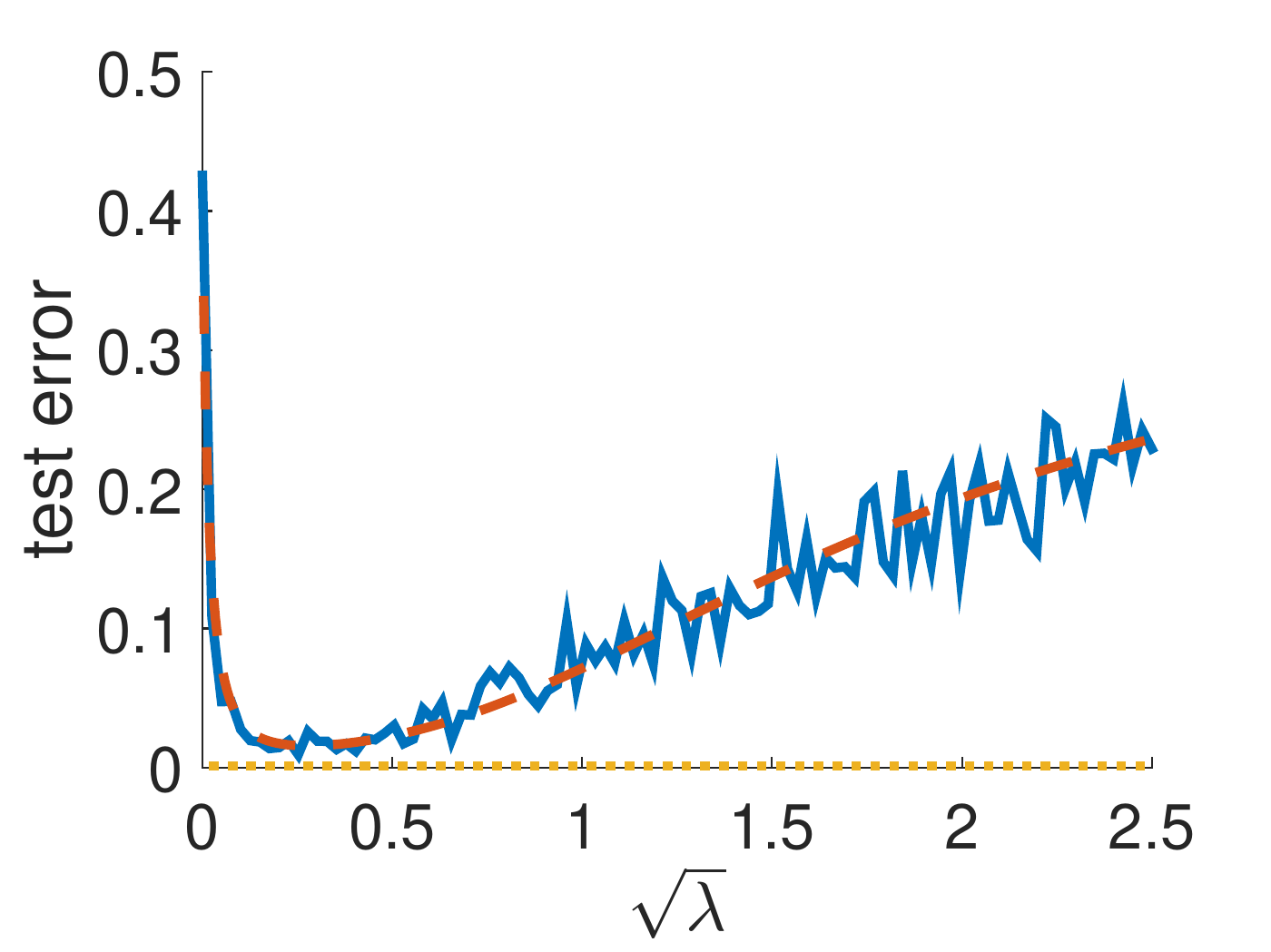} &
\includegraphics[width=\FW, trim = \TRA mm \TRB mm \TRC mm \TRD mm, clip = TRUE]{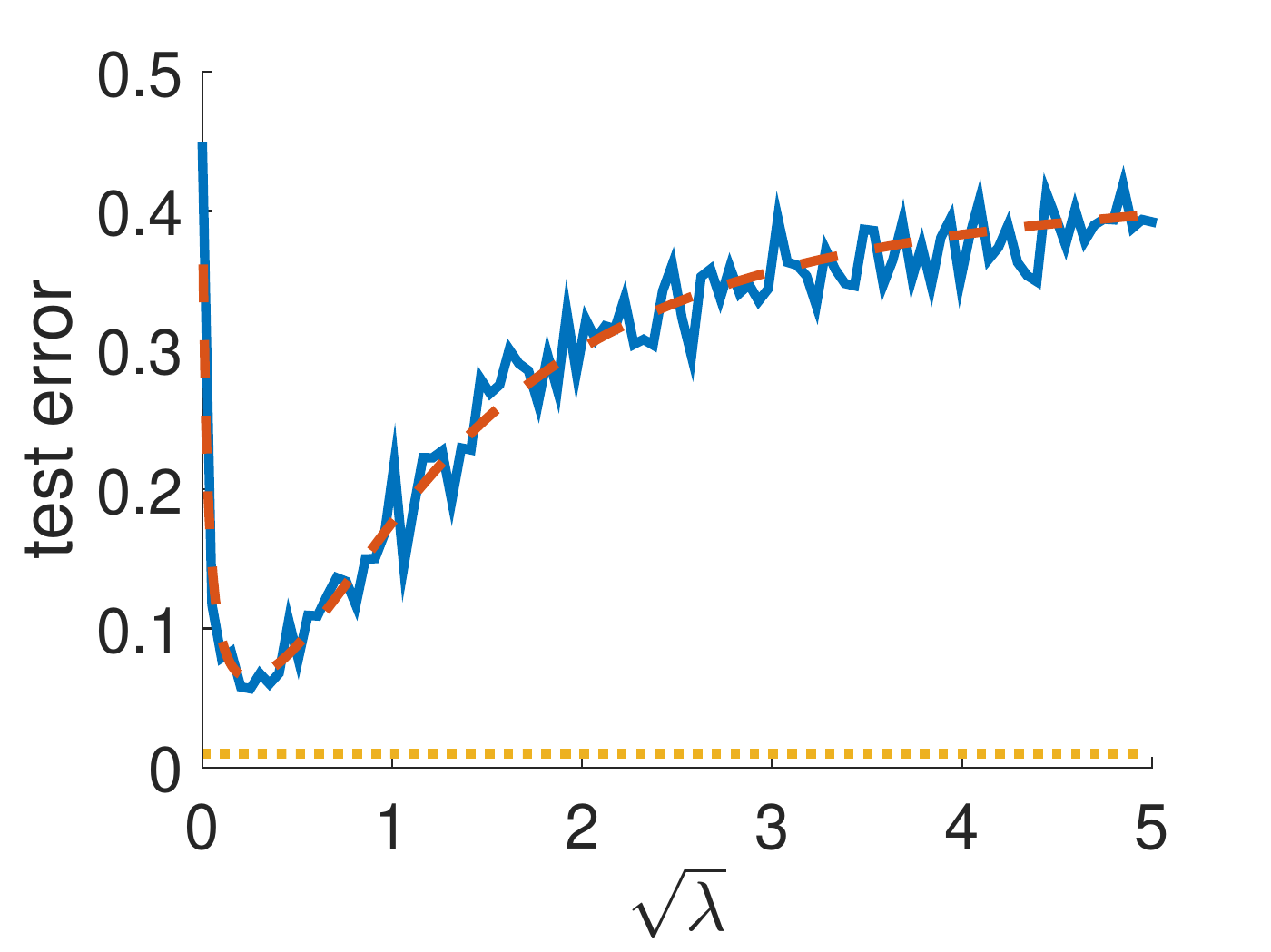} \\
\end{tabular}
\caption{Classification error of RDA in an {\tt AR-1} model. The theoretical formula (red, dashed) is overlaid with the results from simulations (blue, solid; we also display the oracle error (yellow, dotted). In the first column, we keep the signal strength fixed at $\alpha^2 = 1$, whereas in the second column we picked $\alpha$ such as to fix the oracle error at $\Err_{\text{Bayes}} = 0.01$.  We test on 10,000 new samples, and report the average classification error.}
\label{fig:ar1}
\end{figure}

Theorems \ref{theo:rda} and \ref{theo:rda_unequal} give precise information about the error rate of RDA. It is of interest to compare this to classical theories, such as Vapnik-Chervonenkis theory or Rademacher bounds, to if they explain the qualitative behavior of RDA. In this section, we study a simple simulation example, and conclude that existing theory does not precisely explain the behavior of RDA.

We consider a setup with $n =p = 500$, $\smash{\Sigma}$ an auto-regressive ({\tt AR-1}) matrix such that $\smash{\Sigma_{ij} = \rho^{\abs{i - j}}}$, and $\smash{\mu_{\pm 1} \sim \nn(0, \,  \alpha^2 p^{-1} I_{p \times p})}$. This is a natural model when the features can be ordered such that correlations decay with distance; for instance in time series and genetic data.
We run experiments for different values of $\rho$ in two different settings: once with constant effect size $\smash{\alpha^2 = 1}$, and once with constant oracle margin $\smash{\sqrt{\mathbb{E}[\Delta_{n,p}^2]} = 2.3}$. Given $\alpha \geq 0$ and $\rho \in [0, 1)$ one can verify using the description of the limit population spectrum \citep{grenander1984toeplitz}, and by elementary calculations, that the limiting oracle classification margin in the {\tt AR-1} model is \smash{$\Delta = \alpha \sqrt{\p{1 + \rho^2} / \p{1 - \rho^2}}$}; thus, with constant $\alpha^2$ the oracle classifier improves as $\rho$ increases.

Existing results give us some intuition about what to expect. Since $n = p$, classical heuristics based on the theory of \citet{vapnik1971uniform} as well as more specialized analyses \citep{saranadasa1993asymptotic,bickel2004some} predict that unregularized LDA will not work. As we will see, this matches our simulation results.
Meanwhile, \citet{bickel2004some} study worst-case performance of the independence rule relative to the Bayes rule. In our setting, it can be verified that their results imply $\Theta_{\text{IR}} \geq (1 - \rho^2)/(1 + \rho^2) \, \Delta$, where the error rate of the independence rule is $\Phi(-\Theta_{\text{IR}})$. This predicts that independence rules will work better for small correlation $\rho$, which again will match the simulations. 

The existing theory, however, is much less helpful for understanding the behavior of RDA for intermediate values of $\lambda$.
A learning theoretic analysis based on Rademacher complexity suggests that the generalization performance of RDA should depend on terms that scale like $\smash{\sqrt{\lVert{\hw_{\lambda}}\rVert_2^2 \tr{\Sigma} / n} \asymp \sqrt{\lambda^{-2} p / n}}$ for large values of $\lambda$ \citep[e.g.,][]{bartlett2003rademacher}.
In other words, based on a classical approach, we might expect that mildly regularized RDA should not work, but using a large $\lambda$ may help.
Rademacher theory is not tight enough to predict what will happen for $\lambda \approx 1$.

Given this background, Figure \ref{fig:ar1} displays the performance of RDA for different values of $\rho$, along with our theoretically derived error from Theorem \ref{theo:rda}. In the $\smash{\alpha^2 = 1}$ case, we find that---as predicted---unregularized LDA does poorly. However, when $\rho$ is large, mildly regularized RDA does quite well.

Strikingly, RDA is able to benefit from the growth of the oracle classification margin with $\rho$, but only if we use a small positive value of $\lambda$. The analyses based on unregularized LDA or ``infinitely regularized'' independence rules do not cover this case. Moreover, this phenomenon is not predicted by Rademacher theory, which requires $\lambda \gg 1$ to improve over basic Vapnik-Chervonenkis bounds. 
Results from constant margin case $\smash{\sqrt{\mathbb{E}[\Delta_{n,p}^2]} = 2.3}$ reinforce the same interpretations. Finally, our formulas for the error rate are accurate despite the moderate sample size $n = p = 500$.

\subsection{Linear Discriminant Analysis vs. Independence Rules}
\label{lda_vs_ir}

Two points along the RDA risk curve that allow for particularly simple analytic expressions occur as $\lambda \to 0$ and $\lambda \to \infty$:
the former is just classical linear discriminant analysis while the latter is equivalent to an independence rule (or ``na\"ive Bayes'').
In this section, we show how taking these limits we can recover known results about the high-dimensional asymptotics of linear discriminant analysis and na\"ive Bayes. Further, we compare these two methods over certain parameter classes. 

Note that $\lambda \to \infty$ leads to a linear discriminant rule with weight vector $\smash{\hdelta  = \hmu_{+1}-  \hmu_{-1}}$. Usual independence rules take the form $\smash{\diag(\hSigma_c)^{-1}\hdelta}$. We will assume that all features are normalized to have equal variance, $\Sigma_{ii} = \sigma>0$. In this case the $\lambda \to \infty$ rule corresponds to an independence rule with oracle information about the equality of variances; which we still call ``indpendence rule'' for simplicity.

Extending our previous notation, we define the asymptotic margin of LDA and independence rules, by taking the limits of $\Theta(\lambda)$ at 0 and $\infty$:  
\begin{equation*}
 \Theta_{\text{LDA}} = \lim_{\lambda \rightarrow 0} \frac{\alpha^2 \, \tau }{\sqrt{\alpha^2 \, \eta + \theta}} \ \eqand \ \Theta_{\text{IR}} = \lim_{\lambda \to \infty} \frac{\alpha^2 \, \tau }{\sqrt{\alpha^2 \, \eta + \theta}}.
\end{equation*}
Both limits are well-defined and admit simple expressions, as given below.

\begin{theorem}
\label{theo:lda_ir}
Let $H$ be the limit population spectral distribution of the covariance matrices $\Sigma$; and let $T$ be a random variable with distribution $H$. Under the conditions of Theorem \ref{theo:rda}, the margins of LDA and independence rules are equal to 
\begin{equation*}
 \Theta_{\text{LDA}} =\frac{\alpha^2 \, \sqrt{1-\gamma} \, \mathbb{E}_H\left[T^{-1}\right] }{\sqrt{\alpha^2 \, \mathbb{E}_H\left[T^{-1}\right] + \gamma}}  \ \eqand \ \Theta_{\text{IR}} = \frac{\alpha^2 }{\sqrt{\alpha^2 \, \mathbb{E}_H\left[T\right] + \gamma \mathbb{E}_H\left[T^2\right]}}.
\end{equation*}
The formula for LDA is valid for $\gamma<1$ while that for IR is valid for any $\gamma$.
\end{theorem}

This result is proved in the supplement.  The formulas are simpler than Theorem \ref{theo:rda}, as they involve the population spectral distribution $H$ directly through its moments. For RDA, the error rate depends on $H$ implicitly through the Stieltjes transform of the ESD $F$. 

These formulas are equivalent to known results, some of which were obtained under slightly different parametrization. In particular, \cite{serdobolskii2007multiparametric} attributes the IR formula with $H = \delta_1$ to unpublished work by Kolmogorov in 1967, while the \cite{raudys2004results} attributes it to \cite{raudys1967determining}. The LDA formula was derived by \citet{deev1970representation} and \citet{raudys1972amount}; see Section \ref{sec:rda_review}. Here, our goal was to show how these simple formulas can be recovered from the more powerful Theorem \ref{theo:rda}. 

\citet{saranadasa1993asymptotic} also obtains closed form expressions for the limit risk of two classification methods, the D-criterion and the A-criterion. One can verify that these are asymptotically equivalent to LDA and IR, respectively. Our results are consistent with those of \citet{saranadasa1993asymptotic};  but they differ slightly in the modeling assumptions.  In our notation, his results (as stated in his Theorem 3.2 and Corollary 3.1) are:
$
\Theta^S_{\text{LDA}} = \alpha \sqrt{\mathbb{E}_H\left[T^{-1}\right]} \sqrt{1-\gamma}  \ \eqand \ \Theta^S_{\text{IR}} =  \alpha/\sqrt{\mathbb{E}_H\left[T\right]}.
$
These results are nearly identical to Theorem \ref{theo:lda_ir}, but our equations have an extra term involving $\gamma$ in the denominator: $\gamma$ for LDA and $\gamma\mathbb{E}_H[T^2]$ for IR. The reason is that we consider $\mu_{\pm 1}$ as random, whereas \citet{saranadasa1993asymptotic} considers them as fixed sequences of vectors; this extra randomness yields additional variance terms.

Theorem \ref{theo:lda_ir} enables us to compare the worst-case performance of LDA and IR over suitable parameter classes of limit spectra. For $0 < k_1 \le 1 \le  k_2$ we define the class
\begin{equation*}
\mathcal{H}(k_1,k_2) = \left\{H: \mathbb{P}_H([k_1,k_2]) = 1, \, \mathbb{E}_H[T]=1 \right\}.
\end{equation*}
The bounds $k_1,k_2$ control the ill-conditioning of the population covariance matrix. We normalize such that the average population eigenvalue is 1, to ensure that the scaling of the problem does not affect the answer. This parameter space is somewhat similar to the one considered by \cite{bickel2004some}. Perhaps surprisingly, however, a direct comparison over these natural problem classes appears to be missing from the literature, and so we provide it below (see the supplement for a proof).

\begin{corollary}
Under the conditions of Theorem \ref{theo:lda_ir}, consider the behavior of LDA and independence rules for $H\in \mathcal{H}(k_1,k_2)$.

\begin{enumerate}
\item  The worst-case margin of LDA is: 
$$
 \bar\Theta_{\text{LDA}}(\gamma; \alpha^2):=\inf_{H\in \mathcal{H}(k_1,k_2)} \Theta_{\text{LDA}}(H,\gamma; \alpha^2)  = \frac{\alpha^2 \, \sqrt{1-\gamma}}{\sqrt{\alpha^2 + \gamma}}. 
$$
The least favorable distribution for LDA from the class $\mathcal{H}(k_1,k_2)$ is the point mass at 1: $H = \delta_1$, i.e., $\Sigma = I$.

\item The worst-case margin for independence rules is: 
$$
\bar\Theta_{\text{IR}}(\mathcal{H},\gamma; \alpha^2):=\inf_{H\in \mathcal{H}(k_1,k_2)} \Theta_{\text{IR}}(H,\gamma; \alpha^2)  = \frac{\alpha^2 }{\sqrt{\alpha^2 + \gamma (k_1 + k_2 - k_1k_2)}}.
$$
If $k_1<k_2$ the least favorable distribution is the mixture $H = w_1 \delta_{k_1} + w_2 \delta_{k_2}$, where the weights are $w_1 = (k_2-1)/(k_2-k_1)$ and $w_2 = (1-k_1)/(k_2-k_1)$; while if $k_1=k_2=1$, it is the point mass at 1: $H = \delta_1$. 
\end{enumerate}
\label{cor:lda_ir_minimax}
\end{corollary}

This result shows a stark contrast between the worst-case behavior of LDA and independence rules: for fixed signal strength, the worst-case risk of LDA over $\mathcal{H}$ only depends on $\gamma$, and is attained with the limit of identity covariances $\Sigma = I_{p\times p}$ regardless of the values of $k_1, \, k_2$. In contrast, the worst-case behavior of IR  occurs for a least favorable distribution $H$ that is as highly spread as possible. This highlights the sensitivity of IR to ill-conditioned covariance matrices. 
For $0 < \gamma < 1$, we see that IR are better than LDA in the worst case over $\mathcal{H}$, i.e. $ \bar\Theta_{\text{LDA}}(\gamma; \alpha^2)< \bar\Theta_{\text{IR}}(\mathcal{H},\gamma; \alpha^2)$, if and only if
\begin{equation*}
\alpha^2 + 1 > (1-\gamma)(k_1 + k_2 - k_1k_2).
\label{lda_ir_worstcase_comp}
\end{equation*}
In particular IR performs better than LDA for strong signals $\alpha$; with weaker signals, LDA can sometimes have an edge, particularly if the covariance is poorly conditioned, quantified by a large measure of spread $k_1+k_2-k_1k_2 = (k_2-1)(1-k_1)+1$. 

\subsection{Literature Review for High-Dimensional RDA}
\label{sec:rda_review}

There has been substantial work in the former Soviet Union on high-dimensional classification; references on this work include \citet{raudys2004results}, \citet{raudys2012statistical}, and \citet{serdobolskii2007multiparametric}.
 \citet{raudys1967determining}\footnote{\citet{raudys1967determining} is in Russian; see \citet{raudys2004results} for a review.} derived the $n, \, p \to \infty$ asymptotic error rate of independence rules in identity-covariance case $\Sigma = I_{p\times p}$, while \citet{deev1970representation} and \citet{raudys1972amount} obtained the error rate of un-regularized linear discriminant analysis (LDA) for general covariance $\Sigma$, again in the $n, \, p \to \infty$ regime.\footnote{\citet{serdobolskii2007multiparametric} attributes some early results on independence rules with identity covariance to Kolmogorov in 1967, and calls the framework $n,\, p \to \infty$, $p/n \to \gamma$ the Kolmorogov asymptotic regime. However, Kolmogorov apparently never published on the topic.} As shown in the previous section, these results can be obtained as special cases of our more general formulas.

For RDA, \citet{serdobolskii1983minimum} (see also Chapter 5 of \citet{serdobolskii2007multiparametric}) considered a more general setting than this paper: classification with a weight vector of the form $\Gamma(\hSigma_c)^{-1}\hdelta$ instead of just $(\hSigma_c + \lambda I_{p \times p})^{-1}\hdelta$, where the scalar function $\Gamma$ admits the integral representation $\Gamma(x) = \int (x +t)^{-1}\,d\eta(t)$ for a suitable measure $\eta$, and is extended to matrices in the usual way. He derived a limiting formula for the error rate of this classifier under high-dimensional asymptotics. However, due to their generality, his results are substantially more involved and much less explicit  than ours. In some cases it is unclear to us how one could numerically compute his formulas in practice. Furthermore, his results are proved when $\gamma <1$, and show convergence in probability, not almost surely. 
We also note the work of \citet{raudys1995small}, who derived results about the risk of usual RDA with vanishingly small regularization $\lambda = o(1)$, and for the special case $\gamma <1$.

In another line of work, a Japanese school \citep[e.g.,][and references therein]{fujikoshi2011multivariate} has studied the error rates of LDA and RDA under high-dimensional asymptotics, with a focus on obtaining accurate higher-order expansions to their risk. For instance  \citet{fujikoshi1998asymptotic} obtained asymptotic expansions for the error rate of un-regularized LDA, which can be verified to be equivalent to our results in the $\lambda \rightarrow 0$ limit. More recently, \citet{kubokawa2013asymptotic} obtained a second-order expansion of the error rate of RDA with vanishingly small regularization parameter $\lambda = O(1/n)$ in the case $\gamma <1$.

Finally, in the signal processing and pattern recognition literature, \citet{zollanvari2011analytic} provided asymptotic moments of estimators of the error rate of LDA, under an asymptotic framework where $n,p \to \infty$; however, this paper assumes that the covariance matrix $\Sigma$ is {known}.
More recently, \citet{zollanvari2015generalized} provide consistent estimators for the error rate of RDA in a doubly asymptotic framework, using deterministic equivalents for random matrices.
The goal of our work is rather different from theirs, in that we do not seek empirical estimators of the error rate of RDA, but instead seek simple formulas that help us understand the behavior of RDA.

\section*{Acknowledgment}

We are grateful to everyone who has provided comments on this manuscript, in particular David Donoho, Jerry Friedman, Iain Johnstone, Percy Liang, Asaf Weinstein, and Charles Zheng. E. D. is supported in part by NSF grant DMS-1418362. 

{\small
\setlength{\bibsep}{0.2pt plus 0.3ex}
\bibliographystyle{plainnat-abbrev}
\bibliography{references}

\begin{thebibliography}{59}
\providecommand{\natexlab}[1]{#1}
\providecommand{\url}[1]{\texttt{#1}}
\expandafter\ifx\csname urlstyle\endcsname\relax
  \providecommand{\doi}[1]{doi: #1}\else
  \providecommand{\doi}{doi: \begingroup \urlstyle{rm}\Url}\fi

\bibitem[Anderson(2003)]{anderson1958introduction}
T.~W. Anderson.
\newblock \emph{An Introduction to Multivariate Statistical Analysis}.
\newblock Wiley New York, 2003.

\bibitem[Bai and Silverstein(2010)]{bai2010spectral}
Z.~Bai and J.~W. Silverstein.
\newblock \emph{Spectral Analysis of Large Dimensional Random Matrices}.
\newblock Springer, 2010.

\bibitem[Bartlett and Mendelson(2003)]{bartlett2003rademacher}
P.~L. Bartlett and S.~Mendelson.
\newblock Rademacher and {G}aussian complexities: Risk bounds and structural
  results.
\newblock \emph{J. Mach. Learn. Res.}, 3:\penalty0 463--482, 2003.

\bibitem[Bayati and Montanari(2012)]{bayati2012lasso}
M.~Bayati and A.~Montanari.
\newblock The {LASSO} risk for {G}aussian matrices.
\newblock \emph{IEEE Trans. Inform. Theory}, 58\penalty0 (4):\penalty0
  1997--2017, 2012.

\bibitem[Bean et~al.(2013)Bean, Bickel, El~Karoui, and Yu]{bean2013optimal}
D.~Bean, P.~J. Bickel, N.~El~Karoui, and B.~Yu.
\newblock Optimal {M}-estimation in high-dimensional regression.
\newblock \emph{Proc. Natl. Acad. Sci. USA}, 110\penalty0 (36):\penalty0
  14563--14568, 2013.

\bibitem[Bernau et~al.(2014)Bernau, Riester, Boulesteix, Parmigiani,
  Huttenhower, Waldron, and Trippa]{bernau2015cross-study}
C.~Bernau, M.~Riester, A.-L. Boulesteix, G.~Parmigiani, C.~Huttenhower,
  L.~Waldron, and L.~Trippa.
\newblock Cross-study validation for the assessment of prediction algorithms.
\newblock \emph{Bioinformatics}, 30\penalty0 (12):\penalty0 i105--i112, 2014.

\bibitem[Bickel and Levina(2004)]{bickel2004some}
P.~J. Bickel and E.~Levina.
\newblock Some theory for {F}isher's linear discriminant function, ``naive
  {B}ayes'', and some alternatives when there are many more variables than
  observations.
\newblock \emph{Bernoulli}, pages 989--1010, 2004.

\bibitem[Cand\`es and Tao(2007)]{candes2007dantzig}
E.~Cand\`es and T.~Tao.
\newblock The {D}antzig selector: Statistical estimation when $p$ is much
  larger than $n$.
\newblock \emph{Ann. Statist.}, 35\penalty0 (6):\penalty0 2313--2351, 2007.

\bibitem[Chen et~al.(2011)Chen, Paul, Prentice, and Wang]{chen2011regularized}
L.~S. Chen, D.~Paul, R.~L. Prentice, and P.~Wang.
\newblock A regularized {Hotelling's} {$T^2$} test for pathway analysis in
  proteomic studies.
\newblock \emph{J. Amer. Statist. Assoc.}, 106\penalty0 (496), 2011.

\bibitem[Couillet and Debbah(2011)]{couillet2011random}
R.~Couillet and M.~Debbah.
\newblock \emph{Random {M}atrix {M}ethods for {W}ireless {C}ommunications}.
\newblock Cambridge University Press, 2011.

\bibitem[Deev(1970)]{deev1970representation}
A.~Deev.
\newblock Representation of statistics of discriminant analysis and asymptotic
  expansion when space dimensions are comparable with sample size.
\newblock In \emph{Sov. Math. Dokl.}, volume~11, pages 1547--1550, 1970.

\bibitem[Dicker(2014)]{dicker2014ridge}
L.~Dicker.
\newblock Ridge regression and asymptotic minimax estimation over spheres of
  growing dimension.
\newblock \emph{Bernoulli, to appear}, 2014.

\bibitem[Dobriban(2015)]{dobriban2015efficient}
E.~Dobriban.
\newblock Efficient computation of limit spectra of sample covariance matrices.
\newblock \emph{arXiv preprint arXiv:1507.01649}, 2015.

\bibitem[Donoho and Montanari(2015)]{donoho2015variance}
D.~L. Donoho and A.~Montanari.
\newblock Variance breakdown of {Huber} ({M})-estimators: $ n/p \rightarrow m
  \in (1, \, \infty) $.
\newblock \emph{arXiv preprint arXiv:1503.02106}, 2015.

\bibitem[Efron(1975)]{efron1975efficiency}
B.~Efron.
\newblock The efficiency of logistic regression compared to normal discriminant
  analysis.
\newblock \emph{J. Amer. Statist. Assoc.}, 70\penalty0 (352):\penalty0
  892--898, 1975.

\bibitem[El~Karoui(2013)]{karoui2013asymptotic}
N.~El~Karoui.
\newblock Asymptotic behavior of unregularized and ridge-regularized
  high-dimensional robust regression estimators: rigorous results.
\newblock \emph{arXiv preprint arXiv:1311.2445}, 2013.

\bibitem[El~Karoui and K{\"o}sters(2011)]{karoui2011geometric}
N.~El~Karoui and H.~K{\"o}sters.
\newblock Geometric sensitivity of random matrix results: consequences for
  shrinkage estimators of covariance and related statistical methods.
\newblock \emph{arXiv preprint arXiv:1105.1404}, 2011.

\bibitem[Fan et~al.(2011)Fan, Fan, and Wu]{fan2011high}
J.~Fan, Y.~Fan, and Y.~Wu.
\newblock High dimensional classification.
\newblock In T.~Cai and X.~Shen, editors, \emph{High-dimensional Data
  Analysis}, pages 3--37. World Scientific, New Jersey, 2011.

\bibitem[Friedman(1989)]{friedman1989regularized}
J.~H. Friedman.
\newblock Regularized discriminant analysis.
\newblock \emph{J. Amer. Statist. Assoc.}, 84\penalty0 (405):\penalty0
  165--175, 1989.

\bibitem[Fujikoshi and Seo(1998)]{fujikoshi1998asymptotic}
Y.~Fujikoshi and T.~Seo.
\newblock Asymptotic aproximations for {EPMC}s of the linear and the quadratic
  discriminant functions when the sample sizes and the dimension are large.
\newblock \emph{Random Oper. Stoch. Equ.}, 6\penalty0 (3):\penalty0 269--280,
  1998.

\bibitem[Fujikoshi et~al.(2011)Fujikoshi, Ulyanov, and
  Shimizu]{fujikoshi2011multivariate}
Y.~Fujikoshi, V.~V. Ulyanov, and R.~Shimizu.
\newblock \emph{Multivariate Statistics: High-dimensional and Large-sample
  Approximations}.
\newblock John Wiley \& Sons, 2011.

\bibitem[Grenander and Szeg\H{o}(1984)]{grenander1984toeplitz}
U.~Grenander and G.~Szeg\H{o}.
\newblock Toeplitz forms and their applications, 1984.

\bibitem[Hachem et~al.(2007)Hachem, Loubaton, and
  Najim]{hachem2007deterministic}
W.~Hachem, P.~Loubaton, and J.~Najim.
\newblock Deterministic equivalents for certain functionals of large random
  matrices.
\newblock \emph{The Annals of Applied Probability}, 17\penalty0 (3):\penalty0
  875--930, 2007.

\bibitem[Hachem et~al.(2008)Hachem, Khorunzhiy, Loubaton, Najim, and
  Pastur]{hachem2008new}
W.~Hachem, O.~Khorunzhiy, P.~Loubaton, J.~Najim, and L.~Pastur.
\newblock A new approach for mutual information analysis of large dimensional
  multi-antenna channels.
\newblock \emph{IEEE Trans. Inform. Theory}, 54\penalty0 (9):\penalty0
  3987--4004, 2008.

\bibitem[Hastie et~al.(2015)Hastie, Tibshirani, and
  Wainwright]{hastie2015statistical}
T.~Hastie, R.~Tibshirani, and M.~Wainwright.
\newblock \emph{Statistical Learning with Sparsity: The Lasso and
  Generalizations}.
\newblock CRC Press, 2015.

\bibitem[Hoerl and Kennard(1970)]{hoerl1970ridge}
A.~E. Hoerl and R.~W. Kennard.
\newblock Ridge regression: Biased estimation for nonorthogonal problems.
\newblock \emph{Technometrics}, 12\penalty0 (1):\penalty0 55--67, 1970.

\bibitem[Hsu et~al.(2014)Hsu, Kakade, and Zhang]{hsu2014random}
D.~Hsu, S.~M. Kakade, and T.~Zhang.
\newblock Random design analysis of ridge regression.
\newblock \emph{Found. Comput. Math.}, 14\penalty0 (3):\penalty0 569--600,
  2014.

\bibitem[Kleinberg et~al.(2015)Kleinberg, Ludwig, Mullainathan, Obermeyer,
  et~al.]{kleinberg2015prediction}
J.~Kleinberg, J.~Ludwig, S.~Mullainathan, Z.~Obermeyer, et~al.
\newblock Prediction policy problems.
\newblock \emph{American Economic Review}, 105\penalty0 (5):\penalty0 491--95,
  2015.

\bibitem[Kubokawa et~al.(2013)Kubokawa, Hyodo, and
  Srivastava]{kubokawa2013asymptotic}
T.~Kubokawa, M.~Hyodo, and M.~S. Srivastava.
\newblock Asymptotic expansion and estimation of {EPMC} for linear
  classification rules in high dimension.
\newblock \emph{J. Multivariate Anal.}, 115:\penalty0 496--515, 2013.

\bibitem[Ledoit and P{\'e}ch{\'e}(2011)]{ledoit2011eigenvectors}
O.~Ledoit and S.~P{\'e}ch{\'e}.
\newblock Eigenvectors of some large sample covariance matrix ensembles.
\newblock \emph{Probab. Theory Related Fields}, 151\penalty0 (1-2):\penalty0
  233--264, 2011.

\bibitem[Liang and Srebro(2010)]{liang2010interaction}
P.~Liang and N.~Srebro.
\newblock On the interaction between norm and dimensionality: Multiple regimes
  in learning.
\newblock In \emph{ICML}, 2010.

\bibitem[Marchenko and Pastur(1967)]{marchenko1967distribution}
V.~A. Marchenko and L.~A. Pastur.
\newblock Distribution of eigenvalues for some sets of random matrices.
\newblock \emph{Mat. Sb.}, 114\penalty0 (4):\penalty0 507--536, 1967.

\bibitem[Ng and Jordan(2001)]{ng2001discriminative}
A.~Ng and M.~Jordan.
\newblock On discriminative vs. generative classifiers: A comparison of
  logistic regression and naive {B}ayes.
\newblock In \emph{NIPS}, 2001.

\bibitem[Pickrell and Pritchard(2012)]{pickrell2012inference}
J.~K. Pickrell and J.~K. Pritchard.
\newblock Inference of population splits and mixtures from genome-wide allele
  frequency data.
\newblock \emph{PLoS genetics}, 8\penalty0 (11):\penalty0 e1002967, 2012.

\bibitem[Raudys(1967)]{raudys1967determining}
{\v{S}}.~Raudys.
\newblock On determining training sample size of linear classifier.
\newblock \emph{Comput. Systems (in Russian)}, 28:\penalty0 79–87, 1967.

\bibitem[Raudys(1972)]{raudys1972amount}
{\v{S}}.~Raudys.
\newblock On the amount of a priori information in designing the classification
  algorithm.
\newblock \emph{Technical Cybernetics (in Russian)}, 4:\penalty0 168--174,
  1972.

\bibitem[Raudys(2012)]{raudys2012statistical}
{\v{S}}.~Raudys.
\newblock \emph{Statistical and Neural Classifiers: An integrated approach to
  design}.
\newblock Springer Science \& Business Media, 2012.

\bibitem[Raudys and Skurichina(1995)]{raudys1995small}
{\v{S}}.~Raudys and M.~Skurichina.
\newblock Small sample properties of ridge estimate of the covariance matrix in
  statistical and neural net classification.
\newblock \emph{New Trends in Probability and Statistics}, 3:\penalty0
  237--245, 1995.

\bibitem[Raudys and Young(2004)]{raudys2004results}
{\v{S}}.~Raudys and D.~M. Young.
\newblock Results in statistical discriminant analysis: A review of the former
  {S}oviet {U}nion literature.
\newblock \emph{J. Multivariate Anal.}, 89\penalty0 (1):\penalty0 1--35, 2004.

\bibitem[Rifai et~al.(2011)Rifai, Dauphin, Vincent, Bengio, and
  Muller]{rifai2011manifold}
S.~Rifai, Y.~Dauphin, P.~Vincent, Y.~Bengio, and X.~Muller.
\newblock The manifold tangent classifier.
\newblock \emph{Advances in Neural Information Processing Systems},
  24:\penalty0 2294--2302, 2011.

\bibitem[Rubio et~al.(2012)Rubio, Mestre, and Palomar]{rubio2012performance}
F.~Rubio, X.~Mestre, and D.~P. Palomar.
\newblock Performance analysis and optimal selection of large minimum variance
  portfolios under estimation risk.
\newblock \emph{IEEE Journal of Selected Topics in Signal Processing},
  6\penalty0 (4):\penalty0 337--350, 2012.

\bibitem[Russakovsky et~al.(2014)Russakovsky, Deng, Su, Krause, Satheesh, Ma,
  Huang, Karpathy, Khosla, Bernstein, et~al.]{russakovsky2014imagenet}
O.~Russakovsky, J.~Deng, H.~Su, J.~Krause, S.~Satheesh, S.~Ma, Z.~Huang,
  A.~Karpathy, A.~Khosla, M.~Bernstein, et~al.
\newblock Image{N}et large scale visual recognition challenge.
\newblock \emph{International Journal of Computer Vision}, pages 1--42, 2014.

\bibitem[Saranadasa(1993)]{saranadasa1993asymptotic}
H.~Saranadasa.
\newblock Asymptotic expansion of the misclassification probabilities of
  {D}-and {A}-criteria for discrimination from two high dimensional populations
  using the theory of large dimensional random matrices.
\newblock \emph{J. Multivariate Anal.}, 46\penalty0 (1):\penalty0 154--174,
  1993.

\bibitem[Serdobolskii(1983)]{serdobolskii1983minimum}
V.~I. Serdobolskii.
\newblock On minimum error probability in discriminant analysis.
\newblock In \emph{Dokl. Akad. Nauk SSSR}, volume~27, pages 720--725, 1983.

\bibitem[Serdobolskii(2007)]{serdobolskii2007multiparametric}
V.~I. Serdobolskii.
\newblock \emph{Multiparametric {S}tatistics}.
\newblock Elsevier, 2007.

\bibitem[Silverstein(1995)]{silverstein1995strong}
J.~W. Silverstein.
\newblock Strong convergence of the empirical distribution of eigenvalues of
  large dimensional random matrices.
\newblock \emph{J. Multivariate Anal.}, 55\penalty0 (2):\penalty0 331--339,
  1995.

\bibitem[Silverstein and Choi(1995)]{silverstein1995analysis}
J.~W. Silverstein and S.-I. Choi.
\newblock Analysis of the limiting spectral distribution of large dimensional
  random matrices.
\newblock \emph{J. Multivariate Anal.}, 54\penalty0 (2):\penalty0 295--309,
  1995.

\bibitem[Silverstein and Combettes(1992)]{silverstein1992signal}
J.~W. Silverstein and P.~L. Combettes.
\newblock Signal detection via spectral theory of large dimensional random
  matrices.
\newblock \emph{IEEE Trans. Signal Process.}, 40\penalty0 (8):\penalty0
  2100--2105, 1992.

\bibitem[Simard et~al.(2000)Simard, Le~Cun, Denker, and
  Victorri]{simard2000transformation}
P.~Y. Simard, Y.~A. Le~Cun, J.~S. Denker, and B.~Victorri.
\newblock Transformation invariance in pattern recognition: Tangent distance
  and propagation.
\newblock \emph{International Journal of Imaging Systems and Technology},
  11\penalty0 (3):\penalty0 181--197, 2000.

\bibitem[Sutton and McCallum(2006)]{sutton2006introduction}
C.~Sutton and A.~McCallum.
\newblock An introduction to conditional random fields for relational learning.
\newblock \emph{Introduction to statistical relational learning}, pages
  93--128, 2006.

\bibitem[Toutanova et~al.(2003)Toutanova, Klein, Manning, and
  Singer]{toutanova2003feature}
K.~Toutanova, D.~Klein, C.~D. Manning, and Y.~Singer.
\newblock Feature-rich part-of-speech tagging with a cyclic dependency network.
\newblock In \emph{NAACL}, 2003.

\bibitem[Tulino and Verd{\'u}(2004)]{tulino2004random}
A.~M. Tulino and S.~Verd{\'u}.
\newblock Random matrix theory and wireless communications.
\newblock \emph{Communications and Information theory}, 1\penalty0
  (1):\penalty0 1--182, 2004.

\bibitem[Vapnik and Chervonenkis(1971)]{vapnik1971uniform}
V.~N. Vapnik and A.~Y. Chervonenkis.
\newblock On the uniform convergence of relative frequencies of events to their
  probabilities.
\newblock \emph{Theory Probab. Appl.}, 16\penalty0 (2):\penalty0 264--280,
  1971.

\bibitem[Wang and Manning(2012)]{wang2012baselines}
S.~Wang and C.~D. Manning.
\newblock Baselines and bigrams: Simple, good sentiment and topic
  classification.
\newblock In \emph{Proceedings of the 50th Annual Meeting of the Association
  for Computational Linguistics: Short Papers-Volume 2}, pages 90--94.
  Association for Computational Linguistics, 2012.

\bibitem[Wray et~al.(2007)Wray, Goddard, and Visscher]{wray2007prediction}
N.~R. Wray, M.~E. Goddard, and P.~M. Visscher.
\newblock Prediction of individual genetic risk to disease from genome-wide
  association studies.
\newblock \emph{Genome research}, 17\penalty0 (10):\penalty0 1520--1528, 2007.

\bibitem[Yao et~al.(2015)Yao, Bai, and Zheng]{yao2015large}
J.~Yao, Z.~Bai, and S.~Zheng.
\newblock \emph{Large Sample Covariance Matrices and High-Dimensional Data
  Analysis}.
\newblock Cambridge University Press, 2015.

\bibitem[Zhang et~al.(2013)Zhang, Rubio, Palomar, and Mestre]{zhang2013finite}
M.~Zhang, F.~Rubio, D.~P. Palomar, and X.~Mestre.
\newblock Finite-sample linear filter optimization in wireless communications
  and financial systems.
\newblock \emph{IEEE Trans. Signal Process.}, 61\penalty0 (20):\penalty0
  5014--5025, 2013.

\bibitem[Zollanvari and Dougherty(2015)]{zollanvari2015generalized}
A.~Zollanvari and E.~R. Dougherty.
\newblock Generalized consistent error estimator of linear discriminant
  analysis.
\newblock \emph{IEEE Trans. Signal Process.}, 63\penalty0 (11), 2015.

\bibitem[Zollanvari et~al.(2011)Zollanvari, Braga-Neto, and
  Dougherty]{zollanvari2011analytic}
A.~Zollanvari, U.~M. Braga-Neto, and E.~R. Dougherty.
\newblock Analytic study of performance of error estimators for linear
  discriminant analysis.
\newblock \emph{IEEE Trans. Signal Process.}, 59\penalty0 (9):\penalty0
  4238--4255, 2011.

\end{thebibliography}
}

\section{Supplement}

The supplement is organized as follows: Section \ref{sec:comput} describes the efficient computation of the risk formulas, Section \ref{proofs_ridge} has the proofs for ridge regression, and Section \ref{proofs_rda} has the proofs for regularized discriminant analysis. At several locations we refer to equation numbers from the main text.

\section{Efficient computation of the risk formulas}
\label{sec:comput}

Consider the spectral distribution of the companion matrix $\smash{\underline \hSigma = n^{-1} X X^\top}$. Since its spectral distribution $\smash{F_{\underline \hSigma}}$ differs from $\smash{F_{\hSigma}}$ by $|n-p|$ zeros, it follows that $\smash{\underline \hSigma}$ has a limit ESD $\smash{\underline F}$, given by $\smash{\underline{F} = \gamma F + (1-\gamma) I_{[0,\infty)}}$.
The companion Stieltjes transform satisfies the Silverstein equation \citep{silverstein1992signal, silverstein1995analysis}:
\begin{equation}
\label{silv.eq}
-\frac{1}{v(z)} = z - \gamma \int \frac{t \, dH(t)}{1 + tv(z)}.
\end{equation}
It is known that for $z \in \mathcal {S} : = \{u+ iv: v\neq0 \mbox{, or } v=0, u >0\}$, $v(z)$ is the unique solution of the Silverstein equation with $v(z) \in \mathcal {S}$ such that $\sign(\mathrm{Imag}\{v(z)\})=\sign(\mathrm{Imag}(z))$. 

We now explain how to compute the key quantities $m,v,m',v'$ that will come up in our risk formulas. On the interval $v \in [0,\infty)$, \citet{silverstein1995analysis} prove that the functional inverse of $z \to v:=v(z)$ has the explicit form: 
\begin{equation}
\label{ST_inverse}
z(v) = -\frac{1}{v}  +\gamma \int \frac{t \, dH(t)}{1 + tv}.
\end{equation}
This result enables the efficient computation of the function $z\to v(z)$ for $z<0$. Indeed, assuming one can compute the corresponding integral against $H$, one can tabulate \eqref{ST_inverse} on a dense grid of $v_i >0$, to find pairs $(v_i,z_i)$, where $z_i = z(v_i)$. Then for the values $z_i<0$, the \citet{silverstein1995analysis} result shows that $v(z_i)=v_i$. 
Further, the Silverstein equation can be differentiated with respect to $z$ to obtain an explicit formula for $v'$ in terms of $v,H$:
\begin{equation*}
\frac{dv}{dz} = \p{\frac{1}{v^2} -\gamma\int \frac{t^2\,dH(t)}{(1 + tv)^2}}^{-1}
\end{equation*}
Therefore, once $v(z)$ is computed for a value $z$, the computation of $v'(z)$ can be done conveniently in terms of $v(z)$ and $H$, assuming again that the integral involving $H$ can be computed. This is one of the main steps in the \textsc{Spectrode} method for computing the limit ESD  \citep{dobriban2015efficient}. Finally, $m(z)$ can be computed from $v(z)$ via the equation \eqref{dual.ST}, and $m'(z)$ can be computed from $v'(z)$ by differentiating  \eqref{dual.ST}: 
$\gamma\left(m'(z)-1/z^2\right) = v'(z)-1/z^2$.

\section{Proofs for Ridge Regression}
\label{proofs_ridge}

\subsection{Proof of Theorem \ref{theo:ridge}}
\label{ridge_proof}

The risk $r_{\lambda}(X)$ equals
\begin{align*}
r_{\lambda}(X)  & = \EE{ (x \cdot \hw_{\lambda} -  x \cdot w - \varepsilon_0)^2 \cond X} \\
& =  1+ \EE{ \{x \cdot \p{\hw_{\lambda} - w}\}^2 \cond X}
=  1+ \EE{\p{\hw_{\lambda} - w}^\top \Sigma \p{\hw_{\lambda} - w} \cond X} \end{align*}

where $\varepsilon_0$ is the noise in the new observation.  Now $\hw_\lambda = (X^\top X + \lambda \, n \, I_{p \times p})^{-1}$ $ X^\top Y$, and $Y = Xw + \varepsilon$, where $\varepsilon$ is the vector of noise terms in the original data. Hence 
\begin{align*}
\hw_{\lambda} - w & =  (X^\top X + \lambda \, n \, I_{p \times p})^{-1} X^\top ( X w + \varepsilon) - w \\
& = -  \lambda \, n (X^\top X + \lambda \, n \, I_{p \times p})^{-1} w + (X^\top X + \lambda \, n \, I_{p \times p})^{-1} X^\top \varepsilon.
\end{align*}

When we plug this back into the risk formula $r_{\lambda}(X)$, and use that $w, \varepsilon$ are conditionally independent given $X$, we see that the cross-term involving $w, \varepsilon$ cancels. The risk simplifies to 
\begin{align*}
r_{\lambda}(X)  & = 1+ (\lambda \, n)^2 \, \EE{ w^\top (X^\top X + \lambda \, n \, I_{p \times p})^{-1}   \Sigma  (X^\top X + \lambda \, n \, I_{p \times p})^{-1} w \cond X}    \\
&\ \ \ \ \ \ \ \ + \EE{ \varepsilon^\top X (X^\top X + \lambda \, n \, I_{p \times p})^{-1}   \Sigma  (X^\top X + \lambda \, n \, I_{p \times p})^{-1} X^\top \varepsilon \cond X}.
\end{align*}

Now using that the components of $w$ and $\varepsilon$ are each uncorrelated conditional on $X$, we obtain the further simplification
\begin{align*}
r_{\lambda}(X) &= 1+  (\lambda \, n)^2 \, \frac{\alpha^2}{p}  \, \tr\p{\Sigma\p{X^\top X + \lambda \, n \, I_{p \times p}}^{-2}} \\
&\ \ \ \ \ \ \ \  + \tr\p{\Sigma\p{X^\top X + \lambda \, n \, I_{p \times p}}^{-1} X^\top X \p{X^\top X + \lambda \, n \, I_{p \times p}}^{-1}}.
\end{align*}

Introducing $ \hSigma = n^{-1} X^\top X $ and $\gamma_p = p/n$, and splitting the last term in two by using $X^\top X = X^\top X + \lambda \, n \, I_{p \times p}- \lambda \, n \, I_{p \times p}$ this yields
\begin{align}
r_{\lambda}(X) &=  1 + \frac{\gamma_p}{p} \tr \p{ \Sigma \p{\hSigma +  \lambda I_{p \times p} }^{-1}} +  (\lambda \alpha^2 - \gamma_p)  \frac{\lambda }{p} \tr \p{ \Sigma \p{\hSigma +  \lambda I_{p \times p} }^{-2}}.
\label{general_lambda_risk}
\end{align}

For the particular choice $\lambda^* = \gamma_p\alpha^{-2}$, we obtain the claimed formula $r_{\lambda*}(X)$.  Next, we show the convergence of $r_{\lambda}(X)$ for arbitrary fixed $\lambda$. First, by assumption we have $\gamma_p \to \gamma$. Therefore, it is enough to show the almost sure convergence of the two functionals:
\begin{align*}
\frac{1}{p} \tr \p{ \Sigma \p{\hSigma +  \lambda I_{p \times p} }^{-1}} \text{ and  }  \frac{1}{p} \tr \p{ \Sigma \p{\hSigma +  \lambda  I_{p \times p} }^{-2}}.
\end{align*}

The convergence of the first one follows directly from the theorem of \citet{ledoit2011eigenvectors}, given in \eqref{eq:ledoit_peche}. The second is shown later in the proof of Lemma \ref{term_A} in Section \ref{pf:term_A}. In that section it is assumed that the eigenvalues of $\Sigma$ are bounded away from 0 an infinity; but one can check that in the proof of Lemma \ref{term_A} only the upper bound is used, and that holds in our case.  Therefore the risk $r_{\lambda}(X)$ converges almost surely for each $\lambda$. 

Next we find the fomulas for the limits of the two functionals. The limit of $p^{-1} \tr \p{ \Sigma \p{\hSigma +  \lambda I_{p \times p} }^{-1}}$ equals $\kappa(\lambda) = \gamma^{-1}(1/[\lambda v(-\lambda)]-1)$ by \eqref{eq:ledoit_peche}. In the proof of Lemma \ref{term_A} in Section \ref{pf:term_A}, it is shown that the limit of $p^{-1} \tr \p{ \Sigma \p{\hSigma +  \lambda I_{p \times p} }^{-2}}$ is $-\kappa'(\lambda) = [v(-\lambda)-\lambda v'(-\lambda)]/[\gamma(\lambda v(-\lambda))^2]$.

{\bf Simplified expression for $R_\lambda$:} Putting together the results above, we obtain (with $v = v(-\lambda), v' = v'(-\lambda)$) the desired claim:
\begin{align*}
R_\lambda &= 1 + \gamma \kappa(-\lambda) +   (\lambda \alpha^2 - \gamma)  \lambda [-\kappa'(-\lambda)] \\
& = \frac{1}{\lambda v} + (\lambda \alpha^2 - \gamma)  \lambda \frac{ v-\lambda v'}{\gamma(\lambda v)^2} = \frac{1}{\lambda v} \left\{ 1 + \left(\frac{\lambda \alpha^2}{\gamma} -1\right)\left(1-\frac{\lambda v'}{v}\right) \right\}.
\end{align*}

{\bf Second part: Convergence: } First, we note that for $\lambda_p^*=\gamma_p\alpha^{-2}$, the finite sample risk equals by \eqref{general_lambda_risk}

\begin{equation*}
 r_{\lambda_p^*}(X) =  1 + \frac{\gamma_p}{p} \tr \p{ \Sigma \p{\hSigma + \frac{\gamma_p}{\alpha^2} I_{p \times p} }^{-1}}. 
\end{equation*}

Introduce the function $k_p(\lambda,X) = \frac{1}{p}\tr \p{ \Sigma \p{\hSigma + \lambda I_{p \times p} }^{-1}}$. We need to show $k_p(\lambda_p^*,X) \to \kappa(\lambda^*)$. First, we notice that  $\lambda_p^* \to \lambda^*$, and $k_p(\lambda,X) \to \kappa(\lambda)$ almost surely. Second, we verify the equicontinuity of $k_p$ as a function of $\lambda$, by proving the stronger claim that the derivatives of $k_p$ are uniformly bounded: 
$$|k_p'(\lambda,X)| = \left|\frac{1}{p} \tr \p{ \Sigma \p{\hSigma +  \lambda I_{p \times p} }^{-2}}\right|\le \left|\frac{\tr \Sigma}{p}\right|\le K.$$
The last inequality holds for some finite constant $K$ because $\mathbb{E}_{F_\Sigma}X = p^{-1}\tr \Sigma \rightarrow \mathbb{E}_{H}{X} <\infty$ by the convergence and boundedness of the spectral distribution of $\Sigma$.

Therefore, by the equicontinuity of the family $k_p(\lambda,X)$ as a function of $\lambda$, we obtain $k_p(\lambda_p^*,X) \to \kappa(\lambda^*)$. Further, the explicit form of $r_{\lambda_p^*}$ shows that the limit equals $1 + \gamma \kappa(\lambda^*)=1/(\lambda^*v(-\lambda^*))$ by \eqref{eq:ledoit_peche}, as desired.

{\bf Optimality of $\lambda^*$:} The limiting risk $R_\lambda$ is the same if we assume Gaussian observations. In this case, the finite sample Bayes-optimal choice for $\lambda_p$ is $\lambda_p^* = \gamma_p\alpha^{-2}$. We will use the following classical lemma to conclude that the limit of the minimizers $\lambda_p^*$ is the minimizer of the limit. 

\begin{lemma}
Let $f_n(x):\mathcal{I}\to\mathcal{I}$ be an equicontinuous family of functions on an interval $\mathcal{I}$, converging pointwise to a continuous function, $f_n(x) \to f(x)$. Suppose $y_n$ is a minimizer of $f_n$ on $\mathcal{I}$, and $y_n \to y$. Then $y$ is a minimizer of $f$.
\label{equicont_min}
\end{lemma}
\begin{proof}
Since $y_n$ is a minimizer of $f_n$, we have 
\begin{equation}
\label{minimizer}
f_n(y_n)\le f_n(x),
\end{equation}

for all $x \in \mathcal{I}$. But $f_n(y_n)=f_n(y_n)-f_n(y) + f_n(y)-f(y) + f(y)$. As $n\to \infty$, $f_n(y_n)-f_n(y)\to0$ by the convergence of $y_n \to y$ and by the equicontinuity of the family $\{f_n\}$; and  $f_n(x)-f(x)\to0$ for all $x$ by the convergence of $f_n\to f$. Therefore, taking the limit as $n\to \infty$ in \eqref{minimizer}, we obtain $f(y) \le f(x)$ for any $x \in \mathcal{I}$, showing that $y$ is a minimizer of $f$.
\end{proof}

We use the notation $r_{\lambda,p}(X)=r_\lambda(X)$ for the risk, showing that it depends on $p$. Fix an arbitrary sequence of $X$ matrices on the event having probability one where $r_{\lambda,p}(X) \to R_{\lambda}$. By an argument similar to the one given above, the sequence of functions $f_p(\lambda)=r_{\lambda,p}(X)$ equicontinuous in $\lambda$ on the set $\lambda\ge 0$. Since $f_p(\lambda)$ converges to $R_\lambda$ for each $\lambda>0$, and $\lambda_p^*$ is a sequence of minimizers of $f_p(\lambda)$ that converges to $\lambda^*$, by  Lemma \ref{equicont_min} $\lambda^*$ is a minimizer of $R_\lambda$.

\subsection{The Risk of Ridge Regression with Identity Covariance}
\label{pf:ridge_identity}

From Theorem \ref{theo:ridge} it follows that the risk equals the limit $r_\lambda(X)$ \eqref{general_lambda_risk}. Since $\Sigma = I$ this simplifies to 
\begin{align*}
r_{\lambda}(X) &=  1 + \frac{\gamma_p}{p} \tr \p{\p{\hSigma +  \lambda I_{p \times p} }^{-1}} +  (\lambda \alpha^2 - \gamma_p)  \frac{\lambda }{p} \tr \p{ \p{\hSigma +  \lambda I_{p \times p} }^{-2}}.
\end{align*}
By the Marchenko-Pastur theorem, Eq. \eqref{eq:mp_lemma}, it follows that $p^{-1} \tr \p{\p{\hSigma +  \lambda I_{p \times p} }^{-1}}$ $\to m_I(-\lambda; \gamma) $ is defined in \eqref{identity_stieltjes}.

In the proof of Lemma \ref{term_A} in Section \ref{pf:term_A}, it is shown that $p^{-1} \tr \p{ \p{\hSigma +  \lambda I_{p \times p} }^{-2}}$ $\to -\kappa'(\lambda)$. For an identity covariance matrix $\kappa(\lambda) = m_I(-\lambda; \gamma)$ by definition of $\kappa(\lambda)$. Therefore, the limit of the second term equals $m'_I(-\lambda;\gamma)$. We obtain the desired formula: $R_\lambda = 1 + \gamma m_I(-\lambda; \gamma) + \lambda \left( \lambda \alpha^2  - \gamma \right)  m'_I(-\lambda;\gamma)$.
For $\lambda^* = \gamma\alpha^{-2}$, we obtain $R^*=  1 + \gamma m_I(-\lambda^*;\gamma)$. It is a matter of simple algebra to verify the formula \eqref{eq:closed_form} for the risk. 

\subsection{Proof of strong-signal limit of ridge}
\label{pf:strong_signal}

We will first show the results for the strong-signal limit. We start by verifying the following lemma.

\begin{lemma}
Suppose the limit population eigenvalue distribution $H$ has support contained in a compact set bounded away from 0. Let $v(z)$ be the companion Stieltjes transform of the ESD. Then
\begin{enumerate}
\item If $\gamma<1$, $\lim_{\lambda\downarrow0}\lambda v(-\lambda) = 1-\gamma$.
\item If $\gamma>1$, $\lim_{\lambda\downarrow0} v(-\lambda) = v(0)$.
\item If $\gamma=1$, $\lim_{\lambda\downarrow0} \lambda v(-\lambda)^2 = \mathbb{E}_H[T^{-1}]$.
\end{enumerate}
\end{lemma}
\begin{proof} Let $\underline F$ be the ESD of the companion matrix $n^{-1}XX^\top$. It is related to $F$ via $\underline F = (1-\gamma)\delta_0 + \gamma F$.
It is well known \citep[e.g.,][Chapter 6]{bai2010spectral}, that for $H$ whose support is contained in a compact set bounded away from 0, the following hold for $F$ and $\underline F$: if $\gamma<1$, then $F$ has support contained in a compact set bounded away from 0, and $\underline F$ has a point mass of $1-\gamma$ at 0; while if $\gamma>1$, then $\underline F$ has support contained in a compact set bounded away from 0, and $F$ has a point mass of $1-\gamma^{-1}$ at 0.

If $\gamma<1$, we let $Y$ be distributed according to $F$. Since $Y>c>0$ for some $c$, we have by the dominated convergence theorem
$$\lim_{\lambda\downarrow0} \lambda m(-\lambda) = \lim_{\lambda\downarrow0} \mathbb{E}_F\left[\frac{\lambda}{\lambda+Y}\right] = 0.$$
Since $\lambda v(-\lambda) = 1 - \gamma + \gamma \lambda m(-\lambda)$, this shows $\lim_{\lambda\downarrow0}\lambda v(-\lambda) = 1-\gamma$. 

Similarly, if $\gamma>1$, we let $\underline Y$ be distributed according to $\underline F$. Since $\underline Y>c>0$ for some $c$, we have $\lim_{\lambda\downarrow0} \lambda v(-\lambda) = \lim_{\lambda\downarrow0} \mathbb{E}_{\underline F}\left[\frac{1}{\lambda+\underline Y}\right] = \mathbb{E}_{\underline F}\left[\frac{1}{\underline Y}\right] =v(0).$
This shows $\lim_{\lambda\downarrow0} v(-\lambda) = v(0)$.  Further, we can find the equation for $v(0)$ by taking the limit as $z \to 0$ in the Silverstein equation \eqref{silv.eq}. We see that $v$ is well-defined, bounded away from 0, and has positive imginary part for $z\in \mathbb{C}^+$ in a neighborhood of $0$, so the limit is justified by the dominated convergence theorem. This leads to the equation that was claimed: 
$$\frac{1}{v(0)} = \gamma \int \frac{t \, dH(t)}{1 + tv(0)}.$$

Finally, for $\gamma=1$, the Silverstein equation \eqref{silv.eq} is equivalent to
$$zv(z) = - \int \frac{ \, dH(t)}{1 + tv(z)}.$$
Therefore the real quantity $\lim_{\lambda \downarrow0}v(-\lambda) = \lim_{\lambda\downarrow0} \mathbb{E}_{\underline F}\left[\frac{1}{\lambda+\underline Y}\right]$ is not finite; otherwise, we could take the limit as $z\to 0$, $z \in \mathbb{C}^+$, similarly to above, and obtain the contradiction that $0 = \int (1 + tv(0))^{-1}\, dH(t) \neq 0$. Since $v(-\lambda)$ is increasing as $\lambda \downarrow 0$, it follows that $\lim_{\lambda \downarrow0}v(-\lambda) = +\infty$. Multiplying the Silverstein equation above with $v(z)$

$$zv(z)^2 = - \int \frac{ \, dH(t)}{1/v(z) + t}.$$

We can take the limit in this equation as $\lambda = -z \downarrow0$ similarly to above, and note that the right hand side converges to $-\int t^{-1}\, dH(t) = -\mathbb{E}_H[T^{-1}]$ by the dominated convergence theorem. This shows $\lim_{\lambda\downarrow0} \lambda v(-\lambda)^2 = \mathbb{E}_H[T^{-1}]$.
\end{proof}

Consequently, using the formula for the optimal risk from Theorem \ref{theo:ridge}, we have for $\gamma<1$, $\lim_{\alpha^2 \to \infty} R^*(H,\alpha^2,\gamma)   = (1-\gamma)^{-1}$. For $\gamma>1$, 
$$\lim_{\alpha^2 \to \infty} \alpha^{-2}R^*(H,\alpha^2,\gamma) =  \lim_{\alpha^2 \to \infty}\frac{1}{\frac{\gamma}{\alpha^2}v(-\frac{\gamma}{\alpha^2})\alpha^2} = \frac{1}{\gamma v(0)}.$$
Finally, for $\gamma=1$, 
$$\lim_{\alpha^2 \to \infty} \alpha^{-1}R^*(H,\alpha^2,\gamma) =  \lim_{\alpha^2 \to \infty}\frac{1}{\frac{1}{\alpha^2}v(-\frac{1}{\alpha^2})\alpha} = \lim_{\alpha^2 \to \infty}\frac{1}{\left(\frac{1}{\alpha^2}v(-\frac{1}{\alpha^2})^2\right)^{1/2}} = \frac{1}{{E}_H[T^{-1}]^{1/2}}.$$
The explicit formula for $R^*(\alpha^2,1)$ is obtained by plugging in the expression \eqref{identity_stieltjes} into the formula for the optimal risk.

Next we argue about the weak-signal limit. Using the above notations, $\lambda v(-\lambda) = \mathbb{E}_{\underline F}[\lambda/(\lambda + \underline Y)]$, so by the dominated convergence theorem, $\lim_{ \lambda \to \infty}$ $\lambda v(-\lambda) = 1$, and consequently $\lim_{\alpha^2 \to 0}R^*(H,\alpha^2,\gamma)  = 1$. Furthermore, $\lambda[1-\lambda v(-\lambda)] = \gamma  \mathbb{E}_{F}[Y \lambda/(\lambda + Y)]$, leading to $\lim_{ \lambda \to \infty}\lambda[1-\lambda v(-\lambda)] = \gamma\mathbb{E}_F[Y] = \gamma\mathbb{E}_H[T]$, and consequently $\smash{\lim_{\alpha^2 \rightarrow 0} (R^*(H,\alpha^2,\gamma) - 1)/\alpha^2 = \mathbb{E}_HT}$.

\section{Proofs for Regularized Discriminant Analysis}
\label{proofs_rda}

\subsection{Proof of Theorem \ref{theo:rda}}
\label{pf:theo:lda}

We will first outline the high-level steps to prove our main result for classification, Theorem \ref{theo:rda}. We break down the proof into several lemmas, whose proof is deferred to later sections. These lemmas are then put together to prove the theorem in the final part of the proof outline. 

We start with the well-known finite-sample formula for the expected test error of an arbitrary linear classifier $h_{w,b}(x) = \sign(w \cdot x + b)$ in the Gaussian model \eqref{eq:setup}, conditional on the weight parameters $w,b$ and the means $\mu_{\pm 1}$: 

\begin{equation}
\Err_0\p{w,b} =\pi_- \Phi\p{\frac{w^{\top} \mu_{-1} + b}{\sqrt{w^\top \Sigma w}  }} + \pi_+\Phi\p{-\frac{w^{\top} \mu_1+ b}{\sqrt{w^\top \Sigma w}  }}. 
\label{class_err}
\end{equation}


In RDA the weight vector is $\hw_{\lambda} =  \p{\hSigma_c + \lambda I_{p \times p}}^{-1}\hdelta \, $  and the offset is $\hat{b} = \hdelta^\top \p{\hSigma_c + \lambda I_{p \times p}}^{-1}\hmu$. 
The first simplification we notice is that $\hat{b} \to_{a.s.} 0$. 

\begin{alemma} Under the conditions of Theorem \ref{theo:rda}, we have $\hat{b} \to_{a.s.} 0$.
\label{offset_to_zero}
\end{alemma}

Lemma \ref{offset_to_zero} is proved in Section \ref{pf:offset_to_zero}. Since the denominator in the error rate \eqref{class_err} converges almost surely to a fixed, strictly positive constant (see Lemmas \ref{term_A} and \ref{term_C}), this will allow us to use the following simpler formula - that does not involve $\hat b$ - in evaluating the limit of the error rate. 

\begin{equation}
\Err_1\p{w} = \pi_-\Phi\p{\frac{w^{\top} \mu_{-1} }{\sqrt{w^\top \Sigma w}  }} + \pi_+\Phi\p{-\frac{w^{\top} \mu_1}{\sqrt{w^\top \Sigma w}  }}. 
\label{class_err_2}
\end{equation}

Recall that $\mu_{-1} = \bar{\mu} -\delta$,  $\mu_{+1} = \bar{\mu}+\delta$. The second simplification we notice is that $\hw_{\lambda}^\top\bar{\mu} \to_{a.s.} 0$. 

\begin{alemma} Under the conditions of Theorem \ref{theo:rda}, we have $\hw_{\lambda}^\top\bar{\mu} \to_{a.s.} 0$.
\label{reduce_to_same_noncentrality}
\end{alemma}

Lemma \ref{reduce_to_same_noncentrality} is proved in Section \ref{pf:reduce_to_same_noncentrality}.  By the same argument as above, this Lemma allows us to use the following even simpler formula - that does not involve $\bmu$ - in evaluating the limit of the error rate: 
\begin{equation}
\Err_{2}\p{w} = \Phi\p{-\frac{w^{\top} \delta }{\sqrt{w^\top \Sigma w}  }}.
\label{class_err_3}
\end{equation}

To show the convergence of 
$\smash{\Phi\p{-\hw_{\lambda}^{\top} \delta/\sqrt{\hw_{\lambda}^\top \Sigma \hw_{\lambda}}  }}$, we argue that the linear and quadratic forms involving $\hw_{\lambda}$ concentrate around their means, and then apply random matrix results to find the limits of those means. We start with the numerator. 

\begin{alemma} We have the limit $\hw_{\lambda}^{\top} \delta \rightarrow_{a.s.} \alpha^2 m(-\lambda)$, where $m(z)$ is the Stieltjes transform of the limit empirical eigenvalue distribution $F$ of the covariance matrix $\hSigma_c$.
\label{simple_lemma}
\end{alemma}

Lemma \ref{simple_lemma} is proved in Section \ref{pf:simple_lemma}.  To prove the convergence of the denominator, we decompose it as:

\begin{equation}
\hw_{\lambda}^\top \Sigma \hw_{\lambda} = \hdelta^\top \p{\hSigma_c + \lambda I_{p\times p}}^{-1}  \Sigma \p{\hSigma_c + \lambda I_{p\times p}}^{-1}\hdelta  = \tilde A + 2 \tilde B +  \tilde C
\label{big_decomp}
\end{equation}

where $M := \p{\hSigma_c + \lambda I_{p\times p}}^{-1}  \Sigma \p{\hSigma_c + \lambda I_{p\times p}}^{-1}$ and 
\begin{align*}
 \tilde A := \delta^\top M \delta, \,\, \,\, \,\,
 \tilde B := \delta^\top M \p{\hdelta - \delta},  \,\, \,\, \,\,
 \tilde C  := \p{\hdelta - \delta}^\top M \p{\hdelta - \delta} .
\end{align*}

One can show $ \tilde B \rightarrow_{a.s.} 0$ similarly to the analysis of the error terms in the proof of Lemmas \ref{offset_to_zero} and \ref{simple_lemma}; we omit the details. The two remaining terms will converge to nonzero quantities. First we show that:  

\begin{alemma} We have the convergence $ \tilde A := \delta^\top M \delta \rightarrow_{a.s.} \alpha^2 |\kappa'(\lambda)|$, where 
$$\kappa(\lambda) = \frac{1}{\gamma}\left(\frac{1}{\lambda v(-\lambda)}-1\right),$$

with $v$ the companion Stieltjes transform of the ESD of the covariance matrix, defined in \eqref{dual.ST}. Expressing the derivative explicitly, we have the limit $\tilde A \rightarrow_{a.s.} (v-\lambda v')/[\gamma (\lambda v)^2]$. The limit is strictly positive. 
\label{term_A}
\end{alemma}

Lemma \ref{term_A} is proved in Section \ref{pf:term_A}, using \cite{ledoit2011eigenvectors}'s result and a derivative trick similar to that  employed in a similar context by \cite{karoui2011geometric,rubio2012performance,zhang2013finite}.  Finally, the last statement that we need is:  
\begin{alemma} We have the limit 
$$ \tilde C := \p{\hdelta - \delta}^\top M \p{\hdelta - \delta}  \rightarrow_{a.s.}  \frac{v'-v^2}{ \lambda^2 v^4},$$

where $v$ the companion Stieltjes transform of the ESD of the covariance matrix, defined in \eqref{dual.ST}. 
\label{term_C}
\end{alemma}

Lemma \ref{term_C} is proved in Section \ref{pf:term_C}, as an application of the results of  \cite{hachem2008new} and \citet{chen2011regularized}. With all these results, we can now prove Theorem \ref{theo:rda}. 

\begin{proof}[Final proof of Theorem \ref{theo:rda}]
\label{final_proof_rda} 
By the decomposition \eqref{big_decomp} and Lemmas \ref{term_A} and \ref{term_C}, we have the convergence 
$$\hw_{\lambda}^\top \Sigma \hw_{\lambda} \to_{a.s.} \alpha^2 \frac{v-\lambda v'}{\gamma (\lambda v)^2} +  \frac{v'-v^2}{\lambda^2 v^4}.$$

By Lemma \ref{term_C}, the second  term is strictly positive. Therefore, combining with Lemma \ref{simple_lemma} and the continuous mapping theorem, we have
\begin{equation}
\frac{\hw_{\lambda}^{\top}\delta}{\sqrt{\hw_{\lambda}^\top \Sigma \hw_{\lambda}}} \to_{a.s.} \frac{\alpha^2 m(-\lambda)}{  \left[ \alpha^2 \frac{v-\lambda v'}{\gamma (\lambda v)^2} +  \frac{v'-v^2}{ \lambda^2 v^4} \right]^{1/2}}.
\label{main_conv}
\end{equation}

Denote by $\Theta$ the parameter on the right hand side. After algebraic simplification, we obtain that $\Theta$ has exactly the form stated in the theorem for the margin of RDA. To finish the proof, we show that the error rate is indeed determined by $\Theta$. From \eqref{main_conv} and the continuous mapping theorem, recalling the error rate $\Err_{2}\p{w}$ from \eqref{class_err_3}, we have
$\Err_{2}\p{\hw_{\lambda}} \to_{a.s.} \Phi(-\Theta).$

From Lemma \ref{reduce_to_same_noncentrality} and the definition of the error rate $\Err_{1}\p{w}$ from \eqref{class_err_2}, we can move from $\Err_2$ to $\Err_1$: 
$\Err_{2}\p{\hw_{\lambda}} - \Err_{1}\p{\hw_{\lambda}} \to_{a.s.} 0.$

Finally, from Lemma \ref{offset_to_zero} and the definition of the error rate $\Err_{0}\p{w}$ in Equation \eqref{class_err}, we can discard the offset $\smash{\hat b}$, and move from $\Err_1$ to $\Err_0$:
$\Err_{1}\p{\hw_{\lambda}} -  \Err_{0}\p{\hw_{\lambda}, \hat b}\to_{a.s.} 0.$

The last three statements imply that $\smash{\Err_{0}\p{\hw_{\lambda},\hat b}\to_{a.s.} \Phi(-\Theta)}$, which finishes the proof of Theorem \ref{theo:rda}.
\end{proof}

In the proofs of the lemmas we will use the following well-known statement repeatedly: 

\begin{alemma}[Concentration of quadratic forms, consequence of Lemma B.26 in \citet{bai2010spectral}] Let $x \in \RR^p$ be a random vector with i.i.d. entries and $\EE{x} = 0$, for which $\EE{(\sqrt{p}x_i)^2} = \sigma^2$ and $\sup_i \EE{(\sqrt{p}x_i)^{4+\eta}}$ $ < C$ for some $\eta>0$ and $C <\infty$. Moreover, let $A_p$ be a sequence of random $p \times p$ symmetric matrices independent of $x$, with uniformly bounded eigenvalues. Then the quadratic forms $x^\top A_p x $ concentrate around their means:  $\smash{x^\top A_p x - p^{-1} \sigma^2 \tr A_p \rightarrow_{a.s.} 0}$.
\label{quad_form}
\end{alemma}

Lemma \ref{quad_form} requires a small proof, which is provided in Section \ref{pf:quad_form}. The rest of this section contains the proofs of the lemmas. 

\subsubsection{Proof of Lemma \ref{offset_to_zero}}
\label{pf:offset_to_zero}

We start by conditioning on the random variables $\bar{\mu}, \delta$, or equivalently on $\mu_{\pm 1}$. Conditional on $\bar{\mu}, \delta$, we have that $\hmu_{+1} \sim \mathcal{N}(\mu_{+1},2\Sigma/n)$,  independently of $\hmu_{-1} \sim \mathcal{N}(\mu_{-1},2\Sigma/n)$. Therefore, $\hdelta \sim \nn\p{\delta, \, \Sigma/n}$, and independently $\hmu \sim \nn\p{\bar{\mu}, \, \Sigma/n}$. Further - still conditionally on $\bar{\mu}, \delta$ - it holds that $\hSigma_c$ and $\hmu_{\pm1}$ are independent by Gaussianity. This shows that conditionally on $\bar{\mu}, \delta$,  the random variables $\hSigma_c, \hmu, \hdelta$ are independent.

 Hence there exist two standard normal random vectors $Z,W \in \mathbb{R}^p$ independent of $\hSigma_c$ conditionally on $\bar{\mu}, \delta$, such that we can represent 
\begin{equation}
 \hdelta = \delta + \frac{1}{\sqrt{n}} \Sigma^{1/2}Z ,\,\,\,  \hmu = \bmu + \frac{1}{\sqrt{n}} \Sigma^{1/2}W.
 \label{decomposition}
\end{equation}

Crucially, this representation has the same form regardless of the value of $\delta$, $\bar{\mu}$, therefore the random variables $Z,W$ are unconditionally independent of $\delta, \, \bmu, \, \hSigma_c$. The unconditional indepdendence of $\delta,\bmu,\hSigma_c,Z,W$ will lead to convenient simplifications. 

We decompose $\hat{b}$ according to \eqref{decomposition} into the four terms that arise from expanding $\hdelta,\hmu$: 
\begin{align*}
\hat{b} = \hdelta^\top \p{\hSigma_c + \lambda I_{p \times p}}^{-1}\hmu = T_1 + T_2 + T_3 + T_4,
\end{align*}
where $L = \p{\hSigma_c + \lambda I_{p \times p}}^{-1}$ and 
\begin{align}
 T_1  &:=  \delta^\top L \bmu \\
 T_2  &:= n^{-1/2} Z^\top \Sigma^{1/2} L \bmu \\
 T_3 &:= n^{-1/2}\delta^\top L \Sigma^{1/2} W\\
 T_4 &:= n^{-1}  Z^\top \Sigma^{1/2} L \Sigma^{1/2} W.
 \label{T_defs}
\end{align}

The proof proceeds by showing that each of the $T_i$ converge to zero. 

{\bf The first term: $T_1$.} Let us denote $l = L\bmu$. Then by the independence and the zero-mean property of the coordinates of $\delta$

$$\EE{ (\delta^\top l)^4|l}=\sum_{i=1}^p l_i^4 \EE{ \delta_i^4} + \sum_{1\le i\neq j \le p} l_i^2l_j^2 \EE{\delta_i^2}\EE{\delta_j^2}.$$

Note that we have $\EE{(\sqrt{p} \delta_i)^4} \le C'$ for some constant $C'$ by the increasing relation between $L_p$ norms and by $\EE{(\sqrt{p} \delta_i)^{4+\eta}} \le C$. Therefore, with $C$ denoting some constant whose meaning may change from line to line

$$\EE{ (\delta^\top l)^4|l}\le \frac{C}{p^2} \sum_{i=1}^p l_i^4 + \frac{\alpha^4}{p^2} \sum_{1\le i\neq j \le p} l_i^2l_j^2 \le \frac{C}{p^2} \|l\|_2^4.$$

By the boundedness of the operator norm of $L$, we see $\|l\|_2 \le \|\bmu\|_2/\lambda$. By assumption, $\|\bmu\|_2\le C p^{1/4-\eta/2}$ almost surely for sufficiently large $p$. Therefore $\EE{ (\delta^\top l)^4}\le C p^{-(1+2\eta)}$ almost surely for sufficiently large $p$; and this bound is summable in $p$. Using the Markov inequality for the fourth moment, $\mathbb{P}(|\delta^\top l| \ge c) \le \EE{ (\delta^\top l)^4}/c^4$, and by the Borel-Cantelli lemma, it follows that $\delta^\top l \to 0$ almost surely. 

{\bf The second term: $T_2$.} This term differs from $T_1$ because $n^{-1/2}Z$ replaces $\delta$. To show the convergence $T_1 \to_{a.s.}0$ we only used the properties of the first four moments of $\delta$. The moments of $n^{-1/2}Z$ scale in the same way with $p$ as the moments of $\delta$. Therefore, the same proof shows $T_2 \to_{a.s.}0$.

{\bf The last two terms: $T_3$ and $T_4$.} The convergence of these terms follows directly from a well-known lemma, which we cite from \cite{couillet2011random}:

\begin{lemma}[Proposition 4.1 in \cite{couillet2011random}]
Let $x_n \in \mathbb{R}^n $ and $y_n \in \mathbb{R}^n $ be independent sequences of random vectors, such that for each $n$ the coordinates of $x_n$ and  $y_n$ are independent random variables. Moreover, suppose that the coordinates of $x_n$ are identically distributed with mean 0, variance $C/n$ for some $C>0$ and fourth moment of order $1/n^2$. Suppose the same conditions hold for $y_n$, where the distribution of the coordinates of $y_n$ can be different from those of $x_n$. Let $A_n$ be a sequence of $n \times n$ random matrices such that $\|A_n\|$ is uniformly bounded. Then $x_n^\top A_n y_n \to_{a.s.} 0$.
\end{lemma}

While this lemma was originally stated for complex vectors, it holds verbatim for real vectors as well. The lemma applied with $x_n = \delta$, $y_n = n^{-1/2}W$ and $A_n = L\Sigma^{1/2}$ shows convergence of $T_3 \to_{a.s.} 0$; and similarly it shows $T_4 \to_{a.s.} 0$. This finishes the proof of Lemma \ref{offset_to_zero}.

\subsubsection{Proof of Lemma \ref{reduce_to_same_noncentrality}}
\label{pf:reduce_to_same_noncentrality}

This follows from the proof of Lemma \ref{reduce_to_same_noncentrality} by noting $\hw_{\lambda}^\top\bar{\mu}  = \hdelta^\top L \bar{\mu} = T_1 + T_2$, where $L = \p{\hSigma_c + \lambda I_{p \times p}}^{-1}$, and with $T_i$ from Equation \eqref{T_defs}.

\subsubsection{Proof of Lemma \ref{simple_lemma}}
\label{pf:simple_lemma}
\begin{proof} We decompose
\begin{align}
\hw_{\lambda}^{\top} \delta
&= \hdelta^\top \p{\hSigma_c + \lambda I_{p\times p}}^{-1}\delta = \delta^\top \p{\hSigma_c + \lambda I_{p\times p}}^{-1} \delta + \p{\hdelta - \delta}^\top \p{\hSigma_c + \lambda I_{p\times p}}^{-1} {\delta}.
\label{numerator}
\end{align}
We will analyze the two terms separately, and show that the second term converges to 0.  

{\bf The first term in \eqref{numerator}: $ \delta^\top \p{\hSigma_c + \lambda I_{p\times p}}^{-1} \delta$.} Let us denote by $m_{\hSigma_c}(z) = p^{-1} \tr \{(\hSigma_c - zI_{p \times p})^{-1}\}$ the Stieltjes transform of the empirical spectral distribution (ESD) of $\hSigma_c$. We compute expectations using the assumption $\mathbb{E} \delta_i \delta_j= \delta_{ij}\alpha^2/p$ (where $\delta_{ij}$ is the Kronecker symbol that equal 1 if $i=j$ and 0 otherwise), and using that $\delta$ is independent of $\hSigma_c$: 
\begin{equation*} \EE{\delta^\top \p{\hSigma_c + \lambda I_{p\times p}}^{-1} \delta} = \frac{\alpha^2}{p} \EE{\tr\p{\hSigma_c + \lambda I_{p\times p}}^{-1}}  =   \alpha^2 \EE{m_{\hSigma_c}(-\lambda)}
\end{equation*}

We now argue that $\EE{m_{\hSigma_c}(-\lambda)} \rightarrow m(-\lambda)$, where $m(-\lambda)$ is the Stieltjes transform of the limit ESD $F$ of sample covariance matrices $\hSigma =$ $n^{-1} \Sigma^{1/2} E^\top E \Sigma^{1/2}$, where $E$ has i.i.d. entries of mean 0 and variance 1, and the spectrum of $\Sigma$ converges to $H$. Indeed, this follows from the Marchenko-Pastur theorem, given in Equation \eqref{eq:mp_lemma}: $\EE{m_{\hSigma}(-\lambda)} \rightarrow m(-\lambda)$; in addition we need to argue that centering the covariance matrix does not change the limit. Since our centered covariance matrix differs from the usual centered sample covariance matrix, for completeness we state this result as a lemma below:

\begin{alemma}
Let $\hSigma_c$ be the centered and rescaled covariance matrix used in RDA. Then 
\begin{enumerate}
\item Under the conditions of Theorem \ref{theo:rda}, the limit ESD of $\hSigma_c$ equals, with probability 1, the limit ESD $F$ of sample covariance matrices $\hSigma = n^{-1} \Sigma^{1/2} V^\top V \Sigma^{1/2}$, where $V$ is $n \times p$ with i.i.d. entries of mean 0 and variance 1. Therefore centering the covariance does not change the limit. 
\item The Stieltjes transform of $\hSigma_c$, $m_{\hSigma_c}(z)$ converges almost surely and in expectation to $m(z)$, the Stieltjes transform of $F$. For each $z \in \mathbb{C}\setminus \mathbb{R}^+$, $m_{\hSigma_c}(z) \to_{a.s} m(z)$; and $\mathbb{E}m_{\hSigma_c}(z) \to m(z)$.
\end{enumerate}
\label{rescale_center}
\end{alemma}
\begin{proof}[Proof of Lemma \ref{rescale_center}]

Let $u_i$ be the centered data points, $u_i = x_i - \mu_{+1}$ for the positive training examples, and $u_i = x_i - \mu_{-1}$ for the negative training examples. The $u_i$ have mean 0 and covariance matrix $\Sigma$. Let further $\hnu_{+1} = \hmu_{+1} - \mu_{+1}$ be the centered mean of the positive training examples; and define $\hnu_{-1} = \hmu_{-1} - \mu_{-1}$ analogously. We observe that 
\begin{align*}
\hSigma_c &= \frac{1}{n-2} \left[ \sum_{i:y_i=1}(x_i-\hmu_{+1})(x_i-\hmu_{+1})^\top + \sum_{i:y_i=-1}(x_i-\hmu_{-1})(x_i-\hmu_{-1})^\top \right] \\
&= \frac{1}{n-2} \left[ \sum_{i:y_i=1}(u_i-\hnu_{+1})(u_i-\hnu_{+1})^\top + \sum_{i:y_i=-1}(u_i-\hnu_{-1})(u_i-\hnu_{-1})^\top \right] \\
&=  \frac{1}{n-2} \left[ \sum_{i:y_i=1}u_iu_i^\top + \sum_{i:y_i=-1}u_iu_i^\top - \frac{n}{2}(\hnu_{+1}\hnu_{+1}^\top +\hnu_{-1}\hnu_{-1}^\top) \right] \\
& = \frac{n}{n-2} \left[ \frac{1}{n}\Sigma^{1/2} V^\top P V\Sigma^{1/2} \right]. \\
\end{align*}
where $V$ is the $n \times p$ matrix with each row equal to $\Sigma^{-1/2}u_i$, which are i.i.d. Gaussian vectors with i.i.d. entries of mean 0 and variance 1. Also, $P$ is the projection matrix $P = I_{p \times p} - 2n^{-1}(e_{+1}e_{+1}^\top+e_{-1}e_{-1}^\top)$, where the vectors $e_{\pm1}$ are the indicator vectors of the training examples with labels $\pm 1$. Recalling that $\hSigma = n^{-1} \Sigma^{1/2} V^\top V \Sigma^{1/2}$ is the uncentered, $1/n$-normalized covariance matrix, the difference between $\hSigma$ and $\hSigma_c$ is

\begin{align*}
\hSigma - \hSigma_c &=  \frac{1}{n} \Sigma^{1/2} V^\top V \Sigma^{1/2}  - \frac{n}{n-2} \left[ \frac{1}{n}\Sigma^{1/2} V^\top P V\Sigma^{1/2} \right]  \\
&=
 \frac{2}{n(n-2)}   \Sigma^{1/2} V^\top V \Sigma^{1/2}  + \frac{1}{n-2}  \Sigma^{1/2} V^\top (I_{p \times p} - P) V \Sigma^{1/2} = \Gamma_1 + \Gamma_2. 
\end{align*}

It is easy to check that the two error terms $\Gamma_1$ and $\Gamma_2$ are small. Specifically, we can verify that the Frobenius norm $\|\Gamma_1\|_{\text{Fr}}^2 \to 0$. Therefore, By Corollary A.41 in  \cite{bai2010spectral} it follows that $\Gamma_1$ can be ignored when computing the limit ESD. Further, $I_{p \times p} - P$ is of rank at most two by the definition of $P$. Therefore, by Theorem A.44 in  \cite{bai2010spectral}, $\Gamma_2$ does not affect the limit ESD. Putting these together, it follows that $\hSigma_c$ has the same ESD as $\hSigma$; the latter exists due to the Marchenko-Pastur theorem given in Equation \eqref{eq:mp_lemma}. Therefore, the ESD of $\hSigma_c$ converges, with probability 1, to $F$. This finishes the first claim in the Lemma. 

Claim 2 then follows immediately by the properties of weak convergence of probability measures, because the Stieltjes transform is a bounded continuous functional of a probability distribution. 

\end{proof}

Therefore, going back to the proof of Lemma \ref{simple_lemma}, we obtain using Lemma \ref{rescale_center}, that $\EE{m_{\hSigma_c}(-\lambda)} \rightarrow m(-\lambda)$. 
Now note that each eigenvalue of $\p{\hSigma_c + \lambda I_{p\times p}}^{-1}$ is uniformly  bounded in $[0,1/\lambda]$, $\delta$ is independent of $\hSigma_c$,  and the $4+\eta$-th moment of $\sqrt{p}\delta_i$ is uniformly bounded, so by the concentration of quadratic forms, Lemma \ref{quad_form}:
$$ \delta^\top \p{\hSigma_c + \lambda I_{p\times p}}^{-1} \delta - \EE{\delta^\top \p{\hSigma_c + \lambda I_{p\times p}}^{-1} \delta} \rightarrow_{a.s.} 0. $$ 

Putting everything together, we have shown that the first term converges almost surely to $\alpha^2 m(-\lambda)$.

{\bf The second term in \eqref{numerator} converges to zero: } Using the decomposition \eqref{decomposition} from the proof of Lemma \ref{offset_to_zero} in Section \ref{pf:offset_to_zero}, and the notation $L = \p{\hSigma_c + \lambda I_{p \times p}}^{-1}$ we can write
$\varepsilon_2 := \delta^\top \p{\hSigma_c + \lambda I_{p\times p}}^{-1} \p{\hdelta- \delta} = \frac{1}{\sqrt{n}} \delta^\top L \Sigma^{1/2}Z.$

This has the same distribution as $T_3= n^{-1/2} \delta^\top L \Sigma^{1/2}W$, because $Z,W$ are identically distributed, independently of $\delta,\hSigma_c$. In Lemma \ref{offset_to_zero}, we showed $T_3 \to_{a.s.} 0$, so  $\varepsilon_2 \rightarrow_{a.s.} 0$. 
\end{proof}

\subsubsection{Proof of Lemma \ref{term_A}}
\label{pf:term_A}
In this proof it will be helpful to use some properties of smooth complex functions. Let $D$ be a domain, i.e. a connected open set of $\mathbb{C}$. A function $f:D \to \mathbb{C}$ is called analytic on $D$ if it is differentiable as a function of the complex variable $z$ on $D$. The following key theorem, sometimes known as Vitali's theorem, ensures that the derivatives of converging analytic functions also converge.

\begin{alemma}[see Lemma 2.14 in \cite{bai2010spectral}] Let $f_1, f_2,\ldots $ be analytic on the domain $D$, satisfying $|f_n(z)| \le M$ for every $n$ and $z$ in $D$. Suppose that there is an analytic function $f$ on $D$ such that $f_n(z) \to f(z)$ for all $z \in D$. Then it also holds that $f_n'(z) \to f'(z)$ for all $z \in D$.
\label{vitali}
\end{alemma}

We can now prove Lemma \ref{term_A}. 
\begin{proof}[Proof of Lemma \ref{term_A}] We need to prove the convergence of $ \tilde A = \delta^\top M \delta$, where $M = \p{\hSigma_c + \lambda I_{p\times p}}^{-1}  \Sigma \p{\hSigma_c + \lambda I_{p\times p}}^{-1}$. Using the assumption $\mathbb{E} \delta_i \delta_j= \delta_{ij}\alpha^2/p$, and that $\delta,\hSigma_c$ are indpendent, we get $\EE{ \tilde A} = \frac{\alpha^2}{p} \EE{\tr\p{\Sigma \p{\hSigma_c + \lambda I_{p\times p}}^{-2}}}$.

Now, the eigenvalues of $\Sigma$ are bounded above by $B$, and those of $\p{\hSigma_c + \lambda I_{p\times p}}^{-1}$ are bounded above by $1/\lambda$; this shows the eigenvalues of $L$ are uniformly bounded above. Further, $\delta,L$ are independent, and $\delta$ has $4+\eta$-th moments, so we can use the concentration of quadratic forms,  Lemma \ref{quad_form}, to conclude  $ \tilde A -\EE{\tilde A}  \rightarrow_{a.s.} 0$. It remains to get the limit of $\EE{\tilde A}$. 

To do this we use a derivative trick; this technique is similar to those employed in a similar context by \cite{karoui2011geometric,rubio2012performance,zhang2013finite}. We will construct a function with two properties: (1) its derivative is the quantity $\EE{\tilde A}$ that we want, and (2) its limit is convenient to obtain. Based on these two properties, Vitali's theorem will allow us to obtain the limit of $\EE{\tilde A}$. 

Accordingly, consider two general $p \times p$ positive definite matrices $D,E$ and introduce the function
$f_p(\lambda; D,E) = \frac 1p \tr \p{ D \p{E + \lambda I_{p\times p}}^{-1}}.$ Note that the derivative of $f$ with respect to $\lambda$ is $f_p'(\lambda; D,E) = - \frac 1p \tr \p{ D \p{E + \lambda I_{p\times p}}^{-2}}.$

This suggests that the function we should use in the derivative trick is  $f_p(\lambda; \Sigma ,\hSigma_c )$; indeed the limit we want to compute is $\EE{\tilde A}=-\EE{f_p'(\lambda; \Sigma ,\hSigma_c )}$. Conveniently, the limit of $f_p$ is known by the Ledoit-Peche result \eqref{eq:ledoit_peche}:

\begin{equation*} f_p(\lambda; \Sigma ,\hSigma_c ) \rightarrow_{a.s.} \kappa(\lambda) =  \frac{1}{\gamma}\left(\frac{1}{\lambda v(-\lambda)}-1\right), 
\end{equation*}

for all $\lambda \in \mathcal {S} : = \{u+ iv: v\neq0 \mbox{, or } v=0, u >0\}$.

Next we check the conditions for applying Vitali's theorem, Lemma \ref{vitali}. By inspection, the function $f_p(\lambda; \Sigma ,\hSigma_c )$ is an analytic function of $\lambda$ on $\mathcal {S}$ with derivative
$$\frac{\partial \kappa}{\partial \lambda} =\kappa'(\lambda)= - \frac{v-\lambda v'}{\gamma (\lambda v)^2}.$$

Furthermore, $f_p(\lambda; \Sigma ,\hSigma_c )$ is bounded in absolute value: 
$
|f_p(\lambda; \Sigma ,\hSigma_c )| \le \frac{\|\Sigma\|_2}{\lambda}  \le  \frac{B}{\lambda}. 
$

This shows that $f_p(\lambda; \Sigma ,\hSigma )$ is a bounded sequence. Therefore, we can apply Lemma \ref{vitali} to the sequence of analytic functions $f_p$ on the domain $\mathcal {S}$, on the set of full measure on which $f_p$ converges, to obtain that the derivatives also converge on $\mathcal {S}$: $ f'_p(\lambda; \Sigma ,\hSigma ) \rightarrow_{a.s.} \kappa'(\lambda).
$
Finally, since the functions $f_p$ are bounded, the dominated convergence theorem implies that $\mathbb{E}f'_p(\lambda; \Sigma ,\hSigma ) \to \kappa'(\lambda)$.  Since $\EE{\tilde A}=-\EE{f_p'(\lambda; \Sigma ,\hSigma_c )}$, this shows that $\EE{\tilde A} \rightarrow  - \alpha^2 \kappa'(\lambda)$, as required. 

Finally, to see that the limit  of $\EE{\tilde A} > 0$, we use that by assumption the eigenvalues of $\Sigma$ are lower bounded by some $b>0$, and those of $\hSigma_c = n^{-1}\Sigma^{1/2} V^\top P V\Sigma^{1/2}$ are upper bounded by $B\, n^{-1}\|V^\top V\|_2$, hence

$$
\EE{\tilde A} = \frac{\alpha^2}{p} \EE{\tr\p{\Sigma \p{\hSigma_c + \lambda I_{p\times p}}^{-2}}} \ge \frac{\alpha^2 b}{B \, n^{-1} \|V^\top V\|_2}.
$$

Since $V$ has i.i.d. standard normal entries, it is well known that the operator norm of $n^{-1} \|V^\top V\|_2$ is bounded above almost surely (see Theorem 5.8 in \citep{bai2010spectral}, which requires only a fourth moment on the entries of $V$). Therefore, $\liminf_p \EE{\tilde A} \ge c>0$ for some fixed constant $c>0$. 

\end{proof}

\subsubsection{Proof of Lemma \ref{term_C}}
\label{pf:term_C}

We want to find the limit of $\tilde C =$ $ \p{\hdelta - \delta}^\top M \p{\hdelta - \delta}$, where $M = \p{\hSigma_c + \lambda I_{p\times p}}^{-1}  \Sigma \p{\hSigma_c + \lambda I_{p\times p}}^{-1}$. By the decomposition \eqref{decomposition} we have $\hdelta - \delta = n^{-1/2} \Sigma^{1/2}Z$, so  $\tilde  C = \frac{1}{n} Z^\top \Sigma^{1/2} M \Sigma^{1/2} Z$. Since $Z \sim \nn\p{0,I_p}$ independently of $\hSigma_c$ (and thus of $M$), we can write $ \EE{\tilde C} =  \frac{1}{n} \tr \p{ \Sigma^{1/2} M \Sigma^{1/2} }= \frac{\gamma}{p} \tr\p{Q_{\lambda}^2}$, where $Q_{\lambda} :=  \Sigma \p{\hSigma_c + \lambda I }^{-1}$. By the concentration lemma \ref{quad_form}, $\tilde C - \EE{\tilde C} \rightarrow_{a.s.} 0$, so all that remains is to find the limit of $x_0 = p^{-1} \tr\p{Q_{\lambda}^2}$. 

Lemma 1 in \cite{hachem2008new} states that, under the assumptions of Theorem \ref{theo:rda}, one has $\Var{x_0} = O(1/n^2)$, and also $\EE{x_0} = L + O(1/n^2)$ for a certain deterministic quantity $L$. These results imply that $x_0 \to_{a.s.} L$. However, the form of the limit $L$ given in that paper is not convenient for our purposes. 

A convenient form for the limit is obtained in \cite{chen2011regularized}. Their convergence result holds in probability, which is weaker than what we need; this explains why we also need the results of \cite{hachem2008new}. \cite{chen2011regularized} consider sequences of problems of the form $v_i \sim_{iid} \mathcal{N}(\mu_p,\Sigma_p)$, where $\Sigma_p$ is a sequence of covariance matrices that obeys the same conditions we assumed in the statement of Theorem \ref{theo:rda}, and the $\mu_p$ are arbitrary fixed vectors. They form the sample covariance matrix $S_n = {(n-1)}^{-1} \sum_{i=1}^{n} (v_i - \bar v)(v_i - \bar v)^\top$, where $\bar v = n^{-1} \sum_{i=1}^{n} v_i$. In addition to the above conditions, they also assume the additional condition $\sqrt{n}|p/n-\gamma|\to 0$. Then their result states:

\begin{alemma}[Lemma 2 of \cite{chen2011regularized}, pp 1357] 
Under the above conditions,  we have the convergence in probability
$$\frac{1}{p} \tr\p{\left[\Sigma_p \p{S_n+ \lambda I }^{-1}\right]^2} \to_{p} \Theta_2(\lambda,\gamma),$$

where $\Theta_2(\lambda,\gamma)$ is defined in the statement of Theorem 1 of \cite{chen2011regularized}, on pp 1348
$$\Theta_2(\lambda,\gamma) = \frac{1-\lambda m(-\lambda)}{(1-\gamma+\gamma\lambda m(-\lambda))^3}-\lambda \frac{m(-\lambda)-\lambda m'(-\lambda)}{(1-\gamma+\gamma\lambda m(-\lambda))^4}.$$
\label{chen_lemma}
\end{alemma}

Lemma \ref{chen_lemma} is stated for the usual centered covariance matrix $S_n$. In Lemma \ref{term_C}, the covariance matrix of interest is $\hSigma_c =  (n-2)^{-1}\Sigma^{1/2} V^\top P V\Sigma^{1/2}$, where $P$ is the projection matrix $P = I_{p \times p} - 2n^{-1}(e_{+1}e_{+1}^\top+e_{-1}e_{-1}^\top)$. Similarly to what we already argued several times in this paper (e.g. in Lemma \ref{rescale_center}), the two covariance matrices have identical limit ESD. Therefore Lemma \ref{chen_lemma} applies to our setting. Combining this with Lemma 1 in \cite{hachem2008new}, we have the convergence:
$$\frac{1}{p} \tr\p{\left[\Sigma \p{\hSigma_c+ \lambda I }^{-1}\right]^2} \to_{a.s.} \Theta_2(\lambda,\gamma).$$

Next, we notice by the definition of $v$ in \eqref{dual.ST} that $1-\gamma+\gamma\lambda m(-\lambda) = \lambda v(\lambda)$, as well as $1-\lambda m(-\lambda) = \gamma^{-1}(1-\lambda v(-\lambda))$; and by taking derivatives $m(-\lambda) - \lambda m'(-\lambda) = \gamma^{-1}(v(-\lambda) - \lambda v'(-\lambda))$. We rewrite the limit $\Theta_2$ in terms of $v$:
\begin{align*}
\Theta_2(\lambda,\gamma) & = \frac{1-\lambda v}{\gamma(\lambda v)^3}-\lambda \frac{v-\lambda v'}{\gamma(\lambda v)^4} =\frac{v'-v^2}{ \gamma \lambda^2 v^4}.
\end{align*}

Therefore, the limit of $\tilde C$ equals $\gamma \Theta_2$, as stated in Lemma \ref{term_C}. 

\subsubsection{Proof of Lemma \ref{quad_form}}
\label{pf:quad_form}
We will use the following Trace Lemma quoted from \cite{bai2010spectral}. 

\begin{alemma}[Trace Lemma, Lemma B.26 of \cite{bai2010spectral}] Let $y$ be a p-dimensional random vector of i.i.d. elements with mean 0. Suppose that $\EE{y_i^2} = 1$, and let $A_p$ be a fixed $p \times p$ matrix. Then  
$$\EE{|y^\top A_p y-\tr A_p|^q} \le C_q \left\{\left(\EE{y_1^4}\tr[A_pA_p^\top]\right)^{q/2}+\EE{y_1^{2q}}\tr[(A_pA_p^\top)^{q/2}]\right\},$$ 

for some constant $C_q$ that only depends on $q$. 
\label{trace_lemma}
\end{alemma}
Under the conditions of Lemma \ref{quad_form}, the operator norms $\|A_p\|_2$ are bounded by a constant $C$, thus $\tr[(A_pA_p^\top)^{q/2}] \le p C^q$ and $\tr[A_pA_p^\top] \le p C^2$. Consider now a random vector $x$ with the properties assumed in the present lemma. For $y = \sqrt{p}x/\sigma$ and $q = 2+\eta/2$, using that $\EE{y_i^{2q}}\le C$ and the other the conditions in Lemma \ref{quad_form}, Lemma \ref{trace_lemma} thus yields
$$\frac{p^q}{\sigma^{2q}}\EE{|x^\top A_p x-\frac{\sigma^2}{p}\tr A_p|^q} \le C \left\{\left( p C^2\right)^{q/2}+\p{pC}^q\right\},$$ 

or equivalently $\EE{|x^\top A_p x-\frac{\sigma^2}{p}\tr A_p|^{2+\eta}} \le C p^{-(1+\eta/4)}$. Since this bound is summable in $p$, Markov's inequality applied to the $2+\eta$-th moment of $\varepsilon_p = x^\top A_p x-\frac{\sigma^2}{p}\tr A_p$ and the Borel-Cantelli lemma yield the almost sure convergence $\varepsilon_p \to_{a.s.} 0$. 

\subsection{Proof of Theorem \ref{theo:rda_unequal}: Unequal sampling}
\label{sec:unequal_sampling}

This proof is very similar to that of Theorem \ref{theo:rda}, so we will omit some details. Suppose that the classes have unequal probabilities $\mathbb{P}(y_i=+1)=\pi_+$, $\mathbb{P}(y_i=-1)=\pi_-$, and the conditional model $x|y \sim \nn\p{\mu_y, \, \Sigma}$ holds. The Bayes optimal classifier is $\sign(f(x))$, where $f(x) = (x- (\mu_{+1}+\mu_{-1})/2)^\top \Sigma^{-1} (\mu_{+1}-\mu_{-1}) + \log(\pi_+/\pi_-)$. Using the same notation as before, $f(x)$ equals $f(x) = x^\top \Sigma^{-1} \delta - \bmu^\top \Sigma^{-1} \delta + \log(\pi_+/\pi_-)$. The error rate of a generic linear classifier $\sign(x^\top w + b)$ conditional on $w,b$ equals \eqref{class_err}.

We observe $n_{+1}$ samples with label $y_i=1$, and $n_{-1}$ samples with label $y_i=-1$. We consider a general regularized classifier $\sign(\hat f_\lambda(x))$, where $\hat f_\lambda(x) = x^\top \hat w_\lambda + \hat b$, and 
$\hat w_\lambda  = (\hSigma_c + \lambda I_{p \times p})^{-1} \hdelta,\,\,\,\, \hat b = -\hmu^\top \hw_\lambda + c.$

As usual, we have $\hat \delta = (\hmu_{+1}-\hmu_{-1})/2$, $\hat \mu = (\hmu_{+1}+\hmu_{-1})/2$ and  $\hmu_{\pm 1}  =  \sum_{\cb{i : y_i = \pm 1}}  x_i / n_{\pm 1}$. The centered sample covariance matrix $\hSigma_c$ retains its original definition.

We evaluate the limits of the linear and quadratic forms in the error rate, arising when we replace $w$ and $b$ by $\hat w$ and $\hat b$, respectively. We assume that the same regularity conditions as in Theorem $\ref{theo:rda}$ hold, with the additional requirement that the ratios $p/n_{\pm 1}$ each converge  to positive constants: $p/n_{\pm 1} \to \gamma_{\pm 1}>0$. For two independent $p$-dimensional standard normal random variables $Z_{\pm1}$, which are also independent of $\delta,\bmu$, we have the stochastic representation

$$\hmu = \bmu + \frac12 \Sigma^{1/2}\left(\frac{Z_1}{n_1^{1/2}}+\frac{Z_2}{n_2^{1/2}}\right),\,\,\,\,\hdelta = \delta + \frac12 \Sigma^{1/2}\left(\frac{Z_1}{n_1^{1/2}}-\frac{Z_2}{n_2^{1/2}}\right).$$

{\bf The limit of $\hat b = -\hmu^\top \hw_\lambda + c$:} To evaluate this limit, we expand the inner product $-\hmu^\top \hw_\lambda$ using the stochastic representation of $\hdelta$. As in the proof of Theorem \ref{theo:rda}, most terms in the expansion tend to 0 due to independence. Denote by $A_n \approx B_n$ that two random variables are asymptotically almost surely equivalent, i.e. $|A_n - B_n| \rightarrow_{a.s.} 0$. We have by arguments similar to those in Theorem \ref{theo:rda} that
\begin{align*}
-\hmu^\top \hw_\lambda &\approx \frac{1}{4n_{-1}} Z_{-1}^\top \Sigma^{1/2} (\hSigma_c + \lambda I_{p \times p})^{-1} \Sigma^{1/2} Z_{-1} - \frac{1}{4n_{1}} Z_1^\top \Sigma^{1/2} (\hSigma_c + \lambda I_{p \times p})^{-1} \Sigma^{1/2} Z_1 \\
& \approx \left( \frac{1}{4n_{-1}}  - \frac{1}{4n_{1}} \right) \tr\left( \Sigma \left(\hSigma + \lambda I_{p \times p} \right)^{-1}\right) \\
& \approx \frac{\gamma_{-1}-\gamma_{+1}}{4} \kappa(\lambda).
\end{align*}

{\bf The limit of $\mu_{\pm 1}^\top \hw_\lambda$:} We write $\mu_{1}^\top \hw_\lambda = (\bmu + \delta)^\top   (\hSigma_c + \lambda I_{p \times p})^{-1} \hdelta$, and $\mu_{1}^\top \hw_\lambda \approx \delta^\top  (\hSigma_c + \lambda I_{p \times p})^{-1} \delta \approx \alpha^2 m(-\lambda)$. Similarly $\mu_{-1}^\top \hw_\lambda \approx - \alpha^2 m(-\lambda)$.

{\bf The limit of $\hw_\lambda^\top \Sigma \hw_\lambda$:} On expanding this expression using the stochastic representation, the cross-terms vanish asymptotically. Denoting $M =(\hSigma_c + \lambda I_{p \times p})^{-1} \Sigma (\hSigma_c + \lambda I_{p \times p})^{-1} $, we obtain
\begin{align*}
\hw_\lambda^\top \Sigma \hw_\lambda & \approx \delta^\top \Sigma^{1/2} M \Sigma^{1/2}\delta +  \frac{1}{4n_{-1}} Z_{-1}^\top M Z_{-1} + \frac{1}{4n_{1}} Z_1^\top M Z_1 \\
& \approx \frac{\alpha^2}{p} \tr(\Sigma M) + \left( \frac{1}{4n_{-1}}  + \frac{1}{4n_{1}} \right) \tr(M) \\
& \approx \alpha^2 \frac{v-\lambda v'}{\gamma (\lambda v)^2} + \frac{\gamma_{-1}+\gamma_{+1}}{4}  \frac{v'-v^2}{ \lambda^2 v^4}.
\end{align*}
Let us denote by $Q$ the limit obtained. 

{\bf Putting everything together:} Using the above formulas, we see that as claimed in Eq. \ref{eq:unequal_sampling}, $Err_0\p{\hw_\lambda,\hat b} \to_{a.s.} U$:
\begin{eqnarray*}
U = \pi_- \Phi\p{\frac{- \alpha^2 m(-\lambda) + \frac{\gamma_{-1}-\gamma_{+1}}{4} \kappa(\lambda) + c }{\sqrt{Q}  }} +
 \pi_+ \Phi\p{-\frac{ \alpha^2 m(-\lambda)+ \frac{\gamma_{-1}-\gamma_{+1}}{4} \kappa(\lambda) + c}{\sqrt{Q}  }}.
\end{eqnarray*}

\subsection{Proof of Corollary \ref{cor:rda_id}}
\label{rda_null}

From Theorem \ref{theo:rda}, the limit error rate is $\Phi(-\Theta)$, where $\Theta = \alpha^2 m(-\lambda)/\sqrt{\alpha^2\,r(\lambda) + \gamma q(\lambda)}$, and 
$$r(\lambda) =\lim_{p\to\infty}\mathbb{E}\frac{1}{p}\tr\p{\Sigma \p{\hSigma_c+ \lambda I }^{-2}}  ; \,\,\,\, q(\lambda) = \lim_{p\to\infty}\mathbb{E}\frac{1}{p}\tr\p{\left[\Sigma \p{\hSigma_c+ \lambda I }^{-1}\right]^2}.$$

  Since $\Sigma = I_{p \times p}$, we have $r(\lambda) = q(\lambda)$, and both are equal to the limit of $\EE{\frac{1}{p}\tr\p{\p{\hSigma_c+ \lambda I }^{-2}} }$. Analogously to the proof of Lemma  \ref{term_C}, this limit equals $m'_I(-\lambda;\gamma)$. Therefore $\Theta$ simplifies to  $  \frac{\alpha^2}{\sqrt{\alpha^2 + \gamma}}\frac{m_I(-\lambda; \gamma)}{\sqrt{m'_I(-\lambda; \gamma)}}$.

The quantity $m'_I(-\lambda; \gamma)$ can be expressed in terms of $m_I$ by differentiating the Marchenko-Pastur equation: $m_I(z;\gamma) = 1/(1-z-\gamma-\gamma z m_I(z;\gamma))$. We get $m' = m^2 (1 + \gamma m)/( 1 - \gamma z m^2)$, which leads to the claimed expression for $\Theta$. For $\gamma=1$ we get the required formula from Eq.  \eqref{identity_stieltjes}  after some calculations. 
\subsection{Note on the Ledoit-Peche result \eqref{eq:ledoit_peche}}
\label{lp_appendix}

\cite{ledoit2011eigenvectors} prove \eqref{eq:ledoit_peche} in their Lemma 2. Our notation differs from theirs: $\gamma$ here is equal to their $\gamma^{-1}$ (because the role of $n,p$ is reversed); $\hSigma$ here is $S_N$, and $-\lambda$ here corresponds to $z$.  Their limit is given in terms of $m$, but simplifies to 
$$\kappa(\lambda) = \frac{1-\lambda m(-\lambda)}{1-\gamma(1-\lambda m(-\lambda))} =  \frac{1}{\gamma}\left(\frac{1}{\lambda v(-\lambda)}-1\right).$$

Their theorem requires a finite 12th moment on $\Sigma^{-1/2}x_i$, which holds in our case. Their results are stated for uncentered covariance matrices, but it is well known that centering does not affect limit of the eigenvalue distribution of $\hSigma$, so the result still holds \citep[see, e.g., p. 39 of ][]{bai2010spectral}. Furthermore, the normalization factor $1/n$ in front of the covariance matrix can be replaced by $1/(n-2)$, as needed by Regularized Discriminant Analysis. This follows by standard perturbation inequalities in the Frobenius norm \cite[see Corollary A.42 in][]{bai2010spectral}. The calculation is fleshed out in detail in the proof of Lemma \ref{term_C}. Finally, their convergence is stated for $z\in \mathbb{C}^+$, but by standard arguments it can be extended to real $z<0$. 

\subsection{Proof of Theorem \ref{theo:lda_ir}}
\label{pf:theo:lda_ir}

We will use the representation of the margin from theorem \ref{theo:rda}. To evaluate the necessary limits, it is helpful to represent the Stieltjes transforms and their derivatives as expectations with respect to the ESD. Thus, let $Y$ be a random variable distributed according to the ESD $F$, and let $\underline Y$ be a random variable distribued according to the companion ESD $\underline F$. Then $m$, $v$ are the Stieltjes transforms of $Y,\underline Y$, respectively. Hence
\begin{equation}
\label{stoch_rep}
m(-\lambda) = \EE{\frac{1}{Y+\lambda}}, \, m'(-\lambda) = \EE{\frac{1}{(Y+\lambda)^2}}, \, m(-\lambda)-\lambda m'(-\lambda) = \EE{\frac{Y}{(Y+\lambda)^2}}.
\end{equation}

Further $\lambda v(-\lambda)  = 1 + \gamma\left(\lambda m(-\lambda) - 1 \right)$, so $\lambda v(-\lambda) = 1-\gamma \EE{\frac{Y}{Y+\lambda}}$. 

{\bf The error rate of LDA: } As explained in Section \ref{pf:strong_signal}, for $\gamma<1$, $Y$ is supported on a compact set bounded away from 0, so $Y>c>0$ for some $c$. This will allow us to take the limits as $\lambda\downarrow0$ inside the expectation, using the dominated convergence theorem; we will not repeat this fact. Using equation \eqref{stoch_rep}, we have $\lim_{\lambda \downarrow 0} m(-\lambda) = \EE{Y^{-1}}$, $\lim_{\lambda \downarrow 0} [v(-\lambda)-\lambda v'(-\lambda)]\gamma^{-1} = \lim_{\lambda \downarrow 0} m(-\lambda)-\lambda m'(-\lambda) =  \EE{Y^{-1}}$, and by the formula for $\lambda v$, $\lim_{\lambda\downarrow0}\lambda v(-\lambda) = 1-\gamma$ (which was already shown in Section \ref{pf:strong_signal}).

Differentiating the formula for the companion Stieltjes transform, we see $\lambda^2 v'(-\lambda)  = 1+ \gamma\left(\lambda^2 m'(-\lambda)- 1\right)$. Hence, 
$$\lim_{\lambda \downarrow 0}\frac{v'(-\lambda)}{v^2(-\lambda)}=\lim_{\lambda \downarrow 0}\frac{ 1+ \gamma\left(\lambda^2 m'(-\lambda)- 1\right)}{\lambda^2v^2(-\lambda)} = \frac{\lim_{\lambda \downarrow 0}\left[ 1+ \gamma\left(\lambda^2 m'(-\lambda)- 1\right)\right]}{(1-\gamma)^2} = \frac{1}{1-\gamma},$$

where we have used that $\lim_{\lambda \downarrow 0}\lambda^2 m'(-\lambda)=\lim_{\lambda\downarrow 0} \EE{\frac{\lambda^2}{(Y+\lambda)^2}}=0$. We conclude that $\lim_{\lambda \downarrow 0}v'(-\lambda)/v^2(-\lambda)-1 = \gamma/(1-\gamma)$. Putting everything together, we find

$$\Theta_{\text{LDA}} = \lim_{\lambda\downarrow0} \frac{\alpha^2 \, \tau }{\sqrt{\alpha^2 \, \eta + \theta}}  =  \frac{\alpha^2 \, \EE{Y^{-1}} (1-\gamma) }{\sqrt{\alpha^2 \, \EE{Y^{-1}} + \frac{\gamma}{1-\gamma}}}. $$

Let $T$ be a random variable distributed according to the PSD $H$. By taking the limit as $z \to 0$, $z \in \mathbb{C}^+$ in the Marchenko-Pastur equation below, (the validity of the limit is justified by arguments similar to those in Section  \ref{pf:strong_signal}),
\begin{equation*}
\label{eq:mp_eq}
m(z) = \int_{t=0}^{\infty} \frac{\, dH(t)}{t \p{ 1 - \gamma - \gamma z m\p{z} }-z}.
\end{equation*}
  we find $m(0) = \int\frac{1}{t(1-\gamma)}dH(t)$, or equivalently $\EE{Y^{-1}} = \EE{T^{-1}} /(1-\gamma)$. This leads to the formula for $\Theta_{\text{LDA}}$.

{\bf The error rate of IR: } We have $\lambda m(-\lambda) = \EE{\frac{\lambda}{Y+\lambda}}$. As explained in Section \ref{pf:strong_signal}, $Y$ is a bounded random variable, so $\lim_{\lambda\to\infty}\lambda m(-\lambda) = 1$. Similarly $\lim_{\lambda\to\infty} \lambda v(-\lambda) = 1$. Next, we note

$$\lim_{\lambda\to\infty}\lambda^2 \left[m(-\lambda)-\lambda m'(-\lambda)\right]  = \lim_{\lambda\to\infty}\EE{\left(\frac{\lambda}{Y+\lambda}\right)^2Y}=\EE{Y}.$$

Finally, we evaluate the limit of $\lambda^2\left(\frac{v'(-\lambda)}{v^2(-\lambda)}-1\right)$. Noting that $\lambda v$ tends to 1, it is enough to find the limit of $\lambda^4(v'(-\lambda)-v^2(-\lambda))$. We compute

\begin{align*}
&\lambda^2(v'(-\lambda)-v^2(-\lambda))=\EE{\left(\frac{\lambda}{\underline Y+\lambda}\right)^2}-\EE{\frac{\lambda}{\underline Y+\lambda}}^2 =\\
&\EE{\left(1- \frac{\underline Y}{\underline Y+\lambda}\right)^2}-\left(1 - \EE{\frac{\underline Y}{\underline Y+\lambda}}\right)^2 = \EE{\left(\frac{\underline Y}{\underline Y+\lambda}\right)^2}-\EE{\frac{\underline Y}{\underline Y+\lambda}}^2.
\end{align*}

Therefore, 
\begin{align*}
\lim_{\lambda\to\infty} \lambda^4(v'(-\lambda)-v^2(-\lambda)) &= \lim_{\lambda\to\infty} \left\{ \EE{\underline Y^2 \left(\frac{\lambda}{\underline Y+\lambda}\right)^2}-\EE{\underline Y\frac{\lambda}{\underline Y+\lambda}}^2 \right\} \\
&= \EE{\underline Y^2}-\EE{\underline Y}^2.
\end{align*}

Using the relationship $\underline F = \gamma F + (1-\gamma) \delta_0$, we can write $\EE{\underline Y} = \gamma \EE{Y}$ and $\EE{\underline Y^2} = \gamma \EE{Y^2}$. Putting everything together, we find

$$\Theta_{\text{IR}} = \lim_{\lambda\to\infty} \frac{\alpha^2 \, \tau \lambda }{\sqrt{\alpha^2 \, \eta \lambda^2 + \theta \lambda^2}}  =  \frac{\alpha^2}{\sqrt{\alpha^2 \, \EE{Y} + \gamma(\EE{Y^2}-\gamma\EE{Y}^2)}}. $$

Finally, it is known that $\EE{Y} =\EE{T}$, $\EE{Y^2} =\EE{T^2}+ \gamma\EE{T}^2$ (see, e.g., Lemma 2.16 in \cite{yao2015large}). This leads to the claimed formula.

\subsection{Proof of Corollary \ref{cor:lda_ir_minimax}}
\label{pf:cor:lda_ir_minimax}

From Theorem \ref{theo:lda_ir}, minimizing $\Theta_{\text{LDA}}$ is equivalent to minimizing $\mathbb{E}_H[T^{-1}]$ for $H \in \mathcal{H}(k_1,k_2)$. By Jensen's inequality, $\mathbb{E}_H[T^{-1}] \ge 1/\mathbb{E}_H[T]=1$; with equality if $H=\delta_1$. This shows the first claim.

Again by Theorem \ref{theo:lda_ir}, minimizing $\Theta_{\text{IR}}$ over $H \in \mathcal{H}(k_1,k_2)$ amounts to maximizing $\mathbb{E}_H[T^{2}]$ over that class. For this, note that $k_1 \le T \le k_2$ for a random variable $T$ distributed according to $H \in \mathcal{H}(k_1,k_2)$. Therefore $(T-k_1)(T-k_2)\le 0$, and taking expectations we get the upper bound

$$  \mathbb{E}_H[T^{2}] \le (k_1+k_2)\mathbb{E}_H[T] -k_1k_2 = k_1+k_2-k_1k_2.$$

This upper bound is achieved for any $H = w_1\delta_{k_1}+ w_2\delta_{k_2}$. It is now easy to check that there exists a unique set of weights $w_i$ such that a distribution of the above form has unit mean, so that it belongs to $\mathcal{H}(k_1,k_2)$; and those are the weights given in the corollary. 

\end{document}